\documentclass[11pt, a4paper, reqno, openany]{book}
\usepackage{amscd}
\usepackage{dsfont, tocloft}

\usepackage{color}
\usepackage{booktabs}
\usepackage{amsmath, upgreek}
\usepackage{amstext}
\usepackage{amsthm}
\usepackage{amssymb}
\usepackage{amsopn}
\usepackage{amssymb}
\usepackage{amsxtra}
\usepackage{amsfonts}
\usepackage{fancyhdr}
\usepackage{paralist, epsfig}
\usepackage[mathcal]{euscript}
\usepackage[english]{babel}
\usepackage[latin1]{inputenc}

\usepackage{blindtext}

\usepackage{imakeidx, hyperref}
\usepackage{setspace}

\theoremstyle{definition}

\newtheorem{thm}{Theorem}[chapter]

\newtheorem{cor}[thm]{Corollary}
\newtheorem{lem}[thm]{Lemma}
\newtheorem{prop}[thm]{Proposition}
\newtheorem{dfn}[thm]{Definition}
\newtheorem{probl}{Problem}[chapter]

\fancypagestyle{allclear}{
\fancyhead{}
\fancyfoot{}
	
}

\fancypagestyle{bibliography}{
\fancyhead{}
\fancyhead[LO,LE]{\leftmark}
\fancyhead[RO,RE]{\thepage}
}

\fancypagestyle{index}{
\fancyhead{}
\fancyhead[LO,LE]{INDEX}
\fancyhead[RO,LE]{\thepage}
}

\fancypagestyle{plain}{}

\indexsetup{level=\chapter*,toclevel=part,firstpagestyle=index}
\makeindex[columns=2,title=Index,intoc,options={-s perso.ist},]

\AtBeginDocument{\addtocontents{toc}{\protect\thispagestyle{fancy}}}

\usepackage{etoolbox}
\patchcmd{\part}{\thispagestyle{plain}}{\thispagestyle{empty}}
  {}{\errmessage{Cannot patch \string\part}}

\usepackage{tikz,tkz-euclide}
\usetikzlibrary{intersections,patterns}
\usetikzlibrary{hobby}

\usetikzlibrary{fadings}
\tikzfading[name=arrowfading right, right color=transparent!0, left color=transparent!99]
\tikzfading[name=arrowfading left, right color=transparent!99, left color=transparent!0]

\makeatletter
\newsavebox\myboxA
\newsavebox\myboxB
\newlength\mylenA

\newcommand*\xoverline[2][0.75]{%
    \sbox{\myboxA}{$\m@th#2$}%
    \setbox\myboxB\null% Phantom box
    \ht\myboxB=\ht\myboxA%
    \dp\myboxB=\dp\myboxA%
    \wd\myboxB=#1\wd\myboxA% Scale phantom
    \sbox\myboxB{$\m@th\overline{\copy\myboxB}$}%  Overlined phantom
    \setlength\mylenA{\the\wd\myboxA}%   calc width diff
    \addtolength\mylenA{-\the\wd\myboxB}%
    \ifdim\wd\myboxB<\wd\myboxA%
       \rlap{\hskip 0.35\mylenA\usebox\myboxB}{\usebox\myboxA}%
    \else
        \hskip -0.5\mylenA\rlap{\usebox\myboxA}{\hskip 0.5\mylenA\usebox\myboxB}%
    \fi}
\makeatother

\newcommand{\smallfrac}[2]{\scalebox{1.35}{\ensuremath{\frac{#1}{#2}}}}
\newcommand{\medfrac}[2]{\scalebox{1.2}{\ensuremath{\frac{#1}{#2}}}}
\newcommand{\textfrac}[2]{{\textstyle\ensuremath{\frac{#1}{#2}}}}

\newcounter{fig}[chapter]
\newenvironment{fig}[1][]{\refstepcounter{fig}\par\medskip
   \textbf{\small Figure~\thefig. #1}\small}{}

\newcounter{tab}[chapter]
\newenvironment{tab}[1][]{\refstepcounter{tab}\par\medskip
   \textbf{\small Table~\thetab. #1}\small}{}

\usepackage{calc}
\newlength{\depthofsumsign}
\renewcommand{\epsilon}{\varepsilon}
\newcommand{\e}{\operatorname{e}}
\newcommand{\B}{\operatorname{B}}
\newcommand{\Bbar}{\xoverline[0.75]{\operatorname{B}}}
\newcommand{\pr}{\operatorname{pr}}
\newcommand{\dd}{\operatorname{d}\hspace{-1pt}}
\renewcommand{\P}{\operatorname{P}}
\renewcommand{\L}{\operatorname{L}}
\newcommand{\E}{\operatorname{E}}
\newcommand{\V}{\operatorname{V}}
\newcommand{\Cov}{\operatorname{Cov}}
\DeclareMathSymbol{\shortminus}{\mathbin}{AMSa}{"39}
\newcommand{\Bigsum}[2]{\ensuremath{\mathop{\textstyle\sum}_{#1}^{#2}}}

\newcommand{\Conv}{\mathop{\scalebox{1.1}{\raisebox{-0.08ex}{$\ast$}}}}%

\usepackage{pgfplots}
\pgfplotsset{compat=1.10}

\usepackage{pifont}
\newcommand{\filledsquare}{\begin{picture}(0,0)(0,0)\put(-4,1.4){$\scriptscriptstyle\text{\ding{110}}$}\end{picture}\hspace{2pt}}

\renewcommand{\thefig}{\arabic{chapter}.\arabic{fig}}
\renewcommand{\thetab}{\arabic{chapter}.\arabic{tab}}

\begin{document}\thispagestyle{empty}
\allowdisplaybreaks
$ $
\vspace{-10pt}

\begin{center}
{\Huge \bf Lecture Notes on\vspace{8pt}\\High-Dimensional Data}

\vspace{20pt}

\large September 12, 2024

\vspace{20pt}

\Large \phantom{1}Sven-Ake Wegner\hspace{0.5pt}\MakeLowercase{$^{\text{1}}$}
\end{center}

\hspace{1000pt}\footnote{Department of Mathematics, University of Hamburg, Bundesstra\ss{}e 55, 20146 Hamburg, Ger-\linebreak{}\phantom{x}\hspace{12pt}many, e-mail: sven.wegner@uni-hamburg.de}

\newpage\thispagestyle{empty}
$ $

\vspace{400pt}

{\Huge\bf Preface}

\vspace{20pt}

The text below arose from a course on `Mathematical Data Science' that I taught twice for final year BSc Mathematics students in the UK between 2019 and 2020. The notes presently cover the first part (roughly a third) of the course focussing on the characteristics and peculiarities of high-dimensional data. The material is based on Blum, Hopcroft, Kannan \cite{BHK20}, Dasgupta, Schulman \cite{DS}, K\"oppen \cite{K00}, Mahoney \cite{M13}, Pinelis \cite{P20} and Vershynin \cite{V20}. An improved version of the notes appeared as part of the textbook \cite{Weg24a}; we refer the reader in particular to \cite[Chapters 8\hspace{1.5pt}--12]{Weg24a}. I would like to thank my former students who attended the course and helped me with their feedback to write these lecture notes.

\medskip
Hamburg, September 12, 2024

\medskip

Sven-Ake Wegner

%\part{High-dimensional spaces}

%%%%%%%%%%%%%%%%%%%%%%%%%%%%%%%%%%%%%%%%%%%%%%%%
%%%%%%%%%%%%%%%%%%%%%%%%%%%%%%%%%%%%%%%%%%%%%%%%
%%                                            %%
%% Chapter 7: Introduction                    %%
%%                                            %%
%%%%%%%%%%%%%%%%%%%%%%%%%%%%%%%%%%%%%%%%%%%%%%%%
%%%%%%%%%%%%%%%%%%%%%%%%%%%%%%%%%%%%%%%%%%%%%%%%

\chapter{The curse of high dimensions}\label{Ch-INTRO}

Mathematically, data is a subset $A\subseteq\mathbb{R}^d$ where the dimension $d$ is typically a very large number. Concrete examples are as follows.

\smallskip

\begin{compactitem}

\item[1.] A greyscale picture with $320\times240$ pixel can be described by $75\,500$ numbers between $0$\,(=white) and $1$\,(=black) and this gives a vector in $\mathbb{R}^{75\,500}$.

\vspace{3pt}

\item[2.] The current Oxford dictionary counts $273\,000$ English words. An arbitrary English text can be seen as a vector in $\mathbb{R}^{273\,000}$ where each entry is the number of occurrences of each word in the text.

\vspace{3pt}

\item[3.] Netflix has approximately $3\,500$ movies and $200$ million registered users. Each user can give a rating from one to five stars for each movie. We can represent each user by her/his rating vector in $\mathbb{R}^{3\,500}$ containing in each entry the user's rating or a zero if the movie has not been rated. Alternatively, we can represent each movie  by a vector in $\mathbb{R}^{200\,000\,000}$ containing in each entry the ratings of the corresponding user.

\vspace{3pt}

\item[4.] Amazon offers $353$ million items for sale. We can represent any of its customers as an element of $\mathbb{R}^{353\,000\,000}$ by counting the amount of purchases per product.

\vspace{3pt}

\item[5.] The human genome has approximately $10\,000$ genes, so every of the 9 billion humans can be considered as a point in $\mathbb{R}^{10\,000}$.

\vspace{3pt}

\item[6.] When doing medical diagnostic tests, we can represent a subject by the vector containing her/his results. These can include integers like antibody counts, real numbers like temperature, pairs of real numbers like blood pressure, or binary values like if a subject has tested positive or negative for a certain infection.

\vspace{3pt}

\item[7.] Facebook has $2.7$ billion users. If we name the users $1,2,3,\dots$, we can represent user $j$ in $\mathbb{R}^{2\,700\,000\,000}$ by a vector containing a one in the $i$-th entry if $i$ and $j$ are friends and a zero if they are not.

\end{compactitem}

\smallskip

Given such a high-dimensional data set $A$, classical tasks to analyze the data, or make predictions based on it, involve to compute distances between data points. This can be for example
\begin{compactitem}
\item[\filledsquare] the classical euclidean distance (or any other $p$-norm),
\item[\filledsquare] a weighted metric that, e.g., in 6.\ weights every medical information differently,
\item[\filledsquare] a distance that is not even a metric like for instance the cosine similarity.
\end{compactitem}
Related tasks are then to determine the k-nearest neighbors within the data set $A$  of a given point $x\in\mathbb{R}^d$, to cover the set via $A\subseteq\cup_{b\in B}\B_r(b)$ or $A\subseteq\Gamma(B)$ with $B\subset A$ as small as possible, to find a cluster structure, to determine and fit a distribution function, or to detect a more complicated hidden pattern in the set.

\smallskip

From the perspective of classical undergraduate courses like Real Analysis, Linear Algebra and Measure Theory, the space $\mathbb{R}^d$, endowed with euclidean norm, appears well understood: Distances, balls, convex hulls, minimization problems, matrices and volume are well studied in finite dimensions. However, if $d$ is very large, we are faced with the following two obstructions. Firstly, the computation of the aforementioned quantities can be unfeasable in view of the sheer amount of data. Secondly, natural (and very intuitive!) probability distributions such as the uniform distribution $\mathcal{U}(B)$ on the hypercube $B=[-1,1]^d$ or unit ball $B=\B_1(0)\subseteq\mathbb{R}^d$ or the Gaussian distribution $\mathcal{N}(0,1)$ in $d$ dimensions, display several very counterintuitive effects if $d$ is large. This  strongly affects methods that would help to overcome the first obstruction (e.g., drawing a random sample instead of using the whole data set) and has also to be taken into account if we are given empirical data which is subject to errors or noise that we wish to model by a certain probability distribution.

\smallskip

We begin by recalling the following definition.

\begin{dfn}\label{DFN-DISTR} Let $(\Omega,\Sigma,\P)$ be a probability space, let $B\subseteq\mathbb{R}$ be a Borel set and let $X\colon\Omega\rightarrow\mathbb{R}^d$ be a random vector.
\begin{compactitem}

\item[(i)] We say that $X$ is \emph{uniformly distributed on $B\subseteq\mathbb{R}^d$} if
$$
\P\bigl[X\in A\bigr] = {\medfrac{\lambda^d(A\hspace{1pt}\cap\hspace{1pt}B)}{\lambda^d(B)}}
$$
holds for every Borel set $A\subseteq\mathbb{R}^d$. We write $X\sim\mathcal{U}(B)$ in this case and note that below we will focus on the cases of the \emph{hypercube} $B=[-1,1]^d$ and the \emph{(closed) unit ball} $B=\Bbar_1(0)=\{x\in\mathbb{R}^d\:;\:\|x\|_2\leqslant1\}$. In the latter case we can as well drop the closure since $\{x\in\mathbb{R}^d\:;\:\|x\|_2=1\}$ has measure zero.

\vspace{3pt}

\item[(ii)] We say that $X$ is \emph{(spherically) Gaussian distributed with mean zero and variance $\sigma>0$}, if
$$
\P[X\in A] = {\medfrac{1}{(2\pi\sigma^2)^{d/2}}} \int_A\exp\bigl({-\medfrac{\|x\|^2}{2\sigma^2}}\bigr)\dd\lambda^d(x)
$$
holds for every Borel set $A\subseteq\mathbb{R}^d$. We write $X\sim\mathcal{N}(0,\sigma^2,\mathbb{R}^d)$ in this case and note that a straightforward computation (see Proposition \ref{PROP-9-1}) shows that $X=(X_1,\dots,X_d)\sim\mathcal{N}(0,\sigma^2,\mathbb{R}^d)$ holds if and only if $X_i\sim\mathcal{N}(0,\sigma^2,\mathbb{R}^1)$ holds for all $i=1,\dots,d$.
\end{compactitem}
\end{dfn}

We start now with the case of the $d$-dimensional hypercube $H_d:=\{(x_1,\dots,x_d)\in\mathbb{R}^d\:;\:x_i\in[-1,1] \text{ for all } 1\leqslant i\leqslant d\}$. Not surprisingly, $H_d$ is a compact and absolutely convex set. Consider now the `corner' $e=(1,1,\dots,1)\in H_d$. Then
$$
\|e\|_2=\sqrt{1^2+1^2+\cdots+1^2}=\sqrt{d}\xrightarrow{d\rightarrow\infty}\infty
$$
and this will hold for `more and more corners' when $d$ increases. In particular, the mutual distance between `many' corners increases with $d$. On the other hand, if we pick two points $e_i=(0,\dots,0,1,0,\dots,0)$ and $f_j=(0,\dots,0,1,0,\dots,0)$ in the middle of two different `faces' then their distances
$$
\|e_i-e_j\|_2=\sqrt{1^2+1^2}=\sqrt{2}
$$
remains constant. This is in particular true for opposite faces, where $\|e_i-(-e_i)\|_2=2$. The aforementioned computations paint a picture in which we could imagine the hypercube as an object where the corners `spike out' whereas the centers of the faces stay at constant distance to each other. If distances behave non homogeneously, then we should be prepared that volume is distributed inhomogeneous, too. To see this, let $0<\epsilon<1$ and consider the subset
$$
S_{\epsilon,d}=H_d\backslash{}(1-\epsilon)H_d
$$
of $H_d$ which can be thought of an $\epsilon$-thin outer shell of the hypercube.

\begin{center}
\begin{picture}(300,65)(0,0)

\put(40,20){\begin{tikzpicture}[scale=0.5]

\coordinate (1) at (0,0);
 \coordinate (2) at (3,0);

 \coordinate (a) at (0.2,0);
 \coordinate (b) at (2.8,0);

    \draw[thick] (1)--(2);
    
        \draw[black!50!white, thick] (a)--(b);
    
    \node at (3,1) {\small$S_{\epsilon,1}$};
      \draw (0.1,0.1)--(2.3,0.7); 
      \draw (2.9,0.1)--(2.5,0.7); 
    
\end{tikzpicture}}

\put(124,8){\begin{tikzpicture}[scale=0.5]

\coordinate (1) at (0,0);
 \coordinate (2) at (3,0);
 \coordinate (3) at (3,3);
 \coordinate (4) at (0,3);
 
 \coordinate (a) at (0.2,0.2);
 \coordinate (b) at (2.8,0.2);
 \coordinate (c) at (2.8,2.8);
 \coordinate (d) at (0.2,2.8);

    \draw[thick] (1)--(2)--(3)--(4)--cycle;
    
        \draw[black!50!white, thick] (a)--(b)--(c)--(d)--cycle;

    \node at (2.7,3.6) {\small$S_{\epsilon,2}$};
      \draw (1.2,2.9)--(2,3.4);  
          
\end{tikzpicture}}

\put(208,0){\begin{tikzpicture}[scale=0.5]

 \coordinate (1) at (0,0);
 \coordinate (2) at (2.4,-0.6);
 \coordinate (3) at (2.4,1.8);
 \coordinate (4) at (0,2.4);
 \coordinate (5) at (4.4,0.8);
 \coordinate (6) at (4.4,3.2);
 \coordinate (7) at (2,3.8);
 \coordinate (8) at (2,1.4);
 
\tkzInterLL(7,8)(3,4) \tkzGetPoint{9}
\tkzInterLL(5,8)(3,2) \tkzGetPoint{10}

 \draw[thick,draw=black] (1)--(8);

    \draw[thick] (6)--(7)--(4);
 
 \draw [thick, shorten >= 1pt] (8)--(9);
 \draw [thick, shorten <= 1pt] (9)--(7);
  \draw [thick, shorten >= 1pt] (8)--(10);
 \draw [thick, shorten <= 1pt] (10)--(5);

 \coordinate (a) at (0.17,0.07+0.13);
 \coordinate (b) at (2.35,-0.49+0.13);
 \coordinate (c) at (2.35,1.7);
 \coordinate (d) at (0.17,2.26);
 \coordinate (e) at (4.2,0.78+0.13);
 \coordinate (f) at (4.2,3);
 \coordinate (g) at (1.96,3.56);
 \coordinate (h) at (1.96,1.33+0.13);

\tkzInterLL(g,h)(c,d) \tkzGetPoint{i}
\tkzInterLL(h,e)(b,c) \tkzGetPoint{j}

\draw [thick, black!50!white] (a)--(b)--(c)--(d)--cycle;
\draw [thick, black!50!white] (b)--(e)--(f)--(c);
 \draw [thick, black!50!white] (f)--(g)--(d);
  \draw [thick, black!50!white] (h)--(a);

    \draw[thick, black!50!white, shorten >=1pt] (g)--(i);
     \draw[thick, black!50!white, shorten <=1pt] (i)--(h);
     
      \draw[thick, black!50!white, shorten >=1pt] (h)--(j);
     \draw[thick, black!50!white, shorten <=1pt] (j)--(e);   
     
     \draw[thick] (1)--(2)--(3)--(4)--cycle;
       \draw[thick] (2)--(5)--(6)--(3);
       
      \node at (4.2,4.2) {\small$S_{\epsilon,3}$};
      \draw (2.8,3.5)--(3.5,4);      
\end{tikzpicture}}

\end{picture}
\begin{fig}$S_{\epsilon,d}$ for $d=1,2,3$.\end{fig}
\end{center}

Since $(1-\epsilon)H_d=[-1+\epsilon,1-\epsilon]^d$ is a cuboid, we can easily compute
$$
\medfrac{\lambda^d(S_{\epsilon,d})}{\lambda^d(H_{d})}=\medfrac{\lambda^d(H^d)-\lambda^d((1-\epsilon)H_d)}{\lambda^d(H_{d})}=\medfrac{2^d-((2(1-\epsilon))^d}{2^d}=1-(1-\epsilon)^d\xrightarrow{d\rightarrow\infty}1
$$
for $0<\epsilon<1$. This means that most of the volume of the hypercube is located close to its surface. Let us mention that the same holds true if we consider the hypercube $[0,1]^d$ with side length one. Then for its volume we get $\lambda^d([0,1]^d)=1$ and the volume in the $\epsilon$-thick shell equals $\lambda^d([0,1]^d\backslash[\epsilon,1-\epsilon]^d)=1-(1-2\epsilon)^d$ which again converges to $1$. So also here most of the volume lies in an $\epsilon$-thin shell below its surface. On the other hand not much of the volume can be located close to the corners: Consider the corner $e=(e_1,\dots,e_d)$ with $e_i\in\{1,-1\}$ of $H_d=[-1,1]^d$. Then we consider
$$
E_{e}:=\{x\in\mathbb{R}^d\:;\: x_i\in[1-\epsilon,1] \text{ if } e_i=1 \text{ and } x_i\in[-1,-1+\epsilon] \text{ if } e_i=-1 \}
$$
that is a cube with side length $\epsilon$ sitting in the $e$-th corner of $H_d$ 

\begin{center}
\hspace{40pt}\begin{tikzpicture}[scale=0.4]

\coordinate (1) at (0,0);
\coordinate (2) at (3,0);
\coordinate (3) at (3,3);
\coordinate (4) at (0,3);
 
\coordinate (a) at (0.8,0.8);
\coordinate (a1) at (0.8,0);
\coordinate (a2) at (0,0.8);
   
\coordinate (b) at (2.2,0.8);
\coordinate (b1) at (2.2,0);
\coordinate (b2) at (3,0.8);

\coordinate (c) at (2.2,2.2);
\coordinate (c1) at (3,2.2);
\coordinate (c2) at (2.2,3);

\coordinate (d) at (0.8,2.2);
\coordinate (d1) at (0,2.2);
\coordinate (d2) at (0.8,3);
 
\draw[thick] (1)--(2)--(3)--(4)--cycle;
    
%\draw[black!50!white, thick] (1)--(a1)--(a)--(a2)--cycle;
%\draw[black!50!white, thick] (2)--(b1)--(b)--(b2)--cycle;
\draw[black!50!white, thick] (3)--(c1)--(c)--(c2)--cycle;
%\draw[black!50!white, thick] (4)--(d1)--(d)--(d2)--cycle;

\node at (4,3.6) {\small$E_{\scriptscriptstyle(1,1)}$};
\draw (2.7,2.9)--(3.2,3.3);

%\node at (4.2,1.4) {\small$E_{\scriptscriptstyle(1,-1)}$};
%\draw (2.7,0.7)--(3.2,1.1);

%\node at (1.6,3.6) {\small$E_{\scriptscriptstyle(-1,1)}$};
%\draw (0.3,2.9)--(0.8,3.3);

%\node at (-1.6,0.9) {\small$E_{\scriptscriptstyle(-1,-1)}$};
%\draw (0.1,0.2)--(-0.4,0.5); 
          
\end{tikzpicture}\nopagebreak[4]
\begin{fig}$E_{(1,1)}\subseteq H_d$ for $d=2$.\end{fig}
\end{center}

and observe that $\lambda^d(E_e)=\epsilon^d$. We put $0<\epsilon<1/2$ and since there are $2^d$ corners $e$ in $H_d$, we get that the volume in all the corners 
$$
2^d\cdot \epsilon^d =(2\epsilon)^d\xrightarrow{d\rightarrow\infty}0
$$
tends to zero. So we may think of the volume of $H_d$ being located close below  the surface and around the middle points of the faces. The following pictures summarize the three different ways to look at the hypercube in high dimensions.  We advise the reader however to be careful as all three pictures show of course a 2-dimensional cube and are only meant to visualize different effects that take place in high dimensions.

\begin{center}
\begin{picture}(350,110)(0,0)

\put(20,45){\begin{tikzpicture}[scale=0.8]
\coordinate (1) at (0,0);
\coordinate (2) at (3,0);
\coordinate (3) at (3,3);
\coordinate (4) at (0,3);
\fill[pattern=north west lines,draw=black] (1)--(2)--(3)--(4)--cycle;
\end{tikzpicture}}

\put(10,20.5){\begin{minipage}{90pt}\footnotesize
The hypercube is compact and convex.
\end{minipage}}

\put(136,42){\begin{tikzpicture}[scale=0.85]
\coordinate (1) at (-1.5,-1.5);
\coordinate (2) at (1.5,-1.5);
\coordinate (3) at (1.5,1.5);
\coordinate (4) at (-1.5,1.5);
\coordinate (a1) at (-0.7,-0.3);
\coordinate (a2) at (-0.3,-0.7);
\coordinate (b1) at (0.3,-0.7);
\coordinate (b2) at (0.7,-0.3);
\coordinate (c1) at (0.7,0.3);
\coordinate (c2) at (0.3,0.7);
\coordinate (d1) at (-0.3,0.7);
\coordinate (d2) at (-0.7,0.3);
 \draw [black, thick] plot [smooth cycle, tension=20] coordinates {(a1) (1) (a2) (b1) (2) (b2) (c1) (3) (c2) (d1) (4) (d2)};
\end{tikzpicture}}

\put(122,10){\begin{minipage}{105pt}\footnotesize Distances between corners grow with the dimension, while distances between fa\-ces remain constant.
\end{minipage}}

\put(263,45){\begin{tikzpicture}[scale=0.8]
\pgfdeclareradialshading{ring}{\pgfpoint{0cm}{0cm}}%
{rgb(0cm)=(1,1,1);
rgb(0.7cm)=(1,1,1);
rgb(0.719cm)=(1,1,1);
rgb(0.72cm)=(0.3,0.3,0.3);
rgb(0.9cm)=(1,1,1)}
\coordinate (1) at (0,0);
\coordinate (2) at (3,0);
\coordinate (3) at (3,3);
\coordinate (4) at (0,3);

\coordinate (a) at (0.4,0.4);
\coordinate (b) at (2.6,0.4);
\coordinate (c) at (2.6,2.6);
\coordinate (d) at (0.4,2.6);

\begin{scope}
\clip(0,0) rectangle (3,3);
\fill[shading=ring] (-2.35,1.5) circle (2.7);
\fill[shading=ring] (5.35,1.5) circle (2.7);
\fill[shading=ring] (1.5,5.35) circle (2.7);
\fill[shading=ring] (1.5,-2.35) circle (2.7);
\end{scope}

\draw[thick] (1)--(2)--(3)--(4)--cycle;
\end{tikzpicture}}

\put(250,15.5){\begin{minipage}{100pt}\footnotesize
The volume concentrates at the surface and in the middle of the faces.
\end{minipage}}
\end{picture}\vspace{5pt}\nopagebreak[4]
\begin{fig}Three ways to think of the high-dimensional\vspace{-3pt}\\hypercube in terms of geometry and topology, metric and measure.\end{fig}
\end{center}

If we intepret the last of our above findings in terms of a random variable $X\sim\mathcal{U}(H_d)$, then it means that $X$ attains with high probability values close to the surface and close to the middle of the faces. Conversely, this means that every method that draws random points from the hypercube will neglect its corners and its interior. So if we for example are given a data set inside $H_d$ and want to compute a predictor function $H_d\rightarrow\mathbb{R}$ then we need a a very large sample to do so. Indeed, the volume of $H_d$ grows exponentially in $d$ so if we, e.g., want to partition $H_d$ into smaller cubes with fixed volume we will need exponentially many such cubes. Furthermore, if we knew that our data is uniformly distributed then all the data will gather together at the regions indicated above. If we now sample for instance two points then the odds that they are each close to the middle of different (and not opposite) faces are high. The right picture above suggests that we may expect that the angle between any two randomly chosen points will be approximately $90^{\circ}$. Experiments suggest moreover, that the distance of two points chosen at random is close to $\sqrt{\scalebox{0.8}{$\frac32$}d}$.

\smallskip  

The above effect is commonly referred to as the \emph{curse of high dimensions}. Indeed, classical data analysis methods relying, e.g., on $k$-nearest neighbor search, or cosine similarity, can apparently not anymore be applied effectively, if all mutual distances and angles between data points coincide! We will see below that this effect is not restricted to the hypercube, but appears also for the unit ball $\Bbar_1(0)$ and on the whole space $\mathbb{R}^d$ which we will study now.

\medskip

For this, let us consider points drawn at random from $\mathbb{R}^d$ with respect to a Gaussian distribution with mean zero and variance one. As we mentioned in Definition \ref{DFN-DISTR}(ii), the latter formalizes as a random vector $X\colon\Omega\rightarrow\mathbb{R}^d$ whose coordinates are independent normal random variables. Therefore it is very easy (see Problem \ref{PROB-1}) to generate such random points in a simulation and compute angles and norms. If we do so, we detect an effect similar to what we just above observed for the hypercube. We get that the norms of all random points in dimension 100 are close to 10.

\begin{center}
\begin{tikzpicture}

	\begin{axis}
[
axis line style={thick, shorten >=-10pt, shorten <=-10pt},
y=0.0275pt,
x=17pt,
axis y line=left,
axis x line=middle,
axis line style={->},
no markers,
tick align=outside,
major tick length=2pt,
xmin=0,
xmax=14,
xtick={0, 2, ..., 14},
ymin=0,
ymax=3000,
ytick={500, 1000, ..., 3000},
every tick label/.append style={font=\tiny},
xlabel=\small $\|x\|$,
every axis x label/.style={
    at={(ticklabel* cs:1.07)},
    anchor=west,
},
]
	% use TeX as calculator:
		\addplot[ mark=none,fill=black, 
                    fill opacity=0.05] coordinates {
(6.85, 1)
(6.95, 0)
(7.05, 0)
(7.15, 0)
(7.25, 2)
(7.35, 0)
(7.45, 5)
(7.55, 5)
(7.65, 9)
(7.75, 11)
(7.85, 17)
(7.95, 41)
(8.05, 57)
(8.15, 81)
(8.25, 132)
(8.35, 179)
(8.45, 249)
(8.55, 362)
(8.65, 517)
(8.75, 657)
(8.85, 825)
(8.95, 1030)
(9.05, 1250)
(9.15, 1467)
(9.25, 1779)
(9.35, 1886)
(9.45, 2180)
(9.55, 2357)
(9.65, 2626)
(9.75, 2624)
(9.85, 2823)
(9.95, 2703)
(10.05, 2854)
(10.15, 2671)
(10.25, 2580)
(10.35, 2336)
(10.45, 2261)
(10.55, 1999)
(10.65, 1753)
(10.75, 1483)
(10.85, 1262)
(10.95, 1089)
(11.05, 902)
(11.15, 755)
(11.25, 573)
(11.35, 455)
(11.45, 305)
(11.55, 227)
(11.65, 167)
(11.75, 146)
(11.85, 105)
(11.95, 72)
(12.05, 51)
(12.15, 29)
(12.25, 15)
(12.35, 12)
(12.45, 8)
(12.55, 2)
(12.65, 7)
(12.75, 2)
(12.85, 1)
(12.95, 0)
(13.05, 2)
	};
	\end{axis}
\end{tikzpicture}\nopagebreak[4]
\begin{fig}\label{FIG-1}Distributions of the norm of 50\,000 random points $x\sim\mathcal{N}(0,1,\mathbb{R}^{100})$.\end{fig}
\end{center}

If we run the experiment for different dimensions $d$ and each time compute the average over the norms then we see that the latter turn out to be close to $\sqrt{d}$. Moreover, we see that the variances seem to be bounded by a constant independent of the dimension.

\begin{center}
\begin{tabular}{crrrrrrrrrr}
\toprule
$d$ & 1 & 10 & 100 & 1\,000 & 10\,000 & 100\,000 & 1\,000\,000 \\
\midrule
$\frac{1}{n}{\displaystyle\Bigsum{i=1}{n}}\|x^{\scriptscriptstyle(i)}\|$ & 0.73 & 3.05 & 10.10 & 31.61 & 100.03 & 316.21 & 1000.03 \\\vspace{4pt} 
$\sqrt{d}$ & 1.00 & 3.16 & 10.00 & 31.62 & 100.00 & 316.22 & 1000.00 \\
Variance & 0.33 & 0.48 & 0.54 & 0.45 & 0.52 & 0.36 & 0.44 \\
\bottomrule
\end{tabular}\nopagebreak[4]
\begin{tab}\label{TAB-1}Average norms of $n=100$ random points $x^{\scriptscriptstyle(i)}\sim\mathcal{N}(0,1,\mathbb{R}^d)$.\end{tab}
\end{center}

We leave it as Problem \ref{PROB-1} to carry out the experiments and replicate Figure \ref{FIG-1} and Table \ref{TAB-1}. From the theoretical point of view, if we use that $X\sim\mathcal{N}(0,1,\mathbb{R}^d)$ has independent coordinates $X_i\sim\mathcal{N}(0,1)$ for $i=1,\dots,d$, we can compute expectations and variances directly. Since we do not know a priori how the root function interacts with the expectation, we start considering the norm squared. We then get
\begin{equation}\label{EXP-NORM-SQUARED}
\begin{array}{rl}\vspace{3pt}
\E(\|X\|^2) &= \E\bigl(X_1^2+\cdots+X_d^2\bigr) = \E(X_1^2)+\cdots+\E(X_d^2)\\\vspace{3pt}
&= \E\bigl((X_1-\E(X_1))^2\bigr)+\cdots+\E\bigl((X_d-\E(X_d))^2\bigr)\\
&= \V(X_1)+\cdots+\V(X_d)= d
\end{array}
\end{equation}
which corroborates that $\E(\|X\|)$ should be close to $\sqrt{d}$ for large $d$. Also, the variance of the norm squared can be computed directly, namely via
\begin{equation}\label{VAR-NORM-SQUARED}
\begin{array}{rl}\vspace{2pt}
\V(\|X\|^2) &=\V(X_1^2)+\cdots+\V(V_d^2) = d\cdot\V(X_1^2)=d\cdot\bigl(\E(X_1^4)-\E(X_1^2)^2\bigr)\\\vspace{3pt}
& =d\cdot\bigl(\medfrac{1}{\sqrt{2\pi}}{\displaystyle\int_{\mathbb{R}}}x^4\exp(-x^2/2)\dd x-\V(X_1)\bigr)\\
&= (3-1)d= 2d.
\end{array}
\end{equation}
Without the square we get the following result which confirms and quantifies the picture that we got in the experiments, namely that Gaussian random vectors will have norm close to $\sqrt{d}$ and that their distribution does not spread out if $d$ increases.

\begin{thm}\label{EXP-VAR-NORM} Let $X\sim\mathcal{N}(0,1,\mathbb{R}^d)$. Then
\begin{compactitem}\vspace{2pt}
\item[(i)] $\forall\:d\geqslant1\colon |\E(\|X\|-\sqrt{d}\hspace{1.5pt})|\leqslant1/\sqrt{d}$,\vspace{2pt}
\item[(ii)] $\forall\:d\geqslant1\colon \V(\|X\|)\leqslant2$.\vspace{2pt}
\end{compactitem}
In particular, the expectation $\E(\|X\|-\sqrt{d}\hspace{1.5pt})$ converges to zero for $d\rightarrow\infty$.
\end{thm}
\begin{proof}(i) We start with the following equality
$$
\|X\|-\sqrt{d} = \medfrac{\|X\|^2-d}{2\sqrt{d}}-\medfrac{(\|X\|^2-d)^2}{2\sqrt{d}\hspace{1pt}(\|X\|+\sqrt{d}\hspace{1pt})^2}=:S_d-R_d
$$
 which follows from $\|X\|^2-d=(\|X\|-\sqrt{d}\hspace{1.5pt})(\|X\|+\sqrt{d}\hspace{1.5pt})$. Next we use \eqref{VAR-NORM-SQUARED}
and $\|X\|\geqslant0$ to estimate
$$
0\leqslant \E(R_d) \leqslant \medfrac{\E((\|X\|^2-d)^2)}{2d^{3/2}}=\medfrac{\V(\|X\|^2)}{2d^{3/2}}=\medfrac{2d}{2d^{3/2}}=\medfrac{1}{\sqrt{d}}
$$
which shows that $\E(R_d)\rightarrow0$ for $d\rightarrow\infty$. As $\E(\|X\|^2)=d$ we have $\E(S_n)=0$ and thus we get
$$
|\E(\|X\|-\sqrt{d}\hspace{1.5pt})|=|\E(S_d-R_d)|=|-\E(R_d)|\leqslant\medfrac{1}{\sqrt{d}}
$$
as claimed.

\smallskip

(ii) For the variance we obtain
\begin{align*}
\V(\|X\|) &=  \V(\|X\|-\sqrt{d}) = \E\bigl((\|X\|-\sqrt{d})^2\bigr) - \bigl(\E(\|X\|-\sqrt{d})\bigr)^2\\
&\leqslant \E\bigl((\|X\|-\sqrt{d})^2\bigr)= \E\bigl(\|X\|^2 - 2\|X\|\sqrt{d} + d\bigr)\\
 &   = \E(\|X\|^2) - 2\sqrt{d}\hspace{1.5pt}\E(\|X\|) +d= 2d -2\sqrt{d}\hspace{1.5pt}\E\bigl(\|X\| -\sqrt{d} +\sqrt{d}\hspace{1.5pt}\bigr)\\
&  = 2\sqrt{d}\hspace{1.5pt}\E(R_d) \leqslant 2
\end{align*}
which establishes (ii). 
\end{proof}

Let us now consider the mutual distance between two independent random vectors $X,Y\sim\mathcal{N}(0,1,\mathbb{R}^d)$. For the squared distance we can compute directly
\begin{align*}
\E(\|X-Y\|^2) & = \Bigsum{i=1}d\E\bigl((X_i-Y_i)^2\bigr)\\
&= \Bigsum{i=1}d \E(X_i^2)+2\E(X_i)\E(Y_i)+\E(Y_i^2)\\
& = \Bigsum{i=1}d (1+2\cdot0\cdot0+1) = 2d
\end{align*}
which suggests that $\|X-Y\|$ will be close to $\sqrt{2d}$. This is again underpinned by experiments, see Table \ref{TAB-3x}, and we leave it as Problem \ref{PROB-3} to repeat the experiments.
\begin{center}
\begin{tabular}{crrrrrrrrr}
\toprule
$d$ & 1 & 10 & 100 & 1\,000 & 10\,000 & 100\,000  \\
\midrule
$\frac{1}{n(n-1)}{\displaystyle\Bigsum{i\not=j}{}}\|x^{\scriptscriptstyle(i)}-x^{\scriptscriptstyle(j)}\|$ & 1.06 & 4.41 & 14.05 & 44.65 & 141.50 & 447.35  \\\vspace{4pt} 
$\sqrt{2d}$ & 1.41 & 4.47 & 14.14 & 44.72 & 141.42 & 447.21  \\
Variance & 0.74 & 1.10 & 0.92 & 0.89 & 0.96 & 1.08  \\
\bottomrule
\end{tabular}\nopagebreak[4]
\begin{tab}\label{TAB-3x}Average pairwise distances of $n=100$ random points $x^{\scriptscriptstyle(i)}\sim\mathcal{N}(0,1,\mathbb{R}^d)$.\end{tab}
\end{center}
We leave it as Problem \ref{PROB-2} to adapt our proof of Theorem \ref{EXP-VAR-NORM} to establish the following result on the distances without square.

\begin{thm}\label{EXP-VAR-DIST} Let $X$, $Y\sim\mathcal{N}(0,1,\mathbb{R}^d)$. Then
\begin{compactitem}\vspace{2pt}
\item[(i)] $\forall\:d\geqslant1\colon\E(\|X-Y\|-\sqrt{2d})\leqslant1/\sqrt{2d}$,\vspace{2pt}
\item[(ii)] $\V(\|X-Y\|)\leqslant3$.
\end{compactitem}
\end{thm}
\begin{proof} This can be proved exactly as Theorem \ref{EXP-VAR-NORM}.
\end{proof}

To see what is happening to the angles of random points we can now make use of the fact that we know already that $\E(\|X\|^2)=\E(\|Y\|^2)=d$ and $\E(\|X-Y\|^2)=2d$ holds. From this we get the next result.

\begin{thm}\label{EXP-VAR-ANGLE} Let $X,Y\sim\mathcal{N}(0,1,\mathbb{R}^d)$ and $\xi\in\mathbb{R}^d$ be a constant. Then
\begin{compactitem}\vspace{2pt}
\item[(i)] $\E(\langle{}X,Y\rangle{})=0$ and $\V(\langle{}X,Y\rangle{})=d$,\vspace{2pt}
\item[(ii)] $\E(\langle{}X,\xi\rangle{})=0$ and $\V(\langle{}X,\xi\rangle{})=\|\xi\|^2$.
\end{compactitem}
\end{thm}
\begin{proof} (i) Employing the cosine rule leads to
$$
\|X-Y\|^2=\|X\|^2+\|Y\|^2-2\|X\|\|Y\|\cos(\theta)
$$
where $\theta$ is the angle between $X$ and $Y$ as in the following picture.
\begin{center}
\usetikzlibrary{angles}
\begin{tikzpicture}[scale=0.7,
my angle/.style = {draw, 
                   angle radius=7mm, 
                   angle eccentricity=1.1, 
                   right, inner sep=1pt,
                   font=\footnotesize}]
                   
\draw   (0,0) coordinate (a) --
        (5,0) coordinate (c) --
        (0.5,3) coordinate (b) -- cycle;
\pic[my angle] {angle = c--a--b};

\node[circle, fill, black, inner sep=1pt, outer sep=0pt, xshift=0pt, radius=1pt] at (0,0) {};
\node[circle, fill, black, inner sep=1pt, outer sep=0pt, xshift=0pt, radius=1pt] at (5,0) {};
\node[circle, fill, black, inner sep=1pt, outer sep=0pt, xshift=0pt, radius=1pt] at (0.5,3) {};

\node[] at (0.45,0.37) {\small$\theta$};
\node[] at (-0.2,-0.3) {$0$};
\node[] at (5.3,-0.3) {$X$};
\node[] at (0.35,3.3) {$Y$};

\node[] at (-0.55,1.8) {\small$\|Y\|$};
\node[] at (2.9,-0.45) {\small$\|X\|$};
\node[] at (3.95,2.1) {\small$\|X-Y\|$};

\draw   (-0.1,1.8) -- (0.2,1.65);
\draw   (2.38,-0.35) -- (2.0,-0.09);
\draw   (2.95,2.1) -- (2.5,1.75);
\end{tikzpicture}\nopagebreak[4]
\begin{fig}\label{FIG-3a}The cosine rule generalizes the Pythagorean theorem.\end{fig}
\end{center}
Using $\cos\theta=\langle{}\frac{X}{\|X\|},\frac{Y}{\|Y\|}\rangle{}$ yields $\langle{}X,Y\rangle{}=\medfrac{1}{2}\bigl(\|X\|^2+\|Y\|^2-\|X-Y\|^2\bigr)$ whose expectation we can compute as follows
$$
\E(\langle{}X,Y\rangle{}) =\medfrac{1}{2}\bigl(\E(\|X\|^2)+\E(\|Y\|^2)-\E(\|X-Y\|^2)\bigr) =\medfrac{1}{2}(d+d-2d)=0.
$$
The variance computes as follows
\begin{align*}
\V(\langle{}X,Y\rangle{}) & = \V(X_1Y_1+\cdots+X_dY_d) = \V(X_1Y_1)+\cdots+\V(X_dY_d)\\
& = d\cdot\V(X_1Y_1) = d\cdot\bigl(\E(X_1^2)\E(Y_1^2) - \E(X_1)^2\E(Y_1)^2\bigr)\\
& = d\cdot\bigl(\E(X_1^2-\E(X_1))\E(Y_1^2-\E(Y_1))\bigr) = d\cdot\V(X_1)\V(Y_1) = d.
\end{align*}
(ii) For $\xi=(\xi_1,\dots,\xi_d)$ we get
$$
\E(\langle{}X,\xi\rangle{})=\Bigsum{i=1}{d}\E(X_i\xi_i)= \Bigsum{i=1}{d}\xi_i\E(X_i)=0
$$
and
$$
\V(\langle{}X,\xi\rangle{}) = \Bigsum{i=1}{d}\V(\xi_i X_i) = \Bigsum{i=1}{d}\xi_i^2\V(X_i) = \|\xi\|^2
$$
in a straightforward manner.
\end{proof}

Also the statements of Theorem \ref{EXP-VAR-ANGLE} can be underpinned by experiments, see Table \ref{TAB-3} and Problem \ref{PROB-4}.

\begin{center}
\begin{tabular}{crrrrrrrrrr}
\toprule
$d$ & 1 & 10 & 100 & 1\,000 & 10\,000 & 100\,000 \\
\midrule\vspace{2pt}
$\frac{1}{n(n-1)}{\displaystyle\Bigsum{i\not=j}{\phantom{n}}}\langle{}x^{\scriptscriptstyle(i)},x^{\scriptscriptstyle(j)}\rangle$ & -0.01 & 0.05 & -0.20 & 0.69 & 0.91 & 4.07  \\
Variance & 0.76 & 9.83 & 107.23 & 1\,022.83 & 9\,798.55 & 101\,435.71  \\
\bottomrule
\end{tabular}
\begin{tab}\label{TAB-3}Average  scalar product of $n=100$ random points $x^{\scriptscriptstyle(i)}\sim\mathcal{N}(0,1,\mathbb{R}^d)$.\end{tab}
\end{center}

The above means that the distribution of the scalar products spreads out for large $d$. If we however consider the normalized scalar products $\langle{}\frac{x}{\|x\|},\frac{y}{\|y\|}\rangle{}$, then the experiments suggest that average and variance tend to zero as $d\rightarrow\infty$, see Problem \ref{PROB-4}.

\medskip

Working with the expectation describes what happens `in average' and the variance allows to judge roughly the error we make if we reduce to considering the expectation. In the sequel we will, instead of working with expectation and variance, present explicit estimates for $\P\bigl[\hspace{1pt}\bigl|\|X\|-\sqrt{d}\hspace{1.5pt}\bigr|\geqslant\epsilon\hspace{.5pt}\bigr]$ and $\P\bigl[\hspace{1pt}\bigl|\langle{}X,Y\rangle{}\bigr|\geqslant\epsilon\hspace{.5pt}\bigr]$ whenever $X,Y\sim\mathcal{N}(0,1,\mathbb{R}^d)$. Moreover, we will get similar results for $X,Y\sim\mathcal{U}(\B_1(0))$ and establish the following table

\begin{center}
\begin{tabular}{clll}
        \toprule
        \phantom{i}Distribution \hspace{10pt}\phantom{x}&\hspace{20pt} $\|X\|$\hspace{10pt} &\hspace{10pt} $\hspace{10pt}\|X-Y\|$ \hspace{10pt}&\hspace{10pt} $\hspace{10pt}\langle X,Y\rangle$ \hspace{10pt}\\
        \midrule
        $X,Y\sim\mathcal{N}(0,1,\mathbb{R}^d)$\hspace{10pt}\phantom{x} &\phantom{$\sum^{T^T}$}$\hspace{-5pt}\approx\sqrt{d}$ & \phantom{$\sum^{T^{T^T}}$}$\approx\sqrt{2d}$& \phantom{$\sum^{T^{T^T}}$}$\approx0$\\
        $X,Y\sim\mathcal{U}(\operatorname{B}_1(0))$\hspace{10pt}\phantom{x} &\phantom{$\sum^{T^{T^T}}$}\hspace{-6pt}$\approx1$ & \phantom{$\sum^{T^{T^T}}$}$\approx\sqrt{2}$& \phantom{$\sum^{T^{T^T}}$}$\approx0$\\
        \bottomrule
        \end{tabular}\begin{tab}Values of norm and scalar product that will be\vspace{-3pt}\\attained in high dimensions `with high probability.'\end{tab}
        \end{center}

in which we mean, e.g., by $\|X\|\approx\sqrt{d}$, that the probability that $\|X\|-\sqrt{d}$ is larger than some threshold can be estimated by an expression that is small if the dimension is large enough. In the ideal case the threshold is a constant $\epsilon>0$ that we can choose arbitrarily and the bound for the probability depends only on $d$. We will see however that sometimes both expressions depend on $\epsilon$ and $d$. In these cases we have to be careful when we want to consider $\epsilon\rightarrow0$ or $d\rightarrow\infty$. Avoiding limits in this context\,---\,sometimes referred to as `non-asymptotic analysis'\,---\,will lead to theorems that guarantee that for \emph{almost all points} of a data set $A$, e.g., more than $(1-1/n)$--fraction of them, some property holds, or that for every point of the data set a property holds \emph{with high probability}, e.g., greater or equal to $1-1/n$. Here, $n$ can for example depend on $\#A$ or on $d$.

\section*{Problems}

\begin{probl}\label{PROB-1} Replicate the results of Figure \ref{FIG-1} and Table \ref{TAB-1} by running the  corresponding experiments.
\end{probl}

\begin{probl}\label{PROB-2}Let $X$, $Y\sim\mathcal{N}(0,1,\mathbb{R}^d)$. Show the following.

\begin{compactitem}

\item[(i)] $\forall\:d\geqslant1\colon\E(\|X-Y\|-\sqrt{2d})\leqslant1/\sqrt{2d}$. \vspace{3pt}

\item[(ii)] $\forall\:d\geqslant1\colon\V(\|X-Y\|)\leqslant 3$.

\end{compactitem}

\smallskip

{

\small

\emph{Hint:} Check firstly $\V((X_i-Y_i)^2)=3$ by establishing that $X_i-Y_i\sim\mathcal{N}(0,2,\mathbb{R})$ and by using a suitable formula for computing the fourth moment. Conclude then that $\V(\|X-Y\|^2)\leqslant3d$. Adapt finally the arguments we gave above for $\E(\|X\|-\sqrt{d}\hspace{1pt})$ and $\V(\|X\|)$.

}
\end{probl}

\begin{probl}\label{PROB-3} Replicate the results of Table \ref{TAB-3x}. Make additionally a plot of the distribution of mutual distances (this should give a picture similar to Figure \ref{FIG-1}).
\end{probl}

\begin{probl}\label{PROB-4} Replicate the results of Table \ref{TAB-3}. Let your code also compute the averages and variances of scalar products of the normalized vectors, i.e., $\langle{}x/\|x\|,y/\|y\|\rangle{}$.
\end{probl}

\begin{probl} Compute in a simulation norm, distance and scalar product of points that are drawn from the hypercube $H_d$ (coordinate-wise) uniformly, i.e., $x=(x_1,\dots,x_d)$ is drawn such that $x_i\sim\mathcal{U}([-1,1])$ for $i=1,\dots,d$. Make plots and tables similar to Figure \ref{FIG-1} and Table \ref{TAB-1}\,--\,\ref{TAB-3}. Compare the experimental data with our theoretical results above.
\end{probl}

\begin{probl} If you enjoy horror movies, then watch 2002's \emph{Hypercube}.
\end{probl}

%%%%%%%%%%%%%%%%%%%%%%%%%%%%%%%%%%%%%%%%%%%%%%%%
%%%%%%%%%%%%%%%%%%%%%%%%%%%%%%%%%%%%%%%%%%%%%%%%
%%                                            %%
%% Chapter 9: Volume concentration            %%
%%                                            %%
%%%%%%%%%%%%%%%%%%%%%%%%%%%%%%%%%%%%%%%%%%%%%%%%
%%%%%%%%%%%%%%%%%%%%%%%%%%%%%%%%%%%%%%%%%%%%%%%%

\chapter{Concentration of measure}\label{Ch-VC}

In this chapter we denote by $\lambda^d$ the $d$-dimensional Lebesgue measure on $\mathbb{R}^d$. The euclidean norm we denote by $\|\cdot\|$ and balls with respect to this norm are denoted by $\B_r(x)$ and $\Bbar_r(x)$ respectively.

\smallskip

We start with the following observation. Let $A\subseteq\mathbb{R}^d$ be a measurable star-shaped set and assume that zero is a center. Then $(1-\epsilon)A\subseteq A$ and we can interpret $A\backslash(1-\epsilon)A$ as that part of $A$ that is close to its surface.

\begin{center}
\begin{picture}(150,80)(0,0)
\put(0,0){\begin{tikzpicture}[use Hobby shortcut,closed=true, scale=0.5] 
\node at (7,2.8) {\small$A\backslash{}(1-\epsilon)A$};
\draw(4.35,1.8)--(5.1,2.5);
\fill[pattern=north west lines,draw=black]  (-3.5,0.5) .. (-2.5,2.5) .. (-1,3.5).. (1,2.5).. (3,2.5).. (5,0.5) ..(2.5,-1.2).. (0,-0.5).. (-3,-2).. (-3.5,0.5);
\end{tikzpicture}}
\put(9,8){\begin{tikzpicture}[use Hobby shortcut,closed=true, scale=0.425]
\fill[white,draw=black]  (-3.5,0.5) .. (-2.5,2.5) .. (-1,3.5).. (1,2.5).. (3,2.5).. (5,0.5) ..(2.5,-1.2).. (0,-0.5).. (-3,-2).. (-3.5,0.5);
\end{tikzpicture}}
\end{picture}
\nopagebreak[4]
\begin{fig}Volume close to the surface of a `nice' subset of $\mathbb{R}^2$.\end{fig}
\end{center}

If we compute the measure of the latter set, then we get
$$
\lambda^d(A\backslash(1-\epsilon)A)=\lambda^d(A)-\lambda^d((1-\epsilon)A) = \lambda^d(A)-(1-\epsilon)^d\lambda^d(A)
$$
and thus
$$
\smallfrac{\lambda^d(A\backslash(1-\epsilon)A)}{\lambda^d(A)}=(1-\epsilon)^d
$$
which will be close to one if $\epsilon$ is small and $d$ is large. This means that in this case the measure of $A$ is located close to the surface of $A$ whereas its center contributes only little to its volume. Our first aim in this section is to quantify the latter observation in the case that $A$ is the unit ball.

\smallskip

\begin{lem}\label{LEM-9-1} For $d\geqslant 1$ and $0<\epsilon\leqslant1$ we have $(1-\epsilon)^d\leqslant \exp(-\epsilon d)$.
\end{lem}
\begin{proof} Define $f\colon[0,1]\rightarrow\mathbb{R}$ via $f(x)=\exp(-x)+x-1$. Then $f(0)=0$ and $f'(x)=1-\exp(-x)\geqslant0$ holds for all $x\in[0,1]$. Therefore $f(x)\geqslant0$ and thus $\exp(-x)\geqslant1-x\geqslant0$ holds for all $x\in[0,1]$. Evaluating at $x=\epsilon$ and taking the $d$--th power finishes the proof.
\end{proof}

Using the lemma we can now prove that most of the unit ball's volume is located close to its surface. This is very reasonable, since if we think of the ball consisting of many thin shells all of the same depth but with different radii, then the outer ones will have (in high dimensions significantly!) larger volume then the inner ones.

\begin{thm}\label{SC-THM}(Surface Concentration Theorem) Let $d\geqslant1$ and $0<\epsilon\leqslant1$. Then we have
$$
\lambda^d(\hspace{.5pt}\Bbar_1(0)\backslash \B_{1-\epsilon}(0))\geqslant (1-\exp(-\epsilon d))\cdot\lambda^d(\hspace{.5pt}\Bbar_1(0)),
$$
i.e., at least $1-\exp(-\epsilon d)$ fraction of the volume of the $d$-dimensional unit ball is located $\epsilon$-close to its surface.

\begin{center}
\begin{picture}(135,90)(0,0)
\put(0,-5){\begin{tikzpicture}

     \begin{scope}
\fill[pattern=north west lines,opacity=.6,draw] 
  (0,0) circle [radius=1.5];
  \fill [white, draw=black]  (0,0) circle [radius=1.3];
\end{scope}

% Help lines
%\draw[lightgray] (-3,-2) grid (3,2);
   
% Axis 
\draw[->] (-2.0,0) -- (2.0,0) node[right] {};
\draw[->] (0,-1.8) -- (0,1.8) node[above] {};

\coordinate (A) at (0,0);

% circle
\draw (A) circle[radius=1.5];

\draw (A) circle[radius=1.3];

\node (1) at (-0.15,0.9) {\small$1$};

\node (ep) at (-0.8,0.28) {\small$1\!\shortminus\!\epsilon$};

\draw (1.37,0.3)--(1.9,0.65);

\draw[-latex] (0,0)--(-1,0.83);

\draw[-latex] (0,0)--(-0.6,1.37);

\node at (3.15,0.8) {\small$\Bbar_1(0)\backslash \B_{1-\epsilon}(0)$};
\end{tikzpicture}}
\end{picture}\nopagebreak[4]
\begin{fig}Volume close to the surface of the unit ball.\end{fig}
 \end{center}
\end{thm}
\begin{proof} We compute
\begin{align*}
\frac{\lambda^d(\hspace{.5pt}\Bbar_1(0)\backslash{}\B_{1-\epsilon}(0))}{\lambda^d(\hspace{.5pt}\Bbar_1(0))} &= \frac{\lambda^d(\hspace{.5pt}\Bbar_1(0))-\lambda^d(\hspace{.5pt}\B_{1-\epsilon}(0))}{\lambda^d(\hspace{.5pt}\Bbar_1(0))} = 1-\frac{\lambda^d((1-\epsilon)\B_{1}(0))}{\lambda^d(\hspace{.5pt}\Bbar_1(0))}\\
& = 1-\frac{(1-\epsilon)^d\lambda^d(\hspace{.5pt}\B_{1}(0))}{\lambda^d(\hspace{.5pt}\B_1(0))} = 1-(1-\epsilon)^d\geqslant 1-\exp(-\epsilon d)
\end{align*}
where we used Lemma \ref{LEM-9-1} for the final estimate.
\end{proof}

For the second concentration theorem we need the following lemma.

\begin{lem}\label{LEM-9-2} For $d\geqslant3$ we have $\lambda^d(\hspace{.5pt}\Bbar_1(0))=\frac{2\pi}{d}\cdot\lambda^{d-2}(\hspace{.5pt}\Bbar_1(0))$ where the unit balls are taken in $\mathbb{R}^d$ and $\mathbb{R}^{d-2}$, respectively. 
\end{lem}
\begin{proof} Since $\B_1(0)=\{(x_1,\dots,x_d)\in\mathbb{R}^d\:;\:x_1^2+\cdots+x_d^2\leqslant1\}$ we may compute
\begin{align*}
\lambda^d(\hspace{.5pt}\Bbar_1(0)) & = \int_{\hspace{.5pt}\Bbar_1(0)}1\dd\lambda^d\\
& = \int_{x_1^2+\cdots+x_d^2\leqslant1}1\dd\lambda^d(x_1,\dots,x_d)\\
& = \int_{x_1^2+x_2^2\leqslant1} \Big(\int_{x_3^2+\cdots+x_d^2\leqslant1-x_1^2-x_2^2}1\dd\lambda^{d-2}(x_3,\dots,x_d)\Big)d\lambda^2(x_1,x_2)\\
& = \int_{x_1^2+x_2^2\leqslant1} \lambda^{d-2}(\hspace{.5pt}\Bbar_{(1-x_1^2-x_2^2)^{1/2}}(0))\dd\lambda^2(x_1,x_2)\\
& = \int_{x_1^2+x_2^2\leqslant1}(1-x_1^2-x_2^2)^{(d-2)/2}\lambda^{d-2}(\hspace{.5pt}\Bbar_1(0))\dd\lambda^2(x_1,x_2)\\
& = \lambda^{d-2}\bigl(\hspace{.5pt}\Bbar_1(0)\bigr)\cdot\int_0^{2\pi}\int_0^1(1-r^2)^{(d-2)/2}\,r\dd r \dd\varphi\\
& = \lambda^{d-2}(\hspace{.5pt}\Bbar_1(0))\cdot\int_0^{2\pi} \dd\varphi \cdot \int_1^0u^{(d-2)/2} (-1/2) \dd u\\
& = \lambda^{d-2}(\hspace{.5pt}\Bbar_1(0)) \cdot 2\pi \cdot \int_0^1u^{d/2-1}/2 \dd u\\
& = \lambda^{d-2}(\hspace{.5pt}\Bbar_1(0)) \cdot \pi \cdot \medfrac{u^{d/2}}{d/2}\Big|_0^1 = \medfrac{2\pi}{d}\cdot\lambda^{d-2}(\hspace{.5pt}\Bbar_1(0)) 
\end{align*}
where we used Fubini's Theorem for the third and transformation into polar coordinates for the sixth equality.
\end{proof}

After we have seen that most of the volume of the unit ball is concentrated close to its surface we show in the next result that if we designate one point of the surface as the north pole, then most of the unit ball's volume will be concentrated at its waist. As with the surface concentration, this is not surprising, since if we go from the center in the direction of the north pole, then the ball becomes thinner in all of the $d-1$ remaining dimensions. The precise result is as follows.

\begin{thm}\label{WC-THM}(Waist Concentration Theorem) Let $d\geqslant3$ and $\epsilon>0$. Then we have
$$
\lambda^d(W_{\epsilon})\geqslant\bigl(1-\medfrac{2}{\epsilon\sqrt{d-1}}\exp\bigl({-\medfrac{\epsilon^2(d-1)}{2}}\bigr)\bigr)\cdot\lambda^d(\Bbar_1(0))
$$
where $W_{\epsilon}=\bigl\{(x_1,\dots,x_d)\in \Bbar_1(0)\:;\:|x_1|\leqslant\epsilon\bigr\}$.
\begin{center}
\begin{picture}(190,110)(0,0)
\put(0,0){\begin{tikzpicture}
 
% Help lines
%\draw[lightgray] (-3,-2) grid (3,2);
   
% Axis 
\draw[->] (-2.2,0) -- (2.2,0) node[right] {$\mathbb{R}^{d-1}$};
\draw[->] (0,-1.8) -- (0,1.8) node[above] {$\mathbb{R}$};

\coordinate (A) at (0,0);

\coordinate (B1) at (-3,0.3);
\coordinate (B2) at (3,0.3);

\coordinate (B3) at (-3,-0.3);
\coordinate (B4) at (3,-0.3);

% circle
\draw (A) circle[radius=1.5];

 \tkzInterLC[R](B1,B2)(A,1.5cm)
     \tkzGetPoints{E}{F} 
  
 \tkzInterLC[R](B3,B4)(A,1.5cm) 
     \tkzGetPoints{G}{H}

    \draw (G)--(H);
     \draw (E)--(F);
     
          \draw[shorten <=1.8pt, shorten >=5pt] (H)--(-2.1,-0.6);
          \draw[shorten <=1.8pt, shorten >=5pt] (E)--(-2.1,0.6);
     
     \begin{scope}
\path [clip] (B3)--(B4)--(B2)--(B1)--cycle;
\fill[pattern=north west lines,opacity=.6,draw] 
  (0,0) circle [radius=1.5];
\end{scope}

\node at (-2.2,0.54) {\small${\scriptstyle+}\epsilon$};
 
\node at (-2.2,-0.58) {\small${\scriptstyle-}\epsilon$};

\node (W) at (2,0.95) {\small$W_{\epsilon}$};

\draw[shorten <=8pt] (1.9,1)--(1,0.2);

\node (1) at (-0.15,0.9) {\small$1$};
\draw[-latex] (0,0)--(-0.6,1.37);
\end{tikzpicture}}
\end{picture}
\nopagebreak[4]
\begin{fig}Most of the volume is located close to the equator of the unit ball.\end{fig}
 \end{center}
\end{thm}

\begin{proof} We first observe that the statement is trivial for $\epsilon\geqslant1$, so we may assume $\epsilon<1$ for the rest of the proof. We now compute the volume of that part 
$$
H^{\delta}  = \{(x_1,\dots,x_d)\in \Bbar_1(0)\:;\:x_1\geqslant\delta\}
$$
of the unit ball that is located above $\delta\geqslant0$ in $x_1$-direction. We observe that for $x_1\in[\delta,1]$ the corresponding section of $H^{\delta}$ equals
\begin{align*}
H_{x_1}^{\delta} &=\bigl\{(x_2,\dots,x_d)\in\mathbb{R}^{d-1}\:;\:(x_1,x_2,\dots,x_d)\in H^{\delta}\bigr\}\\
& = \bigl\{(x_2,\dots,x_d)\in\mathbb{R}^{d-1}\:;\:\|(x_1,x_2,\dots,x_d)\|\leqslant1\bigr\}\\
& =\bigl\{y\in\mathbb{R}^{d-1}\:;\:x_1^2+\|y\|^2\leqslant1\bigr\}\\
& = \Bbar_{(1-x_1^2)^{1/2}}(0)
\end{align*}
where the ball is taken in $\mathbb{R}^{d-1}$, compare the picture below.
\begin{center}
\begin{picture}(190,75)(0,0)
\put(0,-5){\begin{tikzpicture}
 
% Help lines
%\draw[lightgray] (-3,-2) grid (3,2);
   
% Axis 
\draw[->] (-2.2,0) -- (2.2,0) node[right] {$\mathbb{R}^{d-1}$};
\draw[->] (0,-0.2) -- (0,1.8) node[above] {$\mathbb{R}$};

\coordinate (A) at (0,0);

\coordinate (B1) at (-3,1);
\coordinate (B2) at (3,1);

\coordinate (B3) at (-3,0.92);
\coordinate (B4) at (3,0.92);

% circle

\draw []  (1.5,0) arc(0:180:1.5);

 \tkzInterLC[R](B1,B2)(A,1.5cm)
     \tkzGetPoints{E}{F} 
  
 \tkzInterLC[R](B3,B4)(A,1.5cm) 
     \tkzGetPoints{G}{H}

    \draw (G)--(H);
     \draw (E)--(F);
      
     \begin{scope}
\path [clip] (B3)--(B4)--(B2)--(B1)--cycle;
\fill[pattern=north west lines] 
  (0,0) circle [radius=1.5];
\end{scope}

\node (1) at (-0.75,0.3) {\small$1$};

\node (x1) at (0.3,0.42) {\small$x_1$};

\node (r) at (-0.5,1.13) {\small$r$};

\draw [dashed]  (0,0) -- (G);

\draw (1.14,0.97) to (1.7,1.4);

\node at (4.4,1.8) {\begin{minipage}{150pt}\small Disk with radius $r=(1-x_1^2)^{1/2}$ in $\mathbb{R}^d$ is a ball with radius $r$ in $\mathbb{R}^{d-1}$.\end{minipage}};
\end{tikzpicture}}
\end{picture}\nopagebreak[4]
\begin{fig} Slice of the unit ball $\Bbar_1(0)$.\end{fig}
 \end{center}

Now we employ Cavalieri's Principle and get
\begin{equation}\label{Caval-1}
\begin{aligned}
 \lambda^d\bigl(H^{\delta}\bigr) & = \int_{\delta}^1\lambda^{d-1}(H^{\delta}_{x_1})\dd\lambda^1(x_1)\\
& = \int_{\delta}^1\lambda^{d-1}(\hspace{0.5pt}\Bbar_{(1-x_1^2)^{1/2}}(0))\dd\lambda^1(x_1)\\
& = \int_{\delta}^1(1-x_1^2)^{(d-1)/2}\lambda^{d-1}(\hspace{0.5pt}\Bbar_1(0))\dd\lambda^1(x_1)\\
& = \lambda^{d-1}(\hspace{0.5pt}\Bbar_1(0))\cdot\int_{\delta}^1(1-x_1^2)^{(d-1)/2}\dd\lambda^1(x_1).
\end{aligned}
\end{equation}

Next we estimate the integral at the end of \eqref{Caval-1} in two ways. Firstly, we put $\delta=0$ and observe that $0\leqslant x_1\leqslant1/\sqrt{d-1}$ implies $x_1^2\leqslant1/(d-1)$. This in turn leads to $1-1/(d-1)\leqslant 1-x_1^2$ which implies $(1-1/(d-1))^{(d-1)/2}\leqslant (1-x_1^2)^{(d-1)/2}$. It follows
\begin{equation}\label{Caval-2}
\begin{aligned}
\int_{0}^1(1-x_1^2)^{(d-1)/2}\dd\lambda^1(x_1) & \geqslant \int_{0}^{1/\sqrt{d-1}}(1-x_1^2)^{(d-1)/2}\dd\lambda^1(x_1)\\
& \geqslant \int_{0}^{1/\sqrt{d-1}}(1-1/(d-1))^{(d-1)/2}\dd\lambda^1(x_1)\\
& =  (1-1/(d-1))^{(d-1)/2}\cdot\medfrac{1}{\sqrt{d-1}}\\
& \geqslant \left(1-\smallfrac{d-1}{2}\cdot\smallfrac{1}{d-1}\right)\cdot\medfrac{1}{\sqrt{d-1}}\\
& = \medfrac{1}{2\sqrt{d-1}}
\end{aligned}
\end{equation}
where we used Bernoulli's inequality in the last estimate as $(d-1)/2\geqslant1$ holds due to our assumption $d\geqslant3$. Secondly, we estimate the integral for $\delta=\epsilon>0$ from above. We employ $1-x_1^2\leqslant \exp(-x_1^2)$ from Lemma \ref{LEM-9-2} and $x_1/\epsilon\geqslant1$ to get
\begin{equation}\label{Caval-3}
\begin{aligned}
\int_{\epsilon}^1(1-x_1^2)^{(d-1)/2}\dd\lambda^1(x_1) & \leqslant \int_{\epsilon}^1(\exp(-x_1^2))^{(d-1)/2}\,(x_1/\epsilon)\,\dd\lambda^1(x_1)\\
& = \medfrac{1}{\epsilon} \cdot \int_{\epsilon}^1x_1\exp(-{\textstyle\frac{d-1}{2}}x_1^2)\dd\lambda^1(x_1)\\
& = \medfrac{1}{\epsilon}\cdot\frac{\exp({\textstyle-\frac{d-1}{2}})- \exp({\textstyle-\frac{\epsilon^2(d-1)}{2}})}{-(d-1)}\\
& \leqslant \frac{\exp({\textstyle-\frac{\epsilon^2(d-1)}{2}})}{\epsilon(d-1)}.
\end{aligned}
\end{equation}

Combining \eqref{Caval-1}--\eqref{Caval-3} we obtain
$$
\lambda^d(H^0)\geqslant\frac{\lambda^{d-1}(\hspace{0.5pt}\Bbar_1(0))}{2\sqrt{d-1}} \; \text{ and } \; \lambda^d(H^{\epsilon})\leqslant \lambda^{d-1}(\hspace{0.5pt}\Bbar_1(0)) \frac{\exp({\textstyle-\frac{\epsilon^2(d-1)}{2}})}{\epsilon(d-1)}.
$$

To finish the proof we observe that $H^0$ is the upper half of the unit ball and $H^0\backslash{}H^{\epsilon}$ is the upper half of $W_{\epsilon}$. Therefore we derive from the above
\begin{align*}
\frac{\lambda^d(W_{\epsilon})}{\lambda^d(\hspace{0.5pt}\Bbar_1(0))} & = \frac{2\cdot\lambda^d(H^0\backslash{}H^{\epsilon})}{2\cdot\lambda^d(H^0)} = \frac{\lambda^d(H^0)-\lambda^d(H^{\epsilon})}{\lambda^d(H^0)} = 1-\frac{\lambda^d(H^{\epsilon})}{\lambda^d(H^{0})}\\
& \geqslant 1-\frac{2\sqrt{d-1}\exp({\textstyle-\frac{\epsilon^2(d-1)}{2}})}{\epsilon(d-1)} = 1-\smallfrac{2}{\epsilon\sqrt{d-1}}\exp\bigl({-\smallfrac{\epsilon^2(d-1)}{2}}\bigr)
\end{align*}
as desired.
\end{proof}

The estimate in Theorem \ref{WC-THM} is only of value if the factor on the right hand side is strictly positive. The following lemma explains when this happens.

\begin{lem}\label{WC-LEM} There exists a unique $a_0>0$ such that $\frac{2}{a}\exp(-\frac{a^2}{2})<1$ holds for every $a>a_0$. We have $a_0\in(1,2)$.
\end{lem}

\begin{proof} The function $f\colon(0,\infty)\rightarrow\mathbb{R}$, $f(a)=\frac{2}{a}\exp(-\frac{a^2}{2})$ is strictly decreasing as both of its factors are strictly decreasing and strictly positive. We have $f(1)=2/\sqrt{e}>1$ and $f(2)=1/\e^2<1$. Consequently there is $a_0>0$ such that $f(a_0)=1$ and necessarily $1<a_0<1$. For $a>a_0$ we get $f(a)<1$ and for $a<a_0$ we get $f(a)>1$, so $a_0$ is unique.
\end{proof}

Substituting $a=\epsilon\sqrt{d-1}$ shows that the factor in Theorem \ref{WC-THM} is strictly positive if and only if $\epsilon>a_0/\sqrt{d-1}$ holds with some $a_0\in(1,2)$, compare Problem \ref{Pa0}. That is, as bigger as we choose the dimension, the smaller we can make $\epsilon>0$.

\smallskip

We however \emph{cannot} let $\epsilon$ tend to zero for a fixed dimension $d$. The interpretation of Theorem \ref{WC-THM}, i.e., that $1-\frac{2}{\epsilon\sqrt{d-1}}\exp({-\frac{\epsilon^2(d-1)}{2}})$ fraction of the volume of the $d$-dimensional unit ball is located $\epsilon$-close to the equator, is only valid if $\epsilon$ is not too small. Although this might right now be fairly obvious, we point out that in contrast to the `typical epsilontics', $\epsilon>0$ cannot be thought of being arbitrarily small, but what we can choose for it depends on the space's dimension to be large enough.

\smallskip

For later use we note the following reformulation of Theorem \ref{WC-THM}.

\begin{cor}\label{WC-COR} Let $d\geqslant3$ and $\epsilon>0$. Then we have
$$
\frac{\lambda^d(\{(x_1,\dots,x_d)\in\Bbar_1(0)\:;\:|x_1|>\epsilon\})}{\lambda^d(\hspace{0.5pt}\Bbar_1(0))} \leqslant\smallfrac{2}{\epsilon\sqrt{d-1}}\exp\bigl({-\smallfrac{\epsilon^2(d-1)}{2}}\bigr)
$$
for every $\epsilon>0$.
\end{cor}
\begin{proof} In the notation of Theorem \ref{WC-THM} we have
\begin{align*}
\frac{\lambda^d(\{(x_1,\dots,x_d)\in\Bbar_1(0)\:;\:|x_1|>\epsilon\})}{\lambda^d(\hspace{0.5pt}\Bbar_1(0))} &= \frac{\lambda^d(\hspace{0.5pt}\Bbar_1(0)\backslash W_{\epsilon})}{\lambda^d(\hspace{0.5pt}\Bbar_1(0))}= \frac{\lambda^d(\hspace{0.5pt}\Bbar_1(0))-\lambda^d(W_{\epsilon})}{\lambda^d(\hspace{0.5pt}\Bbar_1(0))}\\
&\leqslant 1 - \Bigr( 1- \smallfrac{2}{\epsilon\sqrt{d-1}}\exp\bigl({-\smallfrac{\epsilon^2(d-1)}{2}}\bigr)\Bigl)\\
& = \smallfrac{2}{\epsilon\sqrt{d-1}}\exp\bigl(-\smallfrac{\epsilon^2(d-1)}{2}\bigr)
\end{align*}
as claimed.
\end{proof}

Now we will use the volume estimates from above in order to establish where points that we draw at random from the $d$-dimensional unit ball will be located and what mutual angles between the points we can expect. To make our proofs formal, we will employ in the rest of the chapter a probability space $(\Omega,\Sigma,\P)$ and random variables $X\colon\Omega\rightarrow\mathbb{R}^d$ with distribution $X\sim\mathcal{U}(\Bbar_1(0))$, compare Definition \ref{DFN-DISTR}.

\begin{thm}\label{UA-THM} Let $d\geqslant3$ and assume that we draw $n\geqslant2$ points $x^{\scriptscriptstyle(1)},\dots,x^{\scriptscriptstyle(n)}$ at random from the $d$-dimensional unit ball. Then
$$
\P\bigl[\|x^{\scriptscriptstyle(j)}\|\geqslant 1-\medfrac{2\ln n}{d} \text{ for all }j=1,\dots,n\bigr]\geqslant 1-\medfrac{1}{n}
$$
holds.
\end{thm}
\begin{proof} In view of the Surface Concentration Theorem, most of the unit ball's volume is located close to its surface and thus a point chosen uniformly at random is likely to be close to the surface. For a formal proof, we put $\epsilon=\frac{2\ln n}{d}$ and consider a random vector $X\sim\mathcal{U}(\Bbar_1(0))$. For fixed $1\leqslant j\leqslant n$ we employ Theorem \ref{SC-THM} to get
$$
\P\bigl[\|X^{\scriptscriptstyle(j)}\|<1-\epsilon\hspace{0.5pt}\bigr]=\frac{\lambda^d(\hspace{0.5pt}\B_{1-\epsilon}(0))}{\lambda^d(\hspace{0.5pt}\Bbar_{1}(0))}\leqslant\exp(-\epsilon{}d)=\exp(-2\ln n)=1/n^2.
$$
This implies
\begin{align*}
\P\bigl[\hspace{1.5pt}\forall\:j\colon\|X^{\scriptscriptstyle(j)}\|\geqslant 1-\epsilon\hspace{0.5pt}\bigr] &=1-\P\bigl[\hspace{1.5pt}\exists\:j\colon\|X^{\scriptscriptstyle(j)}\|\leqslant 1-\epsilon\hspace{0.5pt}\bigr]\\
&= 1-\Bigsum{j=1}{n}1/n^2=1-\medfrac{1}{n}
\end{align*}
as claimed.
\end{proof}

Although we get a trivial estimate if we sub $n=1$ into the Theorem, its proof showed that
$$
\P\bigl[\hspace{1pt}\|x\|\geqslant 1-\epsilon\hspace{1pt}\bigr] \geqslant 1-\exp(-\epsilon d)
$$
holds, so if $d$ is large enough and $\epsilon>0$ is not too small, a single point chosen at random will be $\epsilon$--close to the surface of $\Bbar_1(0)$ with a high probability. We continue by studying angles between points.

\begin{thm}\label{UO-THM} Let $d\geqslant3$ and assume that we draw $n\geqslant2$ points $x^{\scriptscriptstyle(1)},\dots,x^{\scriptscriptstyle(n)}$ uniformly at random from the $d$-dimensional unit ball. Then
$$
\P\bigl[\hspace{1pt}\bigl|\langle{}x^{\scriptscriptstyle(j)},x^{\scriptscriptstyle(k)}\rangle{}\bigr|\leqslant\medfrac{\sqrt{6\ln n}}{\sqrt{d-1}} \text{ for all }j\not=k\hspace{1pt}\bigr]\geqslant 1-\medfrac{1}{n}
$$
holds.
\end{thm}
\begin{proof} We first consider the case of two points $x, y\in\Bbar_1(0)$. As we are interested only in their scalar product it should make no difference if we draw both points at random or if we treat the first point's direction as fixed (but the length still chosen at random) and draw only the second one at random. Since $\lambda^d(\{0\})=0$, the first points will be non-zero with probability one and we can thus arrange a new coordinate system, in which the unit ball is still the same, but in which our first points is directed towards the (new) north pole of the ball. In these coordinates we can then apply the waist concentration theorem to see that the other point is likely to be located close to the new equator. This means in turn that the scalar product of both points will be close to zero.

\begin{center}
\begin{picture}(190,100)(0,0)
\put(0,30){\rotatebox{-20}{\begin{tikzpicture}
 
% Help lines
%\draw[lightgray] (-3,-2) grid (3,2);
   
% Axis 
\draw[->] (-2.2,0) -- (2.2,0) node[right] {$\mathbb{R}^{d-1}$};
\draw[->] (0,-1.8) -- (0,1.8) node[above] {$\mathbb{R}$};

\coordinate (A) at (0,0);

\coordinate (B1) at (-3,0.3);
\coordinate (B2) at (3,0.3);

\coordinate (B3) at (-3,-0.3);
\coordinate (B4) at (3,-0.3);

% circle
\draw (A) circle[radius=1.5];

 \tkzInterLC[R](B1,B2)(A,1.5cm)
     \tkzGetPoints{E}{F} 
  
 \tkzInterLC[R](B3,B4)(A,1.5cm) 
     \tkzGetPoints{G}{H}

    \draw (G)--(H);
     \draw (E)--(F);
     
          %\draw[shorten <=1.8pt, shorten >=5pt] (H)--(-2.1,-0.6);
          %\draw[shorten <=1.8pt, shorten >=5pt] (E)--(-2.1,0.6);
     
     \begin{scope}
\path [clip] (B3)--(B4)--(B2)--(B1)--cycle;
\fill[pattern=north west lines,opacity=.6,draw] 
  (0,0) circle [radius=1.5];
\end{scope}

%\node at (-2.2,0.54) {\small${\scriptstyle+}\epsilon$};
 
%\node at (-2.2,-0.58) {\small${\scriptstyle-}\epsilon$};

%\node (1) at (-0.35,1.15) {\small$e^{\scriptscriptstyle(1)}$};
%\draw[-latex] (0,0)--(0,1.5);

\node (1) at (-0.26,1.67) {\rotatebox{20}{\small$x$}};
\draw[-latex] (0,0)--(0,1.45);

\node (1) at (1.7,0.25) {\rotatebox{20}{\small$y$}};
\draw[-latex] (0,0)--(1.4,0.2);

\end{tikzpicture}}}
\end{picture}\nopagebreak[4]
\begin{fig} We apply he Waist Concentration Theorem\\with respect to the new north pole.\end{fig}
\end{center}

To make the above idea formal, we put $\epsilon=\sqrt{6\ln n}/\sqrt{d-1}$ and consider two random variables $X,Y\colon\Omega\rightarrow\mathbb{R}^d$ with $X,Y\sim\mathcal{U}(\Bbar_1(0))$. Then we can find for each $\omega_1\in\Omega$ with $X(\omega_1)\not=0$ an orthogonal linear map $T(\omega)\colon\mathbb{R}^d\rightarrow\mathbb{R}^d$ with $T(\omega_1)X(\omega_1)=(\|X(\omega_1)\|,0,\dots,0)$. Indeed, we can put $f^{\scriptscriptstyle(1)}=X(\omega)/\|X(\omega)\|$ and extend to an orthonormal basis $\mathcal{F}=\{ f^{\scriptscriptstyle(1)}, f^{\scriptscriptstyle(2)},\dots,f^{\scriptscriptstyle(d)}\}$ by first applying the basis extension theorem and then the Gram-Schmidt procedure. The map $T(\omega_1)$ is then defined as the linear extension of $T(\omega_1)f^{\scriptscriptstyle(j)}:=e^{\scriptscriptstyle(j)}$ where the $e^{\scriptscriptstyle(j)}$ are the standard unit vectors. We obtain $T(\omega_1)X(\omega_1)=T(\omega_1)(\|X(\omega_1)\|f^{\scriptscriptstyle(1)})=\|X(\omega_1)\|e^{\scriptscriptstyle(1)}$ and thus
$$
|\langle{}X(\omega_1),Y(\omega_2)\rangle{}| = |\langle{}T(\omega_1)X(\omega_1),T(\omega_1)Y(\omega_2)\rangle{}| = \|X(\omega_1)\|\cdot|\langle{}e^{\scriptscriptstyle(1)},T(\omega_1)Y(\omega_2)\rangle{}|
$$
for $(\omega_1,\omega_2)\in\Omega\times\Omega$ with $X(\omega_1)\not=0$. We observe that $T(\omega_1)Y\sim\mathcal{U}(\Bbar_1(0))$ holds for each fixed $\omega_1$ and that $\langle{}e^{\scriptscriptstyle(1)},T(\omega_1)Y(\omega_2)\rangle{}$ is the first coordinate of $T(\omega_1)Y(\omega_2)$. Therefore we can apply Corollary \ref{WC-COR} to compute
\begin{align*}
\P\bigl[|\langle{}e^{\scriptscriptstyle(1)},T(\omega_1)Y(\omega_2)\rangle|>\epsilon\bigr] & = \frac{\lambda^d(\{(z_1,\dots,z_d)\in\Bbar_1(0)\:;\:|z_1|>\epsilon\})}{\lambda^d(\hspace{.5pt}\Bbar_1(0))}\\
&\leqslant \medfrac{2}{\epsilon\sqrt{d-1}}\exp\bigl(-\medfrac{\epsilon^2(d-1)}{2}\bigr)\\
& = \medfrac{2}{\sqrt{6\ln n}}\exp\bigl(-\medfrac{6\ln n}{2}\bigr)\\
\phantom{\int}& \leqslant \medfrac{2}{n^3}
\end{align*}
where the last estimate holds as $1/\sqrt{6\ln n}\leqslant1$ is valid for $n\geqslant2$. After observing that in the first line below necessarily $X(\omega_1)\not=0$ holds and by using $\|X(\omega_1)\|\leqslant 1$, we get
\begin{align*}
\P\bigl[\hspace{1pt}\bigl|\langle{}X,Y\rangle{}\bigr|>\epsilon\hspace{1pt}\bigr] & = \P\bigl(\{(\omega_1,\omega_2)\in\Omega^2\:;\:\bigl|\bigl\langle{}X(\omega_1),Y(\omega_2)\bigr\rangle{}\bigr|>\epsilon\}\bigr)\\
\phantom{\int}& =\P\bigl(\{(\omega_1,\omega_2)\in\Omega^2\:;\:\|X(\omega)\|\cdot\bigl|\bigl\langle{}e^{\scriptscriptstyle(1)},T(\omega)Y(\omega)\bigr\rangle{}\bigr|>\epsilon\}\bigr)\\
\phantom{\int}& \leqslant\P\bigl(\{(\omega_1,\omega_2)\in\Omega^2\:;\:\bigl|\bigl\langle{}e^{\scriptscriptstyle(1)},T(\omega_1)Y(\omega_2)\bigr\rangle{}\bigr|>\epsilon\}\bigr)\\
& = \int_{\Omega}  \P\bigl[\hspace{1pt}\bigl|((T(\omega_1)Y(\omega_2))_1\bigr|>\epsilon\hspace{1pt}\bigr]\dd\P(\omega_1)\\
&\leqslant \int_{\Omega}1/n^3\dd\P = \medfrac{1}{n^3}
\end{align*}
where we employed Cavalieri's Principle and the estimates from above. Let finally, random vectors $X^{\scriptscriptstyle(1)},\dots,X^{\scriptscriptstyle(n)}\sim\mathcal{U}(\Bbar_1(0))$ be given. Then 
\begin{align*}
\P\bigl[\hspace{1.5pt}\forall\:j\not=k\colon|\langle{}X^{\scriptscriptstyle(j)},X^{\scriptscriptstyle(k)}\rangle|\leqslant\epsilon\hspace{1pt}\bigr] & = 1- \P\bigl[\hspace{1.5pt}\exists\:j\not=k\colon\bigl|\langle{}X^{\scriptscriptstyle(j)},X^{\scriptscriptstyle(k)}\rangle\bigr|>\epsilon \hspace{1pt}\bigr]\\
& = 1-{\textstyle{n\choose2}}\P\bigl[\bigl|\langle{}X,Y\rangle{}\bigr|>\epsilon\hspace{1pt}\bigr]\\
&\geqslant 1-\medfrac{n^2-n}{2}\cdot\medfrac{2}{n^3}\\
&\geqslant 1-\medfrac{1}{n}
\end{align*}
as desired.
\end{proof}

If we apply the estimate  of Theorem \ref{UO-THM} for only two points $x$, $y\in\Bbar_1(0)$ we get only a probability of greater or equal to $1/2$ for $\langle{}x,y\rangle{}$ being smaller than $\sqrt{6\ln 2}/\sqrt{d-1}$. Going through the proof of of the theorem yields however that 
$$
\P\bigl[\hspace{1.5pt}|\langle{}x,y\rangle{}|\leqslant\epsilon\hspace{1.5pt}\bigr]\geqslant 1-\medfrac{2}{\epsilon\sqrt{d-1}}\exp\bigl(-\medfrac{\epsilon^2(d-1)}{2}\bigr)
$$
from where we can see that for fixed $\epsilon>0$ and $d$ large enough the probability for the scalar product to be $\epsilon$--small will be close to one which means that two points picked at random from the high dimensional unit ball will be almost orthogonal with high probability.
%\footnote{\red{Is there a result analogously to Theorem \ref{GO-THM-2} that boundes $\P[|\langle{}x,y\rangle{}|\leqslant\epsilon]\leqslant\cdots$ from above?}}

\smallskip

We remark here, that when interpreting the scalar product as a measure for how close to orthogonal two vectors are, one usually does normalize the vectors first. In the above result we did not do so, since, by Theorem \ref{UA-THM} we know that the points considered in Theorem \ref{UO-THM} will have norm close to one with high probability, anyway. We leave a formal proof as Problem \ref{PROB-NORMALIZED-UA-THM}.

\section*{Problems}

\begin{probl} Let $\Bbar_1(0)$ denote the unit ball of $\mathbb{R}^d$.
\begin{compactitem}

\item[(i)] Use Lemma \ref{LEM-9-2} to compute $\lambda^d(\Bbar_1(0))$ for $d=1,\dots,10$.

\vspace{3pt}

\item[(ii)] Compute $\lim_{d\rightarrow\infty}\lambda^d(\Bbar_1(0))$.

\vspace{0pt}

\item[(iii)] Show that $\lambda^d(\Bbar_1(0))=\medfrac{\pi^{d/2}}{\Gamma(d/2+1)}$
holds, where $\Gamma\colon(0,\infty)\rightarrow\mathbb{R}$ denotes the Gamma function.
\end{compactitem}
\end{probl}

\smallskip

\begin{probl}\label{Pa0} Show that in Lemma \ref{WC-LEM} indeed $a_0\in(1,\sqrt{5}-1)$ holds.

\smallskip

{

\small

\textit{Hint:} Establish first that $\exp(x)\leqslant 1+x+x^2/2$ holds for $x\leqslant0$.

}
\end{probl}

\begin{probl}\label{PROB-NORMALIZED-UA-THM} Let $d\geqslant3$ and $n$ be such that $2\ln(n)\leqslant d$ holds. Show that
$$
\P\bigl[\bigl|\bigl\langle{}\medfrac{x^{\scriptscriptstyle(j)}}{\|x^{\scriptscriptstyle(j)}\|},\medfrac{x^{\scriptscriptstyle(k)}}{\|x^{\scriptscriptstyle(k)}\|}\bigr\rangle{}\bigr|\leqslant\medfrac{\sqrt{6\ln n}}{\sqrt{d-1}} \text{ for all }j\not=k\bigr]\geqslant 1-\medfrac{1}{n}
$$
holds, when $x^{\scriptscriptstyle(1)},\dots,x^{\scriptscriptstyle(n)}$ are drawn uniformly at random from the $d$--dimensional unit ball.

\smallskip

{

\small

\textit{Hint:} Use that $1\leqslant d^2/(d-2\ln n)^2$ holds and apply then the Theorem of Total Probability.

}

\end{probl}

%%%%%%%%%%%%%%%%%%%%%%%%%%%%%%%%%%%%%%%%%%%%%%%%
%%%%%%%%%%%%%%%%%%%%%%%%%%%%%%%%%%%%%%%%%%%%%%%%
%%                                            %%
%% Chapter 8: Tail bound estimtes             %%
%%                                            %%
%%%%%%%%%%%%%%%%%%%%%%%%%%%%%%%%%%%%%%%%%%%%%%%%
%%%%%%%%%%%%%%%%%%%%%%%%%%%%%%%%%%%%%%%%%%%%%%%%

\chapter{The Bernstein tail bound}\label{Ch-TB}

We start this chapter by explaining a general technique to obtain bounds  for $\P\bigl[\hspace{1.5pt}Y_1+\cdots+Y_d\geqslant a\hspace{.5pt}\bigr]$ where the $Y_i$ are independent random variables. A bound that arises from this `recipe' is often referred to as a Chernoff bound. The basic idea is given in the following Lemma \ref{CHERNOFF-RECEIPE} which establishes the `generic' Chernoff bound.

\begin{lem}\label{CHERNOFF-RECEIPE} Let $(\Omega,\Sigma,\P)$ be a probability space and let $Y_1,\dots,Y_n\colon\Omega\rightarrow\mathbb{R}$ be mutually independent random variables. Then
$$
\P\bigl[\hspace{1.5pt}Y_1+\cdots+Y_n\geqslant a\hspace{.5pt}\bigr]\leqslant \inf_{t>0}\exp(-ta)\prod_{i=1}^n\E(\exp(tY_i))
$$
holds for every $a>0$.
\end{lem}
\begin{proof} For fixed $t>0$ the function $\exp(t\cdot\shortminus)$ is increasing. We put $Y=Y_1+\cdots+Y_n$ and get 
$$
\P[Y\geqslant a] = \P[\exp(tY)\geqslant \exp(ta)]\leqslant\frac{\E(\exp(tY))}{\exp(ta)}
$$
by applying the Markov bound. Now we take the infimum over $t>0$ and employ the independence to get
$$
\P[Y\geqslant a]\leqslant \inf_{t>0}\exp(-ta)\E(\exp(t\Bigsum{i=1}{n} tY_i))=\inf_{t>0}\exp(-ta)\prod_{i=1}^n\E(\exp(tY_i))
$$
as desired.
\end{proof}

The classical Chernoff bound recipe is now to compute or estimate the right hand side by exploiting additional information on the $Y_i$ that is given in concrete situations. Readers unfamiliar with Chernoff bounds can find a classical example on how the recipe works in Problem \ref{PROBL-SIMPLE-CHERNOFF} where the $Y_i$ are independent Bernoulli random variables. In Theorem \ref{MTB-THM} below the additional knowledge comes from the assumption that the moments of each $Y_i$ grow at most like factorial.

\begin{thm}\label{MTB-THM} Let $(\Omega,\Sigma,\P)$ be a probability space and let $Y_1,\dots,Y_d\colon\Omega\rightarrow\mathbb{R}$ be mutually independent random variables with expectation zero and $|\E(Y_i^k)|\leqslant k!/2$ for $i=1,\dots,d$ and $k\geqslant2$. Then we have
$$
\P\bigl[|Y_1+\cdots+Y_d|\geqslant a\bigr]\leqslant 2\exp\bigl(-C\min\bigl(\medfrac{a^2}{d},a\bigr)\bigr)
$$
for every $a>0$ and with $C=1/4$.
\end{thm}
\begin{proof} Let $Y$ be one of the random variables $\pm Y_1,\dots,\pm Y_d$ and $0<t\leqslant 1/2$. Then we estimate
\begin{align*}
\E(\exp(tY)) & = \bigl|\E\bigl(1+tY +\Bigsum{k=2}{\infty}\medfrac{(tY)^k}{k!}\bigr)\bigr| \leqslant 1 + \Bigsum{k=2}{\infty}\medfrac{t^k|\E(Y^k)|}{k!}\\
& \leqslant 1 + \Bigsum{k=2}{\infty}\medfrac{t^k}{2} = 1 + \smallfrac{1}{2}\bigl(\medfrac{1}{1-t} - (t+1)\bigr)\\
& = 1 + \medfrac{1}{2}\medfrac{t^2}{1-t}\leqslant 1+t^2\leqslant \exp(t^2)
\end{align*}
where the last equality can easily be checked by means of calculus. Now we use Lemma \ref{CHERNOFF-RECEIPE} and get
\begin{align*}
\P\bigl[\hspace{1.5pt}Y_1+\cdots+Y_d\geqslant a\hspace{.5pt}\bigr] & \leqslant \inf_{t>0}\exp(-ta)\prod_{i=1}^d\E(\exp(tY_i))\\
& \leqslant \inf_{0<t\leqslant1/2}\exp(-ta)\prod_{i=1}^d\exp(t^2)\\
& \leqslant \inf_{0<t\leqslant1/2}\exp(-ta+dt^2).
\end{align*}
By solving $\frac{\dd}{\dd t}(-ta+dt^2)=-a+2dt=0$ for $t$ and considering the constraint $t\leqslant1/2$ this leads to $t=\min(\frac{a}{2d},\frac{1}{2})$, and further to
$$
-ta+dt^2 \leqslant -\medfrac{1}{2}ta -\medfrac{1}{2}ta + dt\medfrac{a}{2d}=-\medfrac{1}{2}ta \leqslant -\medfrac{1}{2}\min\bigl(\medfrac{a}{2d},\medfrac{1}{2}\bigr)a = -\min\bigl(\medfrac{a^2}{4d},\medfrac{a}{4}\bigr).
$$ 
This implies
$$
\P\bigl[\hspace{1.5pt}Y_1+\cdots+Y_d\geqslant a\hspace{.5pt}\bigr] \leqslant \exp\Bigl(-\min\bigl(\medfrac{a^2}{4d},\medfrac{a}{4}\bigr)\Bigr)
$$
and similiarly we get
$$
\P\bigl[\hspace{1.5pt}Y_1+\cdots+Y_d\leqslant -a\hspace{.5pt}\bigr] = \P\bigl[-Y_1-\cdots-Y_d\geqslant a\hspace{.5pt}\bigr] \leqslant \exp\Bigl(-\min\bigl(\medfrac{a^2}{4d},\medfrac{a}{4}\bigr)\Bigr).
$$
Consequently,
$$
\P\bigl[\hspace{1.5pt}\bigl|Y_1+\cdots+Y_d\bigr|\geqslant a\hspace{.5pt}\bigr] \leqslant 2\exp\Bigl(-\min\bigl(\medfrac{a^2}{4d},\medfrac{a}{4}\bigr)\Bigr)
$$
as desired.
\end{proof}

\section*{Problems}

\begin{probl}\label{PROBL-SIMPLE-CHERNOFF}(Classic example for a Chernoff bound) Let $Y_1,\dots,Y_n$ be independent Bernoulli random variables with $\P[X_i=1]=p\in[0,1]$ and $Y=Y_1+\cdots+Y_n$. Let $\delta>0$.

\begin{compactitem}

\item[(i)] Show that $\E(\exp(tY_i))\leqslant\exp(p(\exp(t)-1))$ holds for every $t>0$.

\vspace{3pt}

\item[(ii)] Use Lemma \ref{CHERNOFF-RECEIPE} to conclude the following classic Chernoff bound
$$
\P\bigl[\hspace{.5pt}X\geqslant(1+\delta)np\hspace{1.5pt}\bigr]\leqslant\Bigl(\smallfrac{\e^{\delta}}{(1+\delta)^{1+\delta}}\Bigr)^{np}.
$$

{

\small

\emph{Hint: }It is often not necessary to compute the infimum in Lemma \ref{CHERNOFF-RECEIPE} explicitly. Here, one can for example simply choose $t=\log(1+\delta)$.

}

\vspace{3pt}

\item[(iii)] Assume you are rolling a fair dice $n$ times. Apply (ii) to estimate the probability to roll a six in at least 70\% of the experiments.

\vspace{3pt}

\item[(iv)] Compare the estimate of (ii) with what you get when applying the Markov bound respectively the Chebychev bound, instead. Run a simulation of the experiment to test how tight the predictions of the three bounds are.

\end{compactitem}

\end{probl}

\begin{probl}\label{PROBL-EXP-CHERNOFF} (`naive exponential estimate') Let $Y_1,\dots,Y_d$ be independent random variables and assume that $|\E(Y_i^k)|\leqslant k!$ holds for all $k\geqslant0$ and $i=1,\dots,d$.

\begin{compactitem}

\item[(i)] Use the series expansion of $\exp(\cdot)$ and the assumption to get
$$
\E(\exp(tY_i))\leqslant \Bigsum{k=0}{\infty}t^k=\left\{\begin{array}{cl}\textfrac{1}{1-t} &\text{ for } t\in(-1,1),\\\infty & \text{otherwise.}\end{array}\right.
$$

\item[(ii)] Show by means of calculus that
$$
\inf_{t\in(0,1)}\exp(-ta)\prod_{i=1}^d\smallfrac{1}{1-t}=\left\{\begin{array}{cl}(\textfrac{a}{d})^d\exp(d-a) &\text{ if } a>d,\\ 1 & \text{otherwise.}\end{array}\right.
$$

\vspace{4pt}

\item[(iii)] Derive an estimate for $\P\bigl[|Y_1+\cdots+Y_d|\geqslant a\bigr]$ from the above.

\vspace{4pt}

\item[(iv)] Compare the bound in (iii) with the bound of Theorem \ref{MTB-THM}.

\end{compactitem}
\end{probl}

%%%%%%%%%%%%%%%%%%%%%%%%%%%%%%%%%%%%%%%%%%%%%%%%
%%%%%%%%%%%%%%%%%%%%%%%%%%%%%%%%%%%%%%%%%%%%%%%%
%%                                            %%
%% Chapter 10: Random vectors                 %%
%%                                            %%
%%%%%%%%%%%%%%%%%%%%%%%%%%%%%%%%%%%%%%%%%%%%%%%%
%%%%%%%%%%%%%%%%%%%%%%%%%%%%%%%%%%%%%%%%%%%%%%%%

\chapter{Gaussian random vectors}\label{Ch-RV}

We now turn to points that are chosen at random from the whole space according to a Gaussian distribution with mean zero and unit variance. Our first result is the so-called \emph{Gaussian Annulus Theorem} which states that a point chosen at random will, with high probability, be located in a small annulus around the sphere with radius $\sqrt{d}$. On a first sight that might seem surprising as the density function has its largest values around zero and decays if we move away from zero. The results of Chapter \ref{Ch-VC} however show that close to the origin there is located only little volume and since for the probability to be in a Borel set $A\subseteq\mathbb{R}^d$ we ought to integrate the density function over $A$, the lack of volume around the origin makes this integral small if $A$ is located in the vicinity of zero. At a certain radius then the concentration of volume compensates the smaller values of the density function. Going beyond this radius, the small values of the density function determine the value of the integral and make it small again.

\smallskip

Before we can start, we need to fill in several facts about Gaussian random vectors, namely that

\begin{compactitem}

\item[\filledsquare] spherical Gaussian random vectors are precisely those which are independently coordinate-wise Gaussian (Proposition \ref{PROP-9-1}),\vspace{1pt}
\item[\filledsquare] a linear combination of normal random variables is  Gaussian with mean zero and variance given by the sum of the squares of the coefficients (Proposition \ref{LK-GAUSS-LEM}).

\end{compactitem}

The reader familiar with the above can skip the proofs and jump directly to Lemma \ref{LEM-10-1} which will be the final preparation for the Gaussian Annulus Theorem.

\smallskip

\begin{prop}\label{PROP-9-1} Let $(\Omega,\Sigma,\P)$ be a probability space and $X\colon\Omega\rightarrow\mathbb{R}^d$ be a random vector with coordinate functions $X_1,\dots,X_d$. Then $X$ is spherically Gaussian distributed with mean zero and variance $\sigma>0$ if and only if its coordinates are independent and
$$
\P[\hspace{1pt}X_i\in A\hspace{1pt}] = {\medfrac{1}{\sqrt{2\pi}\sigma}} \int_A\exp\bigl(-\medfrac{x^2}{2\sigma^2}\bigr)\dd\lambda(x)
$$
holds for every Borel set $A\subseteq\mathbb{R}$ and every $i=1,\dots,d$.
\end{prop}
\begin{proof} Let $A=A_1\times\cdots\times A_d\subseteq\mathbb{R}^d$ be given with Borel sets $A_i\subseteq\mathbb{R}$. By Fubini's Theorem we compute
\begin{align*}
\P[\hspace{1pt}X\in A\hspace{1pt}] &= {\medfrac{1}{(2\pi\sigma^2)^{d/2}}} \int_A\exp\bigl(-\medfrac{\|x\|^2}{2\sigma^2}\bigr)\dd\lambda^d(x)\\
& = {\medfrac{1}{(2\pi\sigma^2)^{d/2}}} \int_A\exp\bigl(-\medfrac{x_1^2+\cdots+x_d^2}{2\sigma^2}\bigr)\dd\lambda^d(x)\\
& = {\medfrac{1}{(2\pi\sigma^2)^{d/2}}} \int_A\exp\bigl(-\medfrac{x_1^2}{2\sigma^2}\bigr)\cdots\exp\bigl(-\medfrac{x_d^2}{2\sigma^2}\bigr)\dd\lambda^d(x)\\
& = \medfrac{1}{\sqrt{2\pi}\sigma} \int_{A_1}\exp\bigl(-\medfrac{x_1^2}{2\sigma^2}\bigr)\dd\lambda(x_1)\cdots \medfrac{1}{\sqrt{2\pi}\sigma}\int_{A_d}\exp\bigl(-\medfrac{x_d^2}{2\sigma^2}\bigr)\dd\lambda(x_d)\\
\phantom{\int}& = \P[\hspace{.5pt}X_1\in A_1\hspace{.5pt}]\cdots\P[\hspace{.5pt}X_d\in A_d\hspace{.5pt}]
\end{align*}
where the first equality is true if $X\sim\mathcal{N}(0,1)$ holds and the last if $X_i\sim\mathcal{N}(0,1)$ for each $i=1,\dots,d$. This computation now shows both implications.

\smallskip

\textquotedblleft{}$\Rightarrow$\textquotedblright{} Let $B\subseteq\mathbb{R}$ be Borel and let $1\leqslant i\leqslant d$ be fixed. Then we have $\P[X_i\in B] = \P[X\in\mathbb{R}^{i-1}\times{}B\times\mathbb{R}^{d-i}]$ and reading the computation above from the beginning until line four shows
$$
\P[\hspace{0.5pt}X_i\in B\hspace{0.5pt}] = \medfrac{1}{\sqrt{2\pi}\sigma} \int_{B}\exp\bigl(-\medfrac{x_i^2}{2\sigma^2}\bigr)\dd\lambda(x_i)
$$
as all the other integrals are equal to one.

\smallskip

\textquotedblleft{}$\Leftarrow$\textquotedblright{} Let $A=A_1\times\cdots\times A_d\subseteq\mathbb{R}^d$ with bounded and open intervals $A_i\subseteq\mathbb{R}$ be an open cuboid. Since the $X_i$'s are independent, we get $\P[\hspace{0.5pt}X\in A\hspace{0.5pt}] = \P[\hspace{0.5pt}X_1\in A_1\hspace{0.5pt}]\cdots\P[\hspace{0.5pt}X_d\in A_d\hspace{0.5pt}]$. Reading now our initial computation backwards until the very first equality, we get
$$
\P[\hspace{0.5pt}X\in A\hspace{0.5pt}]= {\medfrac{1}{(2\pi\sigma^2)^{d/2}}} \int_A\exp\bigl(-\medfrac{\|x\|^2}{2\sigma^2}\bigr)\dd\lambda^d(x)
$$
which extends now to every Borel set in $\mathbb{R}^d$ as the cuboids are a generator.
\end{proof}

As mentioned we need also to show that the distribution of the sum of Gaussian random variables is again Gaussian. We start with the case of two summands.

\begin{lem}\label{SUM-GAUSS-LEM-0} Let $X,\,Y\colon\Omega\rightarrow\mathbb{R}$ be independent real random variables with probability density functions $\rho$ and $\sigma$. Then the random variable $X+Y$ has a probability density function which is given by the convolution $\rho\Conv\sigma\colon\mathbb{R}\rightarrow\mathbb{R}$, $(\rho\Conv\sigma)(s)=\int_{\mathbb{R}}\rho(s-t)\sigma(t)\dd\lambda(t)$.
\end{lem}
\begin{proof} We first show that $\varphi\colon\mathbb{R}^2\rightarrow\mathbb{R}$, $\varphi(x,y):=\rho(x)\sigma(y)$ is a probability density function for the random vector $[\begin{smallmatrix}X\\Y\end{smallmatrix}]\colon\Omega\rightarrow\mathbb{R}^2$. Indeed, for a rectangle $[a,b]\times[c,d]\subseteq\mathbb{R}^2$ with $a<b$ and $c<d$ we compute
\begin{align*}
\P_{\textstyle[\begin{smallmatrix}X\\Y\end{smallmatrix}]}\bigl([a,b]\times[c,d]\bigr) & = \P\bigl[\hspace{0.5pt}(X,Y)\in [a,b]\times[c,d]\hspace{0.5pt}\bigr]\\ 
& = \P\bigl[\hspace{1pt}X\in [a,b] \text{ and } Y\in[c,d]\hspace{0.5pt}\bigr]\\
& = \P\bigl[\hspace{1pt}X\in [a,b]\bigr]\cdot\P\bigl[\hspace{1pt}Y\in[c,d]\hspace{0.5pt}\bigr]\\
& = \int_{[a,b]}\rho\dd\lambda\,\cdot\,\int_{[c,d]}\sigma\dd\lambda\\
& = \int_{[a,b]\times[c,d]}\varphi\dd\lambda^2
\end{align*}
where we used the independence of $X$ and $Y$ in the second equality and Tonelli's Theorem in the last. Since the rectangles generate the Borel $\sigma$-algebra, we get
\begin{equation}\label{EQ-7-0}
\P_{\textstyle[\begin{smallmatrix}X\\Y\end{smallmatrix}]}(A) = \int_A\varphi\dd\lambda^2
\end{equation}
for every Borel set $A$ which establishes the claim. Now we compute for $a<b$
\begin{equation}\label{EQ-7-1}
\begin{aligned}
\P_{X+Y}\bigl((a,b])\bigr) &= \P\bigl[X+Y\in (a,b)\bigr]= \P\bigl(\bigl\{\omega\in\Omega\:;\:X(\omega)+Y(\omega)\in(a,b)\bigr\}\bigr)\\
& = \P\bigl(\bigl\{\omega\in\Omega\:;\:\bigl[\begin{smallmatrix}X(\omega)\\Y(\omega)\end{smallmatrix}\bigr]\in D \bigr\}\bigr) = \P_{\textstyle[\begin{smallmatrix}X\\Y\end{smallmatrix}]}(D) =  \int_D\varphi\dd\lambda^2
\end{aligned}
\end{equation}
where $D = \{(s,t)\in\mathbb{R}^2\:;\:a<s+t<b\}$ is a diagonal strip in the plane as illustrated on the left of following picture and \eqref{EQ-7-0} was used in the last equality. We put $C:=\mathbb{R}\times(a,b)$ as shown on the right of the following picture and see that $\psi\colon C\rightarrow D$, $\psi(s,t)=(s,t-s)$ is a diffeomorphism between open sets and $\det(\operatorname{J}_{\psi}(s,t))=1$ holds for all $(s,t)\in C$. 
\begin{center}
\begin{picture}(375,90)(0,0)
\put(165,30){\begin{tikzpicture}[scale=0.8]
\draw[->] (-1,0) -- (1,0) node[above, midway] {$\psi$};
\draw[->] (1,-0.3) -- (-1,-0.3) node[below, midway] {$\psi^{-1}$};
\end{tikzpicture}}
\put(0,-2.5){\begin{tikzpicture}[scale=0.8]
 
% Help lines
%\draw[lightgray] (-3,-2) grid (3,2);
   
% Axis 
\draw[->] (-1.5,0) -- (3.8,0) node[right] {$s$};
\draw[->] (0,-0.5) -- (0,3.1) node[above] {$t$};

\coordinate (C1) at (-2,0);
\coordinate (C2) at (4.3,0);
  
\coordinate (C3) at (0,-0.5);
\coordinate (C4) at (0,3.1);

\coordinate (A1) at (-0.5,3);
\coordinate (A2) at (3.2,-0.5);
\coordinate (A3) at (-2,3);
\coordinate (A4) at (1.7,-0.5);

\draw(A1)--(A2);
\draw(A3)--(A4);

  \begin{scope}
\fill(A1)--(A2)--(A4)--(A3)--cycle[pattern=north east lines,opacity=.6];
\end{scope}

 \tkzInterLL(A1,A2)(C1,C2)
     \tkzGetPoint{E}
     
\draw (E) -- +(0,0.17); 
     
\node[] at (2.7,0.37) {\small$b$};

 \tkzInterLL(A3,A4)(C1,C2)
     \tkzGetPoint{F}
     
     \draw (F) -- +(0,-0.17); 
     
\node[] at (1.17,-0.37) {\small$a$};

 \tkzInterLL(A1,A2)(C3,C4)
     \tkzGetPoint{E}
     
\draw (E) -- +(0.17,0); 
     
\node[] at (0.37,2.55) {\small$b$};

 \tkzInterLL(A3,A4)(C3,C4)
     \tkzGetPoint{F}
     
     \draw (F) -- +(-0.17,0); 
     
\node[] at (-0.37,1.08) {\small$a$};

\node at (-2.1,1.3) {\small$s\mapsto a\!-\!s$};
\draw (-2,1.5)--(-1.48,2.45);
           
\node at (2.4,2.4) {\small$s\mapsto b\!-\!s$};
\draw (2.35,2.2)--(1.21,1.45);

\node (W) at (-1.28,3.4) {\small$D$};
\draw (-1.2,3.2)--(-0.8,2.5);
\end{tikzpicture}}

\put(230,0){\begin{tikzpicture}[scale=0.8]
 
% Help lines
%\draw[lightgray] (-3,-2) grid (3,2);
   
% Axis 
\draw[->] (-1.5,0) -- (3.8,0) node[right] {$s$};
\draw[->] (0,-0.5) -- (0,3.1) node[above] {$t$};

\coordinate (C1) at (-2,0);
\coordinate (C2) at (4.3,0);
  
\coordinate (C3) at (0,-0.5);
\coordinate (C4) at (0,3.1);

\coordinate (A1) at (-1.5,2.4);
\coordinate (A2) at (3.8,2.4);
\coordinate (A3) at (-1.5,1.08);
\coordinate (A4) at (3.8,1.08);

\draw(A1)--(A2);
\draw(A3)--(A4);

  \begin{scope}
\fill(A1)--(A2)--(A4)--(A3)--cycle[pattern=north east lines,opacity=.6];
\end{scope}

 \tkzInterLL(A1,A2)(C3,C4)
     \tkzGetPoint{E}
     
\draw (E) -- +(0.15,0.15); 
     
\node[] at (0.29,2.7) {\small$b$};

 \tkzInterLL(A3,A4)(C3,C4)
     \tkzGetPoint{F}
     
     \draw (F) -- +(-0.15,-0.15); 
     
\node[] at (-0.29,0.8) {\small$a$};

\node (W) at (4.28,1.6) {\small$C$};
\draw (4.1,1.65)--(3.4,1.9);

\end{tikzpicture}}
\end{picture}
 \end{center}
We can thus apply the change of variables formula and compute
\begin{equation}\label{EQ-7-2}
\begin{aligned}
\int_D\varphi\dd\lambda^2 & = \int_{F}\varphi\circ\psi\dd\lambda^2\\
& = \int_{(a,b)}\int_{\mathbb{R}}\varphi(s,t-s)\dd\lambda(t)\dd\lambda(s)\\
& = \int_{(a,b)}\int_{\mathbb{R}}\rho(s)\sigma(t-s)\dd\lambda(t)\dd\lambda(s)\\
& = \int_{(a,b)}(\rho\Conv\sigma)(s)\dd\lambda(s)
\end{aligned}
\end{equation}
where we used Fubini's Theorem in the second equality. Since the open intervals generate the Borel $\sigma$--algebra, we conclude from \eqref{EQ-7-1} and \eqref{EQ-7-2} that
$$
\P_{X+Y}(A) = \int_{A}\rho\Conv\sigma\dd\lambda
$$
holds for every Borel set as desired.
\end{proof}

In order to extend the above to more than two summands, we need to establish first that the convolution of density functions is commutative and associative.

\begin{lem}\label{CONV-ASS-COMM} Let $\rho,\,\sigma,\,\tau\colon\mathbb{R}\rightarrow\mathbb{R}$ be integrable. Then we have $\rho\Conv\sigma=\sigma\Conv\rho$ and $(\rho\Conv\sigma)\Conv\tau=\rho\Conv(\sigma\Conv\tau)$.
\end{lem}
\begin{proof} By substituting $u:=s-t$ we see that
$$
(\rho\Conv\sigma)(s)=\int_{\mathbb{R}}\rho(s-t)\sigma(t)\dd\lambda(t)=\int_{\mathbb{R}}\rho(u)\sigma(s-u)\dd\lambda(u)=(\sigma\Conv\rho)(s)
$$
holds for every $s\in\mathbb{R}$.

\smallskip

The proof of Lemma \ref{SUM-GAUSS-LEM-0} shows as a byproduct that the convolution of any two density functions is again a density function and in particular integrable. We thus get
\begin{align*}
[(\rho\Conv\sigma)\Conv\tau](s)&=\int_{\mathbb{R}}(\rho\Conv\sigma)(u)\tau(s-u)\dd\lambda(u)\\
& = \int_{\mathbb{R}}\int_{\mathbb{R}}\rho(t)\sigma(u-t)\dd\lambda(t)\tau(s-u)\dd\lambda(u)\\
& = \int_{\mathbb{R}}\rho(t)\int_{\mathbb{R}}\sigma(u-t)\tau(s-u)\dd\lambda(u)\dd\lambda(t)\\
& = \int_{\mathbb{R}}\rho(t)\int_{\mathbb{R}}\sigma(s-v-t)\tau(v)\dd\lambda(v)\dd\lambda(t)\\
& = \int_{\mathbb{R}}\rho(t)(\sigma\Conv\tau)(s-t)\dd\lambda(t)\\
&=[\rho\Conv(\sigma\Conv\tau)](s)
\end{align*}
where we substituted $v:=s-u$ in the fourth equality.
\end{proof}

In view of Lemma \ref{CONV-ASS-COMM} we can drop parentheses and simply write $\rho_1\Conv\cdots\Conv\rho_n$ for the convolution of $n$ density functions. Employing this, we can now extend Lemma \ref{SUM-GAUSS-LEM-0} to the sum of more than two random variables.

\begin{prop}\label{SUM-GAUSS-LEM} For $i=1,\dots,d$ let $X_i\colon\Omega\rightarrow\mathbb{R}$ be independent Gaussian random variables with zero mean and unit variance. Then $X:=X_1+\cdots+X_d$ is a Gaussian random variable with mean zero and variance $d$.
\end{prop}
\begin{proof} Since the sum of measurable functions is measurable, $X=X_1+\cdots+X_d$ is a random variable. We compute
$$
\E(X) = \E(X_1)+\cdots+\E(X_d)=0 \; \text{ and } \; \V(X) = \V(X_1)+\cdots+\V(X_d) = d
$$
where the second equality holds as the $X_i$ are independent. We now denote by $\rho = (2\pi)^{-1/2}\exp(-x/2)$ the Gaussian density with mean zero and variance one. Then we need to show that
\begin{equation}\label{EQ-GAUSS}
(\rho\Conv\cdots\Conv\rho)(x) = \medfrac{1}{\sqrt{2\pi d}}\exp\bigl(-\medfrac{x}{2d}\bigr)
\end{equation}
holds where on the left hand side $d$--many $\rho$'s appear. We will now show \eqref{EQ-GAUSS} for $d=2$. The extension to general $d$ is then immediate by iteration. We compute for arbitrary $s\in\mathbb{R}$
\begin{equation}\label{CONV-COMP}
\begin{aligned}
(\rho\Conv\rho)(s) & =\frac{1}{2\pi}\int_{\mathbb{R}}\exp\bigl(-(s-t)^2/2\bigr)\exp\bigl(-t^2/2\bigr)\dd\lambda(t)   \\
& = \smallfrac{1}{2\pi}\int_{\mathbb{R}}\exp\bigl(-(t^2-2\,t\,(s/2)+(s/2)^2)-s^2/4\bigr)\dd\lambda(t)\\
& = \smallfrac{1}{2\pi}\exp(-s^2/4)\int_{\mathbb{R}}\exp\bigl(-(t-s/2)^2)\bigr)\dd\lambda(t)\\
& = \smallfrac{1}{2\pi}\exp(-s^2/4)\int_{\mathbb{R}}\exp(-u^2)\dd\lambda(u)\\
& = \smallfrac{1}{2\pi}\exp(-s^2/4)\sqrt{\pi}\\
& = \smallfrac{1}{\sqrt{4\pi}}\exp(-s^2/4)
\end{aligned}
\end{equation}
where we used the substitution $u:=t-s/2$.
\end{proof}

Let us now show what we later really need, namely the case of a linear combination of Gaussian random variables with zero mean and unit variance.

\begin{prop}\label{LK-GAUSS-LEM} For $i=1,\dots,d$ let $X_i\sim\mathcal{N}(0,1)$ be independent Gaussian random variables and let $\lambda_i\not=0$ be real number. Then $X:=\lambda_1X_1+\cdots+\lambda_dX_d\sim\mathcal{N}(0,\sigma^2)$ with $\sigma^2:=\lambda_1^2+\cdots+\lambda_d^2\not=0$.
\end{prop}

\begin{proof} We observe that for $\lambda\not=0$, $X\sim\mathcal{N}(0,1)$ and a Borel set $A$ the following holds
$$
\P\bigl[\lambda X\in A\bigr]=\P\bigl[X\in \textfrac{1}{\lambda}A\bigr]=\int_{\frac{1}{\lambda}A}\rho(x)\dd x = \int_{A}\rho(\textfrac{1}{\lambda})\textfrac{1}{\lambda}\dd x
$$
where $\rho$ is the Gaussian density function. This means that the random variable $\lambda X$ is given by the density function $\rho(\frac{\cdot}{\lambda})\frac{1}{\lambda}$. Subbing in the Gaussian density, the latter becomes
$$
\rho(\medfrac{t}{\lambda})\medfrac{1}{\lambda} = \medfrac{1}{\sqrt{2\pi}}\exp\bigl(-(\medfrac{t}{2})^2/2\bigr)\medfrac{1}{\lambda} = \medfrac{1}{\sqrt{2\pi}\lambda}\exp\bigl(-(\medfrac{t^2}{2\lambda^2})\bigr)
$$
for $t\in\mathbb{R}$, which shows that $\lambda X\sim\mathcal{N}(0,\lambda^2)$. We know now that $\lambda_iX_i\sim\mathcal{N}(0,\lambda_i^2)$ holds for $i=1,\dots,d$ and need to establish that their sum's density function is given by
$$
(\rho_1\Conv\cdots\Conv\rho_d)(x) = (2\pi\sigma^2)^{-1/2}\exp(-x/(2\sigma^2))
$$
where 
$$
\rho_i(x)=(2\pi\sigma_i^2)^{-1/2}\exp(-x/(2\sigma_i^2)).
$$
This can be achieved similarly to \eqref{CONV-COMP}. Indeed, let us abbreviate $a:=\sigma_1^2$ and $b:=\sigma_2^2$ and $c:=a+b$, i.e., $(b+a)/c=1$. Then we get
\begin{align*}
(\rho_1\Conv{}\rho_2)(s) & =\medfrac{1}{2\pi\sqrt{ab}}\int_{\mathbb{R}}\exp\bigl(-\medfrac{(s-t)^2}{2a}\bigr)\exp\bigl(-\medfrac{t^2}{2b}\bigr)\dd\lambda(t)\\
& =\medfrac{1}{2\pi\sqrt{ab}}\int_{\mathbb{R}}\exp\bigl(-\medfrac{b(s^2-2st+t^2)+at^2}{2ab}\bigr)\dd\lambda(t)\\
& =\medfrac{1}{2\pi\sqrt{ab}}\int_{\mathbb{R}}\exp\bigl(-\medfrac{t^2(b+a)-2stb+bs^2}{2ab}\bigr)\dd\lambda(t)\\
& =\medfrac{1}{2\pi c^2\sqrt{ab/c}}\int_{\mathbb{R}}\exp\bigl(-\medfrac{t^2(b+a)/c-2stb/c+bs^2/c}{2ab/c}\bigr)\dd\lambda(t)\\
& =\medfrac{1}{2\pi c^2\sqrt{ab/c}}\int_{\mathbb{R}}\exp\bigl(-\medfrac{(t-(bs)/c)^2 - (sb/c)^2 +s^2(b/c)}{2ab/c}\bigr)\dd\lambda(t)\\
& =\medfrac{1}{\sqrt{2\pi c}}\exp(- \medfrac{(sb/c)^2 -s^2(b/c)}{2ab/c})\medfrac{1}{\sqrt{2\pi(ab/c)}}\int_{\mathbb{R}}\exp\bigl(-\medfrac{(t-(bs)/c)^2}{2ab/c}\bigr)\dd\lambda(t)\\
& =\medfrac{1}{\sqrt{2\pi c}}\exp(- \medfrac{(sb/c)^2c^2 -s^2(b/c)c^2}{2abc})\\
& =\medfrac{1}{\sqrt{2\pi c}}\exp(- \medfrac{s^2(b^2 -bc)}{2abc})\\
& =\medfrac{1}{\sqrt{2\pi c}}\exp(- \medfrac{s^2}{2c})
\end{align*}
where the integral can be seen to be one by the substitution $u:=t-(bs)/c$. Iteration of this completes the proof.
\end{proof}

We mention that Proposition \ref{LK-GAUSS-LEM} extends further to the case of a linear combination of arbitrary Gaussian random variables which not necessarily have mean zero or unit variance. The proof then becomes again one step more complicated than the above, as the means $\mu_i$ appear inside the exponential functions. Since for our purposes below the result as stated above will be sufficient, we leave this last generalization to the reader as Problem \ref{SUM-GAUSS-PROB}.

\smallskip

Our final preparation is the following formula for the integral of the Gaussian function against a polynomial which we will need, in order to estimate moments in the next proof.

\begin{lem}\label{LEM-10-1} For $k\in\mathbb{N}$ the equality
$$
\medfrac{1}{\sqrt{2\pi}}\int_{\mathbb{R}}t^{2k}\exp(-t^2/2)\dd t = \medfrac{(2k)!}{2^kk!}
$$
holds.
\end{lem}
\begin{proof} We first use Lebesgue's Theorem to see that
\begin{align*}
\medfrac{\operatorname{d}^k}{\dd a^k}\Bigl(\int_{\mathbb{R}}\exp(-at^2)\dd t\Bigr)\Big|_{a=\frac{1}{2}} &= \Bigl(\int_{\mathbb{R}}\medfrac{\operatorname{d}^k}{\dd a^k}\exp(-at^2)\dd t\Bigr)\Big|_{a=\frac{1}{2}}\\
& = \int_{\mathbb{R}}(-t^2)^k\exp(-at^2)\dd t\Big|_{a=\frac{1}{2}}\\
& = (-1)^k\int_{\mathbb{R}}t^{2k}\exp(-t^2/2)\dd t
\end{align*}
holds where $a>0$ and $k\in\mathbb{N}$ is fixed. On the other hand substituting $u:=\sqrt{a}t$ leads to
$$
\int_{\mathbb{R}}\exp(-at^2)\dd t = \medfrac{1}{\sqrt{a}}\int_{\mathbb{R}}\exp(-u^2)\dd u = \sqrt{\medfrac{\pi}{a}}
$$
now for every $a>0$. Combining both equations we obtain
$$
\int_{\mathbb{R}}t^{2k}\exp(-t^2/2)\dd t = (-1)^k\medfrac{\operatorname{d}^k}{\dd a^k}\Bigl(\int_{\mathbb{R}}\exp(-at^2)\dd t\Bigr)\Big|_{a=\frac{1}{2}}= (-1)^k\medfrac{\operatorname{d}^k}{\dd a^k}\Bigl(\medfrac{\sqrt{\pi}}{\sqrt{a}}\,\Bigr)\Big|_{a=\frac{1}{2}}
$$
and thus need to compute the derivative on the right hand side. For that purpose we compute
\begin{align*}
\smallfrac{\operatorname{d}^k}{\dd a^k}(a^{-1/2}) & = (-1/2)\cdot(-1/2-1)\cdots(-1/2-(k-1))\,a^{-1/2-k}\\
& = \smallfrac{(-1)^k}{2^k}(1\cdot 3\cdots(2k-3)\cdot (2k-1))\,a^{-1/2-k}\\
& = \smallfrac{(-1)^k}{2^k}\smallfrac{1\cdot2\cdot 3\cdots(2k-3)\cdot (2k-2)\cdot(2k-1)\cdot(2k)}{2\cdot4\cdots(2k-2)\cdot(2k)}\,a^{-1/2-k}\\
& = \smallfrac{(-1)^k}{2^k}\smallfrac{(2k)!}{2^k k!}\,a^{-1/2-k}.
\end{align*}
Putting everything together we arrive at
\begin{align*}
\smallfrac{1}{\sqrt{2\pi}}\int_{\mathbb{R}}t^{2k}\exp(-t^2/2)\dd t & = \smallfrac{1}{\sqrt{2\pi}}(-1)^k\smallfrac{\operatorname{d}^k}{\dd a^k}\Bigl(\smallfrac{\sqrt{\pi}}{\sqrt{a}}\,\Bigr)\Big|_{a=\frac{1}{2}}\\
& = \smallfrac{(-1)^k}{\sqrt{2}}\smallfrac{(-1)^k}{2^k}\smallfrac{(2k)!}{2^k k!}\,a^{-1/2-k}\Big|_{a=\frac{1}{2}}\\
& = \smallfrac{1}{\sqrt{2}}\,\smallfrac{1}{2^k}\,\smallfrac{(2k)!}{2^k k!}2^{1/2+k}
\end{align*}
which after cancelation reduces precisely to the expression on the right hand side in the Lemma.
\end{proof}

Now we can prove the first main result of this chapter. Notice that the inequality is only non-trivial, i.e., the right hand side strictly less than one, if $\epsilon>4\ln2\,(\approx2.8)$ holds. This means, although we are using the variable $\epsilon$ here, the typical intuition that we can let $\epsilon\rightarrow0$ does not apply.

\begin{thm}\label{GA-THM}(Gaussian Annulus Theorem) Let $x\in\mathbb{R}^d$ be drawn at random with respect to the spherical Gaussian distribution with zero mean and unit variance. Then
$$
\P\bigl[\hspace{1pt}\big|\|x\|-\sqrt{d}\hspace{1pt}\big|\geqslant\epsilon\bigr]\leqslant 2\exp(-c\hspace{1pt}\epsilon^2).
$$
holds for every $0\leqslant\epsilon\leqslant\sqrt{d}$ where $c=1/16$.
\end{thm}
\begin{proof} Let a random vector $X\sim\mathcal{N}(0,1)$ be given and let $X_1,\dots,X_d$ denote its coordinate functions. By Proposition \ref{PROP-9-1} we have $X_i\sim\mathcal{N}(0,1)$ for every $i=1,\dots,d$. We define $Y_i:=(X_i^2-1)/2$ and see that
\begin{equation}\label{EQ-10-1}
\begin{aligned}
\P\bigl[\hspace{1pt}\big|\|X\|-\sqrt{d}\hspace{1pt}\big|\geqslant\epsilon\bigr] &\leqslant \P\bigl[\hspace{1pt}\big|\|X\|-\sqrt{d}\hspace{1pt}\big|\cdot\bigl(\|X\|+\sqrt{d}\hspace{1pt}\bigr)>\epsilon\cdot\sqrt{d}\hspace{1.5pt}\bigr]\\
&= \P\bigl[\hspace{1pt}\bigl|X_1^2+\cdots+X_d^2-d\hspace{1pt}\bigr|>\epsilon\sqrt{d}\hspace{1.5pt}\bigr]\\
&= \P\bigl[\hspace{1pt}\bigl|(X_1^2-1)+\cdots+(X_d^2-1)\hspace{1pt}\bigr|>\epsilon\sqrt{d}\hspace{1.5pt}\bigr]\\
&= \P\bigl[\hspace{1pt}\bigl|{\textstyle\medfrac{X_1^2-1}{2}}+\cdots+{\textstyle\frac{X_d^2-1}{2}}\hspace{1pt}\bigr|>{\textstyle\frac{\epsilon\sqrt{d}}{2}}\hspace{1.5pt}\bigr]\\
&= \P\bigl[\hspace{1pt}\bigl|Y_1+\cdots+Y_d\hspace{1pt}\bigr|>{\textstyle\frac{\epsilon\sqrt{d}}{2}}\hspace{1.5pt}\bigr]
\end{aligned}
\end{equation}
where we used $\|X\|+\sqrt{d}\geqslant\sqrt{d}$ for the inequality. Since we have $\E(X_i)=0$ and $\V(X_i)=1$ we get
$$
\E(Y_i)=\E\bigl(\medfrac{X_i^2-1}{2}\bigr) = {\textstyle\frac{1}{2}}\bigl(\E(X_i^2)-1\bigr) = 0
$$
for every $i=1,\dots,d$. We now estimate for $k\geqslant2$ the $k$-th moment of $Y_i$. For this we firstly note that
$$
|X_i^2(\omega)-1|^k\leqslant X_i^{2k}(\omega)+1
$$
holds for every $\omega\in\Omega$. Indeed, if $|X_i(\omega)|\leqslant1$, then $0\leqslant X_i^2(\omega)\leqslant1$ and thus $|X_i^2(\omega)-1|^k=(1-X_i^2(\omega))^k\leqslant1$. If otherwise $|X_i(\omega)|>1$, then $X_i^2(\omega)-1>0$ and therefore $|X_i^2(\omega)-1|^k=(X_i^2(\omega)-1)^k\leqslant X_i^{2k}(\omega)$. We employ the above and estimate
\begin{align*}
|\E(Y_i^k)| & = \bigl|\E\bigl(\bigl(\medfrac{X_i^2-1}{2}\bigr)^k\bigr)\bigr| = \smallfrac{1}{2^k}\bigl|\E\bigl((X_i^2-1)^k\bigr)\bigr|\\
&\leqslant \smallfrac{1}{2^k}\int_{\Omega}|X_i^2-1|^k\dd\P \leqslant \smallfrac{1}{2^k}\int_{\Omega}X_i^{2k}+1\dd\P = \smallfrac{1}{2^k}\bigl(\E(X_i^{2k})+1\bigr)\\
& = \smallfrac{1}{2^k}\Bigl(\medfrac{1}{\sqrt{2\pi}}\int_{\mathbb{R}}t^{2k}\exp(-t^2/2)\dd t +1\Bigr) = \smallfrac{1}{2^k}\Bigl(\smallfrac{(2k)!}{2^kk!}+1\Bigr)\\
& = \smallfrac{(2k)\cdot(2k-1)\cdot(2k-2)\cdot(2k-3)\cdots5\cdot4\cdot3\cdot2\cdot1}{(2k)^2\cdot(2k-2)^2\cdots(2\cdot3)^2\cdot(2\cdot2)^2\cdot(2\cdot1)^2}k!+\smallfrac{1}{2^k}\\
& = \smallfrac{(2k-1)\cdot(2k-3)\cdots5\cdot3\cdot1}{(2k)\cdot(2k-2)\cdots6\cdot4\cdot2}k!+\smallfrac{1}{2^k}\\
& \leqslant 1\cdot 1\cdots 1\cdot \smallfrac{3}{4}\cdot\smallfrac{1}{2}\cdot k! +\smallfrac{k!}{4\cdot 2}\\
&= \bigl(\smallfrac{3}{4}+\smallfrac{1}{4}\bigr)\smallfrac{k!}{2}
\\
&=\smallfrac{k!}{2}
\end{align*}
where we used  $X_i\sim\mathcal{N}(0,1)$ to compute $\E(X_i^{2k})$ and Lemma \ref{LEM-10-1} to evaluate the integral. We thus established that the $Y_i$ satisfy the assumptions of Theorem \ref{MTB-THM} and we can thus make use of Bernstein's inequality
$$
\P\bigl[\hspace{1pt}|Y_1+\cdots+Y_d|\geqslant a\hspace{0.5pt}\bigr]\leqslant 2\exp\bigl(-C\min\bigl(\smallfrac{a^2}{d},a\bigr)\bigr)
$$
with $a:={\textstyle\frac{\epsilon\sqrt{d}}{2}}>0$ to continue our estimate from \eqref{EQ-10-1}. Since $\sqrt{d}\geqslant\epsilon$ holds, we get
\begin{align*}
\P\bigl[\hspace{1pt}\big|\|x\|-\sqrt{d}\hspace{1pt}\big|\geqslant\epsilon\hspace{0.5pt}\bigr]& \leqslant \P\bigl[\hspace{1pt}\bigl|Y_1+\cdots+Y_d\hspace{1pt}\bigr|\geqslant\medfrac{\epsilon\sqrt{d}}{2}\hspace{1pt}\bigr] \\
&\leqslant 2\exp\bigl(-C\min\bigl(\medfrac{(\epsilon\sqrt{d}/2)^2}{d},\medfrac{\epsilon\sqrt{d}}{2}\bigr)\bigr)\\
&\leqslant 2\exp\bigl(-C\min\bigl(\medfrac{\epsilon^2}{4},\medfrac{\epsilon^2}{2}\bigr)\bigr)\\
&= 2\exp\bigl(-C\medfrac{\epsilon^2}{4}\bigr).
\end{align*}
Finally, we recall that the constant in Theorem \ref{MTB-THM} was $C=1/4$ and thus the exponent equals $-c\hspace{1pt}\epsilon^2$ with $c=1/16$.
\end{proof}

In Chapter \ref{Ch-INTRO} we propagated the intuition that a point chosen at random will have norm close to $\sqrt{d}$ with `a high probability.' More precisely, Figure \ref{FIG-1} suggested that the deviation from $\sqrt{d}$ to be expected will not increase with $d$. Theorem \ref{GA-THM} allows to quantify this intuition as follows. If we pick, e.g., $\epsilon=10$ then a random point's norm will satisfy $\sqrt{d}-10\leqslant \|x\|\leqslant \sqrt{d}+10$ with probability greater or equal to $0.99$ whenever we are in a space with dimension $d\geqslant100$.

\smallskip

We note the following byproduct the proof of Theorem \ref{GA-THM} which gives a similar estimate but of the square of the norm. Notice that one of our first results in Chapter \ref{Ch-INTRO} was the equality $\E(\|X\|^2)=d$ for any fixed $d$\hspace{1pt}---\hspace{1pt}and that we on the other hand proved a statement without the squares, namely $\E(\|X\|)\approx\sqrt{d}$, asymptotically.

\begin{prop}\label{GA-COR} Let $x\in\mathbb{R}^d$ be drawn at random with respect to the spherical Gaussian distribution with zero mean and unit variance. Then
$$
\P\bigl[\hspace{1pt}\big|\|x\|^2-d\hspace{1pt}\big|\geqslant\epsilon\bigr]\leqslant 2\exp\bigl(-c\hspace{1pt}\min\bigl(\medfrac{\epsilon^2}{2d},\epsilon\bigr)\bigr).
$$
holds for every $\epsilon>0$ where $c=1/8$.
\end{prop}
\begin{proof}We define the random variables $Y_i$ as in the proof of Theorem \ref{GA-THM} and then proceed using the Bernstein inequality as in the last part of the previous proof. This way we obtain
$$
\P\bigl[\hspace{1pt}\bigl|\|X\|^2-d\bigr|\geqslant\epsilon\hspace{1pt}\bigr] = \P\bigl[\hspace{1pt}\bigl|Y_1+\cdots+Y_d\bigr|\geqslant\medfrac{\epsilon}{2}\hspace{1pt}\bigr]\leqslant 2\exp\bigl(-C\min\bigl(\smallfrac{\epsilon^2}{4d},\medfrac{\epsilon}{2}\bigr)\bigr)
$$
which leads to the claimed inequality taking $C=1/4$ into account.
\end{proof}

\smallskip

In the next two theorems we divide by $\|x\|$ and $\|y\|$ where $x$ and $y$ are picked at random. Notice however that $x=0$ or $y=0$ occurs only with probability zero. To be completely precise, we could use the conditional probability $\P[|\langle{}x/\|x\|,y/\|y\|\rangle{}|\leqslant\epsilon\,|\,x\not=0\text{ and }y\not=0]$. We will instead tacitly assume this. Observe that the bound below is non-trivial if $\epsilon>2/(\sqrt{d}-7)$. In this result it thus possible to think of $\epsilon\rightarrow0$, \emph{but only if simultaneously $d\rightarrow\infty$ suitably fast}.

\smallskip

\begin{thm}\label{GO-THM} Let $x,y\in\mathbb{R}^d$ be drawn at random with respect to the spherical Gaussian distribution with zero mean and unit variance. Then for every $\epsilon>0$ and for all $d\geqslant1$ the estimate
$$
\P\bigl[\hspace{1pt}\big|\bigl\langle{}\medfrac{x}{\|x\|},\medfrac{y}{\|y\|}\bigr\rangle{}\big|\geqslant\epsilon\hspace{1pt}\bigr]\leqslant \medfrac{2/\epsilon+7}{\sqrt{d}}
$$
holds.
\end{thm}
\begin{proof}Let $y\in\mathbb{R}^d\backslash\{0\}$ be fixed and let $X\colon\Omega\rightarrow\mathbb{R}^d$ be a random vector with $X\sim\mathcal{N}(0,1)$ whose coordinate functions we denote by $X_i$ for $i=1,\dots d$. We define
$$
U_y\colon\Omega\rightarrow\mathbb{R}^d,\;U_y(\omega)=\bigl\langle{}X(\omega),\medfrac{y}{\|y\|}\bigr\rangle{}=\sum_{i=1}^d\medfrac{y_i}{\|y\|}X_i(\omega)
$$
which is itself a random variable which has by Proposition \ref{LK-GAUSS-LEM} Gaussian distribution with $\mu=0$ and
$$
\sigma^2=\sum_{i=1}^d\medfrac{y_i^2}{\|y\|^2}= \medfrac{1}{\|y\|^2}\sum_{i=1}^d y_i^2 = 1.
$$
This shows that $U_y\sim\mathcal{N}(0,1)$ is normally distributed. In particular, we have
$$
\P\bigl[\hspace{0.5pt}|U_y|\leqslant \delta\hspace{0.5pt}\bigr] = \medfrac{1}{\sqrt{2\pi}}\int_{-\delta}^{\delta}\exp(-t^2/2)\dd t
$$
for any $\delta>0$. Observe that the right hand side is independent of $y$. Therefore, if $Y\colon\Omega\rightarrow\mathbb{R}^d$ is any other random vector, we can consider the composite random vector $U_Y(X)\colon\Omega\rightarrow\mathbb{R}^d$ and get for any $\delta>0$
\begin{align*}
\P\bigl[\hspace{1pt}|U_Y(X)|\leqslant \delta\hspace{1pt}\bigr] & = \P\bigl(\{\omega\in\Omega\:;\:(X(\omega),Y(\omega))\in A\}\bigr)= \int_A\rho(x)\rho(y)d(x,y)\\
&  =  \int_{\mathbb{R}}\rho(x)\int_{A_y}\rho(x)\dd x \dd y = \int_{\mathbb{R}}\rho(x)\P(\{\omega\in\Omega\:;\:X(\omega)\in A_y\}) \dd y  \\
& = \int_{\mathbb{R}}\rho(t)\underbrace{\P(\{\omega\in\Omega\:;\:|U_y(X(\omega))|\leqslant\delta\})}_{\text{independent of $y$}} \dd y = 1\cdot\P\bigl[|U_y(X)|\leqslant\delta\bigr]  \\
&= \medfrac{1}{\sqrt{2\pi}}\int_{-\delta}^{\delta}\exp(-t^2/2)\dd t  \geqslant 1-\medfrac{2}{\sqrt{2\pi}}\int_{\delta}^{\infty}1/t^2\dd t = 1-\medfrac{2}{\sqrt{2\pi}\hspace{1pt}\delta}
\end{align*}
where $\rho$ is the Gaussian density function and $A=\{(x,y)\in\mathbb{R}^2\:;\:|U_y(x)|\leqslant\delta\}$. 
Let now $X$ and $Y\colon\Omega\rightarrow\mathbb{R}^d$ both be normally distributed random vectors and let $\epsilon>0$ be given. We first use the Theorem of Total Probability and drop the second summand. Then we rewrite the first term by using the random variable $U_Y(X)$ from above and the second term such that the Gaussian Annulus Theorem becomes applicable. We get
\begin{align*}
\P\bigl[\hspace{0.5pt}\big|\bigl\langle{}\textfrac{X}{\|X\|},\textfrac{Y}{\|Y\|}\bigr\rangle{}\big|\leqslant\epsilon\hspace{0.5pt}\bigr] & \geqslant \P\bigl[\hspace{0.5pt}\big|\bigl\langle{}\textfrac{X}{\|X\|},\textfrac{Y}{\|Y\|}\bigr\rangle{}\big|\leqslant\epsilon\;\big|\;\|X\|\geqslant\textfrac{\sqrt{d}}{2}\hspace{1pt}\bigr]\cdot\P\bigl[\|X\|\geqslant\textfrac{\sqrt{d}}{2}\hspace{1pt}\bigr] \\
&= \P\bigl[\big|\bigl\langle{}X,\textfrac{Y}{\|Y\|}\rangle{}|\leqslant\epsilon\,\|X\|\;\big|\;\|X\|\geqslant\textfrac{\sqrt{d}}{2}\hspace{1pt}\bigr]\cdot\P\bigl[\|X\|\geqslant\textfrac{\sqrt{d}}{2}\hspace{1pt}\bigr] \\
&\geqslant \P\bigl[|U_Y(X)|\leqslant\textfrac{\epsilon\sqrt{d}}{2}\hspace{1pt}\bigr]\cdot\P\bigl[|\|X\|-\sqrt{d}|\geqslant\textfrac{\sqrt{d}}{2}\hspace{1pt}\bigr] \\
&\geqslant   \Bigl(1-\medfrac{2}{\sqrt{2\pi}\hspace{1pt}(\epsilon\sqrt{d}/2)}\Bigr)\Bigl(1-2\exp\bigl(-c(\sqrt{d}/2)^2\bigr)\Bigr) \\
&=  \Bigl(1-\medfrac{4}{\sqrt{2\pi d}\hspace{1.5pt}\epsilon}\Bigr)\Bigl(1-2\exp\bigl(-\medfrac{cd}{4}\bigr)\Bigr) \\
&\geqslant  1-  \medfrac{4}{\sqrt{2\pi}\hspace{1.5pt}\epsilon}\medfrac{1}{\sqrt{d}}   -2\exp\bigl(-\medfrac{d}{64}  \bigr)
\end{align*}
since $c=1/16$. We see now already that the second summand is the one that goes slower to zero than the third. To get rid of the exponential term we thus only need to find the right constant. We claim that
$$
2\exp\bigl(-\medfrac{x}{64}\bigr)\leqslant \medfrac{7}{\sqrt{x}}
$$
holds for every $d\geqslant1$. Indeed, the above is equivalent to $f(d):=\frac{49}{4}\exp\bigl(\frac{d}{32}\bigr)-d\geqslant0$. Considering $f\colon(0,\infty)\rightarrow\mathbb{R}$ we get $f'(d)=0$ if and only if $d=32\ln(\frac{128}{49})$ and at this point needs to be the global minimum of $f$. Using a calculator one can see that $f(32\ln(\frac{128}{49}))>0$. Consequently, we get
$$
\P\bigl[\hspace{1pt}\big|\bigl\langle{}\textfrac{X}{\|X\|},\textfrac{Y}{\|Y\|}\bigr\rangle{}\big|\geqslant\epsilon\hspace{1pt}\bigr] \leqslant \medfrac{4}{\sqrt{2\pi}\hspace{1.5pt}\epsilon}\medfrac{1}{\sqrt{d}}   +\medfrac{7}{\sqrt{d}} \leqslant \bigl(\medfrac{2}{\epsilon}+7\bigr) \medfrac{1}{\sqrt{d}}
$$
which concludes the proof.
\end{proof}

Theorem \ref{GO-THM} now also quantifies our intuition that two points chosen at random will have scalar product close to zero and thus will be `almost orthogonal' with `a high probability.' If we pick for instance $\epsilon = 0.1$ and $d\geqslant 100\,000$, then by Theorem \ref{GO-THM} we get that two points drawn at random will, with probability greater or equal to $0.9$, have (normalized) scalar product less or equal to $0.1$. This corresponds to an angle of $90^{\circ}\pm 6^{\circ}$.
%\footnote{\red{Is there a result that extends to $k$-many points, similar to Theorem \ref{UO-THM} in the unit ball case?}}

\smallskip

The proof of the previous result shows also the following.

\begin{cor}\label{COR-ANGLE-ESTIM} Let $X\sim\mathcal{N}(0,1,\mathbb{R}^d)$ and $0\not=\xi\in\mathbb{R}^d$ be fixed. Then we have
$$
\P\big[\hspace{1pt}|\langle{}X,\xi\rangle{}|\geqslant\epsilon\hspace{1pt}\bigr]\leqslant \medfrac{4}{\sqrt{2\pi}}\medfrac{\|\xi\|}{\epsilon\sqrt{d}}
$$
for every $\epsilon>0$.
\end{cor}
\begin{proof} In the notation of the proof of Theorem \ref{GO-THM} we calculate
\begin{align*}
\P\big[\hspace{1pt}\bigl|\langle{}X,\xi\rangle{}\bigr|\leqslant\epsilon\hspace{1pt}\bigr] & = \P\big[\hspace{1pt}\bigl|\langle{}X,\medfrac{\xi}{\|\xi\|}\rangle{}\bigr|\leqslant\medfrac{\epsilon}{\|\xi\|}\bigr] = \P\big[\hspace{1pt}\bigl|U_{\xi}(X)\bigr|\leqslant\medfrac{\epsilon}{\|\xi\|}\bigr] \geqslant 1-\medfrac{4}{\sqrt{2\pi}}\medfrac{\|\xi\|}{\epsilon\sqrt{d}}
\end{align*}
which establishes the result.
\end{proof}

\smallskip

Since the probability that two random points are \emph{exactly} orthogonal is obviously zero, we have on the other hand to expect that the probability of being \emph{very} close to zero is again small. The estimate below is non-trivial if $\epsilon<1$ and $d>16\ln(\frac{2}{1-\epsilon})$. In particular we can now let $\epsilon$ tend to zero \emph{for a fixed $d$} if we wish so. However, the threshold on the left hand side tends to zero for $d\rightarrow\infty$ even if we fix $0<\epsilon<1$ since we divide by $\sqrt{d}$ inside the square bracket.

\begin{thm}\label{GO-THM-2}Let $x$, $y\in\mathbb{R}^d$ be drawn at random with respect to the spherical Gaussian distribution with zero mean and unit variance. Then for $\epsilon>0$ and $d\geqslant1$ we have
$$
\P\bigl[\bigl|\bigl\langle{}\medfrac{x}{\|x\|},\medfrac{y}{\|y\|}\bigr\rangle\bigl|\leqslant\medfrac{\epsilon}{2\sqrt{d}}\,\bigr]\leqslant \epsilon+2\exp(-cd)
$$
where $c=1/16$.
\end{thm}
\begin{proof} Let $X,Y\colon\Omega\rightarrow\mathbb{R}^d$ be random vectors with normal distribution. Let $U_Y(X)$ be defined as in the proof of Theorem \ref{GO-THM}. We compute
\begin{align*}
\P\bigl[\hspace{1pt}\bigl|\bigl\langle{}\medfrac{X}{\|X\|},\medfrac{Y}{\|Y\|}\bigr\rangle\bigl|\hspace{2pt}\hspace{2pt}\geqslant\medfrac{\epsilon}{2\sqrt{d}}\,\bigr] & = 1-\P\bigl[\hspace{1pt}\bigl|\bigl\langle{}{\textstyle\frac{X}{\|X\|},\frac{Y}{\|Y\|}}\bigr\rangle\bigl|\hspace{2pt}<\hspace{1pt}\medfrac{\epsilon}{2\sqrt{d}}\hspace{1.5pt}\bigr]\\
& = 1-\P\bigl[\medfrac{1}{\|X\|}\cdot\bigl|U_Y(X)\bigl|<\medfrac{1}{2\sqrt{d}}\cdot\epsilon\hspace{1pt}\bigr]\\
& \geqslant 1-\P\bigl[\medfrac{1}{\|X\|}<\medfrac{1}{2\sqrt{d}} \:\text{ or }\: \bigl|U_Y(X)\bigl|<\epsilon\hspace{1pt}\bigr]\\
& \geqslant 1-\Bigl(\P\bigl[\hspace{1pt}\|X\|>2\sqrt{d}\hspace{1.5pt}\bigr] +\P\bigl[\hspace{1pt}\bigl|U_Y(X)\bigl|<\epsilon\hspace{1.5pt}\bigr]\Bigr)\\
& = \P\bigl[\hspace{1pt}\|X\|\leqslant 2\sqrt{d}\hspace{2pt}\bigr] -\P\bigl[\hspace{1pt}\bigl|U_Y(X)\bigl|<\epsilon\hspace{1pt}\bigr]\\
& = \P\bigl[\hspace{1pt}\|X\|-\sqrt{d} \leqslant \sqrt{d}\hspace{2pt}\bigr] -\medfrac{1}{\sqrt{2\pi}}\int_{-\epsilon}^{\epsilon}\exp(-t^2/2)\dd t\\
& \geqslant \P\bigl[\bigl|\|X\|-\sqrt{d}\hspace{1.5pt}\bigr| \leqslant \sqrt{d}\hspace{2pt}\bigr] -\medfrac{1}{\sqrt{2\pi}}2\epsilon\\
& \geqslant 1-2\exp\Bigl(-\medfrac{1}{16}(\sqrt{d})^2\Bigr) -\sqrt{2/\pi}\,\epsilon\\
& \geqslant 1-2\exp\bigl(\medfrac{d}{16}\bigr) -\epsilon
\end{align*}
where we used Theorem \ref{GA-THM} and the formula for $\P\bigl[\hspace{1pt}|U_Y(X)|\leqslant \delta\hspace{1pt}\bigr]$ that we established in the proof of Theorem \ref{GO-THM}.
\end{proof}

If we now fix the dimension such that the exponential term is already quite small (e.g., $d=100$, then $2\exp(-cd)<0.01$) and pick, e.g., $\epsilon=0.1$, then the result shows that the (normalized) scalar product of two random points is less than $0.005$ only with a probability less than $0.11$.

\smallskip

Combining the last two results we can think of the scalar products for a fixed high dimension $d$, with high probability, to be small numbers which are however bounded away from zero.

\smallskip

Our last result in this section quantifies our finding from Chapter \ref{Ch-INTRO} that for two points $x$ and $y$ drawn at random from a unit Gaussian we have $\|x-y\|\approx\sqrt{2d}$. Using $\epsilon=18$ in (ii) below shows that the values possible for the dimension are $d>324$. In this case the estimate is however still trivial. Interesting cases appear only if $\epsilon>18$ and $d$ so large that the sum on the right hand side of (ii) stays below one.

\begin{thm}\label{THM-distance-of-two-points-sqrt-2d} Let $x$, $y\in\mathbb{R}^d$ be drawn at random with respect to a spherical Gaussian distribution with zero mean and unit variance. Then we have the following.\vspace{3pt}
\begin{compactitem}
\item[(i)] $\forall\:\epsilon>18\;\exists\:d_0\in\mathbb{N}\;\forall\:d\geqslant d_0\colon \P\bigl[\hspace{1pt}\bigl|\|x-y\|-\sqrt{2d}\hspace{1.5pt}\bigr|\geqslant\epsilon\hspace{1pt}\bigr]\leqslant\medfrac{18}{\epsilon}.$\vspace{3pt}

\item[(ii)] $\forall\:\epsilon>18,\,d\geqslant\epsilon^2\colon\P\bigl[\hspace{1pt}\bigl|\|x-y\|-\sqrt{2d}\hspace{1.5pt}\bigr|\geqslant\epsilon\hspace{1pt}\bigr]\leqslant\medfrac{18}{\epsilon} +\medfrac{8}{\sqrt{d}}.$
\end{compactitem}
\end{thm}
\begin{proof} We use the same trick as in the beginning of the proof of the Gaussian Annulus Theorem (Theorem \ref{GA-THM}) and then use the cosine law to get
\begin{align*}
\P\bigl[\bigl|\|X-Y\|-\sqrt{2d}\bigr|\geqslant\epsilon\hspace{1pt}\bigr]&\leqslant\P\bigl[\hspace{1pt}\bigl|\|X-Y\|-\sqrt{2d}\hspace{1pt}\bigr|(\|X-Y\|+\sqrt{2d}\hspace{1pt})\geqslant\epsilon\sqrt{2d}\hspace{2pt}\bigr]\\
&=\P\bigl[\hspace{1pt}\bigl|\|X-Y\|^2-2d\bigr|\geqslant\epsilon\sqrt{2d}\hspace{2pt}\bigr]\\
&=1-\P\bigl[\hspace{1pt}\bigl|\|X-Y\|^2-2d\bigr|\leqslant\epsilon\sqrt{2d}\hspace{2pt}\bigr]\\
&=1-\P\bigl[-\epsilon\sqrt{2d}\leqslant\|X\|^2+\|Y\|^2-2\langle{}X,Y\rangle{}-2d\leqslant\epsilon\sqrt{2d}\hspace{2pt}\bigr]\\
& \leqslant1-\P\bigl[\hspace{1pt}\bigr|\|X\|^2-d\hspace{1pt}\bigr|\leqslant\medfrac{\epsilon\sqrt{2d}}{3}\,\bigr]^2\cdot\P\bigl[\bigr|\langle{}X,Y\rangle{}\bigr|\leqslant\medfrac{\epsilon\sqrt{2d}}{6}\,\bigr]=:(\circ)
\end{align*}
Now we employ Proposition \ref{GA-COR} to obtain
\begin{align*}
\P\bigl[\bigr|\|X\|^2-d\hspace{1pt}\bigr|\geqslant\medfrac{\epsilon\sqrt{2d}}{3}\,\bigr]&\leqslant2\exp\bigl(-\medfrac{1}{8}\hspace{1pt}\min\bigl(\medfrac{(\epsilon\sqrt{2d}/3)^2}{2d},\medfrac{\epsilon\sqrt{2d}}{3}\bigr)\bigr)\\
& \leqslant 2\exp\bigl(-\medfrac{1}{8}\hspace{1pt}\min\bigl(\medfrac{\epsilon^2}{9},\medfrac{\epsilon^2\sqrt{2}}{3}\bigr)\bigr)\\
& = 2\exp\bigl(-\medfrac{\epsilon^2}{72}\bigr)
\end{align*}
where we used $\sqrt{d}\geqslant\epsilon$ for the second estimate. This leads to
\begin{equation}\label{EST-1}
P\bigl[\hspace{1pt}\bigr|\|X\|^2-d\hspace{1pt}\bigr|\leqslant\medfrac{\epsilon\sqrt{2d}}{3}\,\bigr]\geqslant1-2\exp\bigl(-\medfrac{\epsilon^2}{72}\bigr)
\end{equation}
which will allow us below to estimate the squared term in $(\circ)$. In order to estimate the non-squared term we use Theorem \ref{GO-THM} from which we get with $\delta:=\epsilon/(6\sqrt{2d})>0$ the estimate
\begin{align*}
\medfrac{2/\delta+7}{\sqrt{d}}&\geqslant\P\bigl[\hspace{1pt}\big|\bigl\langle{}X,Y\bigr\rangle{}\big|\geqslant\delta\|X\|\|Y\|\hspace{1pt}\bigr]\\
&\geqslant \P\bigl[\hspace{1pt}\big|\bigl\langle{}X,Y\bigr\rangle{}\big|\geqslant\delta\|X\|\|Y\|\,\big|\,\|X\|,\|Y\|\leqslant\sqrt{2d}\hspace{2pt}\bigr]\cdot \P\bigl[\hspace{1pt}\|X\|,\|Y\|\leqslant\sqrt{2d}\hspace{2pt}\bigr]\\
& \geqslant \P\bigl[\hspace{1pt}\big|\bigl\langle{}X,Y\bigr\rangle{}\big|\geqslant\delta\sqrt{2d}\sqrt{2d}\hspace{2pt}\bigr]\cdot \P\bigl[\hspace{1pt}\|X\|\leqslant\sqrt{2d}\hspace{2pt}\bigr]^2\\
& =\P\bigl[\hspace{1pt}\big|\bigl\langle{}X,Y\bigr\rangle{}\big|\geqslant\delta\hspace{1pt}6\sqrt{2d}\cdot\medfrac{\sqrt{2d}}{6}\hspace{1pt}\bigr]\cdot \P\bigl[\hspace{1pt}\big|\|X\|^2-d\big|\leqslant d\hspace{1pt}\bigr]^2\\
& \geqslant\P\bigl[\hspace{1pt}\big|\bigl\langle{}X,Y\bigr\rangle{}\big|\geqslant\epsilon\medfrac{\sqrt{2d}}{6}\hspace{1pt}\bigr]\cdot \bigl(1-2\exp\bigl(-\medfrac{1}{8}\min(d/2,d)\bigr)\bigr)^2\\
&=\P\bigl[\hspace{1pt}\big|\bigl\langle{}X,Y\bigr\rangle{}\big|\geqslant\medfrac{\epsilon\sqrt{2d}}{6}\hspace{1pt}\bigr]\cdot \bigl(1-2\exp\bigl(-\medfrac{d}{16}\bigr)\bigr)^{2}
\end{align*}
where we used Proposition \ref{GA-COR} with $\epsilon=d$ to treat the second factor in the last estimate. Bringing the squared factor in the last expression on the other side and recalling $\delta:=\epsilon/(6\sqrt{2d})$ leads to
\begin{equation}\label{EST-2}
\begin{array}{rl}
\P\bigl[\bigr|\langle{}X,Y\rangle{}\bigr|\leqslant\medfrac{\epsilon\sqrt{2d}}{6}\,\bigr]&=1-\P\bigl[\bigr|\langle{}X,Y\rangle{}\bigr|\geqslant\medfrac{\epsilon\sqrt{2d}}{6}\,\bigr]\\
&\geqslant 1-\medfrac{2/\delta+7}{\sqrt{d}}\cdot\bigl(1-2\exp\bigl(-\medfrac{d}{16}\bigr)\bigr)^{-2}\\
&\geqslant 1- \bigl(\medfrac{12\sqrt{2}}{\epsilon}+\medfrac{7}{\sqrt{d}} \bigr)\cdot\bigl(1-2\exp\bigl(-\medfrac{d}{16}\bigr)\bigr)^{-2}\\
\end{array}
\end{equation}
Now we combine \eqref{EST-1} and \eqref{EST-2} to continue estimating
\begin{align*}
(\circ) & = 1-\P\bigl[\bigr|\|X\|^2-d\bigr|\leqslant\medfrac{\epsilon\sqrt{2d}}{3}\,\bigr]^2\cdot\P\bigl[\bigr|\langle{}X,Y\rangle{}\bigr|\leqslant\medfrac{\epsilon\sqrt{2d}}{6}\,\bigr]\\
&\leqslant 1-\bigl(1-2\exp\bigl(-\medfrac{\epsilon^2}{72}\bigr)\bigr)^2\cdot\Bigr[1- \bigl(\medfrac{12\sqrt{2}}{\epsilon}+\medfrac{7}{\sqrt{d}} \Bigr)\cdot\bigl(1-2\exp\bigl(-\medfrac{d}{16}\bigr)\bigr)^{-2}\hspace{1pt}\Bigr]\\
& \leqslant 1-\bigl(1-4\exp\bigl(-\medfrac{\epsilon^2}{72}\bigr)\bigr)\cdot\Bigr[1- \frac{\textstyle\frac{12\sqrt{2}}{\epsilon}+\frac{7}{\sqrt{d}}}{\textstyle(1-2\exp(-\frac{d}{16}))^2}\Bigr]\\
& \leqslant 4\exp\bigl(-\medfrac{\epsilon^2}{72}\bigr) + \frac{\textstyle\frac{12\sqrt{2}}{\epsilon}+\frac{7}{\sqrt{d}}}{\textstyle(1-2\exp(-\frac{d}{16}))^2}=:f(\epsilon,d)
\end{align*}
where we see that $\lim_{d\rightarrow\infty}f(\epsilon,d)=4\exp(-\epsilon^2/72)+12\sqrt{2}/\epsilon$ holds. Using a calculator one verifies that $4\exp(-\epsilon^2/72)<1/\epsilon$ holds for $\epsilon>17.5$. From this we get that for suitable large $d$ and $\epsilon>17.5$ the estimate
$$
f(\epsilon,d)\leqslant\medfrac{1}{\epsilon}+\medfrac{12\sqrt{2}}{\epsilon}<\medfrac{18}{\epsilon}
$$
holds which implies (i). For (ii) we continue estimating $f(\epsilon,d)$. For $\epsilon>17.5$, we pick $d\geqslant\epsilon^2>306.25$, so $d\geqslant307$. For this choice of $d$ we check with a calculator that $(1-2\exp(-d/16))^{-2}\leqslant1.0002$ which leads to an estimate for the fraction. With this we get
\begin{align*}
f(\epsilon,d) & = 4\exp\bigl(-\medfrac{\epsilon^2}{72}\bigr) + \frac{\textstyle\frac{12\sqrt{2}}{\epsilon}+\frac{7}{\sqrt{d}}}{\textstyle1-4\exp(-\frac{d}{16})}\leqslant \medfrac{1}{\epsilon}+1.0002\bigl(\medfrac{12\sqrt{2}}{\epsilon} +\medfrac{7}{\sqrt{d}}\bigr)\leqslant\medfrac{18}{\epsilon}+\medfrac{8}{\sqrt{d}}
\end{align*}
as desired.
\end{proof}

Since we do need this later, let us formulate the following probability estimate that we established in the proof above.

\begin{cor}\label{COR-distance-of-two-points-sqrt-2d} Let $x$, $y\in\mathbb{R}^d$ be drawn at random with respect to the spherical Gaussian distribution with zero mean and unit variance. Then
$$
\forall\:\epsilon>18,\,d\geqslant\epsilon^2\colon \P\bigl[\hspace{2pt}\bigl|\|x-y\|^2-2d\hspace{1pt}\bigr|\geqslant\epsilon\sqrt{2d}\hspace{2pt}\bigr]\leqslant \medfrac{18}{\epsilon}+\medfrac{8}{\sqrt{d}}
$$
holds.\hfill\qed
\end{cor}

\section*{Problems}

\begin{probl}\label{SUM-GAUSS-PROB} Generalize Proposition \ref{LK-GAUSS-LEM} as follows. For $i=1,\dots,d$ let $X_i\sim\mathcal{N}(\mu_i,\sigma_i)$ be independent Gaussian random variables. Let $\lambda_i\not=0$ be real numbers. Show that $X:=\lambda_1X_1+\cdots+\lambda_dX_d$ is again a Gaussian random variable with mean $\mu=(\mu_1+\cdots+\mu_d)/d$ and $\sigma^2=\lambda_1^2\sigma_1^2+\cdots+\lambda_d^2\sigma_d^2$.
\end{probl}

%%%%%%%%%%%%%%%%%%%%%%%%%%%%%%%%%%%%%%%%%%%%%%%%
%%%%%%%%%%%%%%%%%%%%%%%%%%%%%%%%%%%%%%%%%%%%%%%%
%%                                            %%
%% Chapter 11: Random Projections             %%
%%                                            %%
%%%%%%%%%%%%%%%%%%%%%%%%%%%%%%%%%%%%%%%%%%%%%%%%
%%%%%%%%%%%%%%%%%%%%%%%%%%%%%%%%%%%%%%%%%%%%%%%%

\chapter{Random projections}\label{Ch-RP}

Let $d, k\geqslant1$ with $k\leqslant{}d$. In the sequel we identify $\mathbb{R}^{k\times d}\cong\L(\mathbb{R}^d,\mathbb{R}^k)$ and refer to its elements sometimes as matrices and sometimes as linear maps.

\smallskip

Let us firstly consider a matrix $A\in\mathbb{R}^{k\times d}$ whose $k$ rows are orthogonal. Then the map $\mathbb{R}^k\rightarrow\mathbb{R}^d$, $x\mapsto A^Tx$ is injective and we can regard $\mathbb{R}^k$ via $A^T$ as a subspace of $\mathbb{R}^d$. The linear map $P\colon\mathbb{R}^d\rightarrow\mathbb{R}^d$, $Px:=A^TAx$ is then idempotent, its range is isomorphic to $\mathbb{R}^k\subseteq\mathbb{R}^d$ in the aforementioned sense, and thus $P$ is an \emph{orthogonal projection onto $\mathbb{R}^k$}. Using the basis extension theorem and the Gram-Schmidt procedure it is easy to see that for every linear subspace $V\subseteq \mathbb{R}^d$ with $\dim V=k$ there exists an orthogonal projection and that each such projection can be realized by a matrix $A\in\mathbb{R}^{k\times d}$ with $k$ orthogonal rows.

\smallskip

In this chapter our aim is to understand the properties of maps given by matrices that we pick at random. For this let $(\Omega,\Sigma,\P)$ be a probability space.

\begin{dfn}\label{DFN-RLM} A measurable function $U\colon\Omega\rightarrow\mathbb{R}^{k\times d}$ is called a \textit{random matrix}.
\end{dfn}

In order to specify the underlying distribution, we extend Proposition \ref{PROP-9-1} to matrices. We first note that if a matrix
$$
U=\begin{bmatrix}
u_{11} & u_{12} & \cdots & u_{1d} \\
u_{21} & u_{22} & \cdots & u_{2d} \\
\vdots &  &  & \vdots \\
u_{k1} & u_{k2} & \cdots & u_{kd} \\
\end{bmatrix}
$$
is given, the map $U\colon\mathbb{R}^d\rightarrow\mathbb{R}^k$ is given by $Ux = \bigl(\langle{}u_1,x\rangle{},\langle{}u_2,x\rangle{},\dots,\langle{}u_k,x\rangle{}\bigr)$ for $x\in\mathbb{R}^d$, if we denote by $u_i=(u_{i1},u_{i2},\dots,u_{id})$ the $i$-th row of $U$ for $i=1,\dots,k$. If $U\colon \Omega\rightarrow\mathbb{R}^{k\times d}$ is a random matrix, then we denote by $u_{ij}\colon\Omega\rightarrow\mathbb{R}$ the random variables that arise as the coordinate functions of $U$ and by $u_{i}\colon\Omega\rightarrow\mathbb{R}^d$ the random vectors that arise, if we combine the coordinate functions of the $i$-th row back into a vector.
Employing this notation, we get the following result.

\begin{prop}\label{PROP-RANDOM-MATRIX} Let $U\colon\Omega\rightarrow\mathbb{R}^{k\times d}$ be a random matrix with entries $u_{ij}\colon\Omega\rightarrow\mathbb{R}$ and row vectors $u_{i}\colon\Omega\rightarrow\mathbb{R}^d$. Then the following are equivalent.
\begin{compactitem}

\item[(i)] $U\sim\mathcal{N}(0,1,\mathbb{R}^{k\times d})$.

\vspace{2pt}

\item[(ii)] $\forall\:i=1,\dots,k\colon u_i\sim\mathcal{N}(0,1,\mathbb{R}^d)$.
 
\vspace{2pt}

\item[(iii)] $\forall\:i=1,\dots,d,\,j=1,\dots,k\colon u_{ij}\sim\mathcal{N}(0,1)$.

\end{compactitem}
\end{prop}

\begin{proof} Since $\mathbb{R}^{k\times d}\cong\mathbb{R}^{kd}$ holds, we see that a random matrix $U\in\mathbb{R}^{k\times d}$ can be identified with a real random vector $U\colon\Omega\rightarrow\mathbb{R}^{dk}$. In that sense we understand (i). Having done this rearrangement the equivalences follow immediately from Proposition \ref{PROP-9-1}.
\end{proof}

By Proposition \ref{PROP-RANDOM-MATRIX} we can write $U\sim\mathcal{N}(0,1,\mathbb{R}^{k\times d})$ and understand this matrix-wise, row vector-wise or element-wise. If we draw $U\in\mathbb{R}^{k\times d}$ at random with $U\sim\mathcal{N}(0,1,\mathbb{R}^{k\times d})$, then by Theorem \ref{GO-THM} we have
$$
\P\bigl[\hspace{1pt}\bigl|\bigl\langle{}\medfrac{u_i}{\|u_i\|},\medfrac{u_\ell}{\|u_\ell\|}\bigr\rangle{}\bigr|\leqslant\epsilon\hspace{1pt}\bigr]\geqslant \medfrac{2/\epsilon+7}{\sqrt{d}}
$$
for $i\not=\ell$, which means that the rows are almost orthogonal with high probability for large dimensions $d$.
%\footnote{\red{This would read much better if there'd be an analogon of Theorem \ref{UO-THM} for the Gaussian case (cf.\ earlier footnote). On the other hand is the JL-Theorem exactly this: By JL such a matrix is almost isometric (up to the constant) and isometric maps are precisely the orthogonal ones!}}
We will therefore call a matrix $U\sim\mathcal{N}(0,1,\mathbb{R}^{k\times d})$ a \emph{random projection from $\mathbb{R}^d$ to $\mathbb{R}^k$}. Observe however that $U$ needs in general neither to be orthogonal, nor a projection.

\smallskip

We will study now, how a random projection in the aforementioned sense does distorts length. The bound in the result below is non-trivial for $k\epsilon^2>16\ln2\,(\approx11.1)$.

\begin{thm}\label{RP-THM}(Random Projection Theorem) Let $U\in\mathbb{R}^{k\times d}$ be a random matrix with $U\sim\mathcal{N}(0,1)$. Then for every $x\in\mathbb{R}^d\backslash\{0\}$ and every $0<\epsilon<1$ we have
$$
\P\bigl[\hspace{1pt}\big|\|Ux\|-\sqrt{k}\hspace{1pt}\|x\|\big|\geqslant\epsilon\sqrt{k}\hspace{1pt}\|x\|\hspace{1pt}\bigr]\leqslant 2\exp(-ck\epsilon^2)
$$
where $c=1/16$.
\end{thm}
\begin{proof}Let $x\in\mathbb{R}^d\backslash\{0\}$ be fixed and let $U\colon\Omega\rightarrow\mathbb{R}^{k\times d}$ be a random matrix with $U\sim\mathcal{N}(0,1)$. We consider
$$
U(\cdot)\medfrac{x}{\|x\|}\colon\Omega\rightarrow\mathbb{R}^k
$$
and observe that each of its coordinate functions
$$
\pr_i(U(\cdot)\medfrac{x}{\|x\|}) = \bigl\langle{}u_i,\medfrac{x}{\|x\|}\bigr\rangle{} = \Bigsum{j=1}{d}u_{ij}\medfrac{x_j}{\|x\|}
$$
is a linear combination of random variables $u_{ij}\sim\mathcal{N}(0,1)$ by Proposition \ref{PROP-RANDOM-MATRIX} whose squared coeffients ${\textstyle\frac{x_j^2}{\|x\|^2}}$ sum up to one. By Proposition \ref{LK-GAUSS-LEM} this means $U(\cdot){\textstyle\frac{x}{\|x\|}}\sim\mathcal{N}(0,1,\mathbb{R}^k)$ and we can deduce

\begin{align*}
\P\bigl[\hspace{1pt}\bigl|\|Ux\|-\sqrt{k}\hspace{1pt}\|x\|\bigr|\geqslant \epsilon\sqrt{k}\hspace{1pt}\|x\|\hspace{1pt}\bigr] & = \P\bigl[\hspace{1pt}\bigl|\|U\medfrac{x}{\|x\|}\|-\sqrt{k}\hspace{1pt}\bigr|\geqslant \epsilon\sqrt{k}\,\bigr]\\
&\leqslant 2\exp\bigl(-\medfrac{(\epsilon\sqrt{k})^2}{16}\bigr)\\
&=2\exp\bigl(-\medfrac{\epsilon^2 k}{16}\bigr)
\end{align*}
where we applied Theorem \ref{GA-THM} in the space $\mathbb{R}^k$ to get the estimate.
\end{proof}

Notice that we have on both sides of the estimate the expression $\epsilon\sqrt{k}$. On the right hand side we need this number to be suitable large to get the probability down, but on the right we would rather have it small in order to get the distortion small. This conundrum we can solve by  `absorbing' the factor $\sqrt{k}$ into the notation. That is, for a random matrix $U\sim\mathcal{N}(0,1,\mathbb{R}^{k\times d})$ we consider the \emph{Johnson-Lindenstrauss projection} 
$$
T_{\scriptscriptstyle U}\colon\mathbb{R}^{d}\rightarrow\mathbb{R}^k,\;T_{\scriptscriptstyle U}x:={\textstyle\frac{1}{\sqrt{k}}}Ux.
$$
Then Theorem \ref{RP-THM} reads
$$
\P\bigl[\hspace{1pt}\big|\|T_{\scriptscriptstyle U}x\|-\|x\|\big|\geqslant\epsilon\|x\|\hspace{1pt}\bigr]\leqslant 2\exp(-ck\epsilon^2)
$$
and we can think of $\epsilon$ being small but $k$ being so large that $k\epsilon^2$ is large. If we do this, then we get that $\big|\|T_{\scriptscriptstyle U}x\|-\|x\|\big|<\epsilon\|x\|$ holds with high probability. Since this is a multiplicative error estimate, we can further rewrite this as
$$
\medfrac{\|T_{\scriptscriptstyle U}x\|}{\|x\|}\approx 1.
$$
If we restrict ourselves to $x\in K$ with, e.g., a fixed compact set $K\subseteq\mathbb{R}^d$, then we can even pick $k$ and $\epsilon$ accordingly and achieve that $\|T_{\scriptscriptstyle U}x\|\approx\|x\|$ holds  on whole $K$ with high probability.

\smallskip

The following result will extend the latter to mutual distances within a sets of $n$ points.

\begin{thm}\label{JL-LEM}(Johnson-Lindenstrauss Lemma) Let $0<\epsilon<1$ and $n\geqslant1$ and let $k\geqslant\frac{48}{\epsilon^2}\ln n$. Let $U\sim\mathcal{N}(0,1,\mathbb{R}^{k\times d})$ be a random matrix. Then the Johnson-Lindenstrauss projection $T_{\scriptscriptstyle U}$ satisfies for any set of $n$--many points $x_1,\dots,x_n\in\mathbb{R}^d$ the estimate
$$
\P\bigl[(1-\epsilon)\|x_i-x_j\|\leqslant\|T_{\scriptscriptstyle U}x_i-T_{\scriptscriptstyle U}x_j\|\leqslant (1+\epsilon)\|x_i-x_j\| \text{ for all } i,j\bigr]\geqslant 1-\medfrac{1}{n}.
$$
\end{thm}
\begin{proof} Firstly, we fix $i\not=j$, put $x:=x_i-x_j$ and compute
\begin{align*}
&\hspace{-35pt}\P\bigl[\hspace{1pt}\|T_{\scriptscriptstyle U}x_i-T_{\scriptscriptstyle U}x_j\|\not\in \bigl((1-\epsilon)\|x_i-x_j\|,(1+\epsilon)\|x_i-x_j\|\bigr)\bigr] \\
\phantom{xxxx}& = \P\bigl[\hspace{1pt}\|T_{\scriptscriptstyle U}x\|\not\in \bigl((1-\epsilon)\|x\|,(1+\epsilon)\|x\|\bigr)\bigr]\\
& =\P\bigl[\hspace{1pt}\bigl|\|T_{\scriptscriptstyle U}x\|-\|x\|\bigr|\geqslant \epsilon\|x\|\hspace{1pt}\bigr]\\
& =\P\bigl[\hspace{1pt}\bigl|\|Ux\|-\sqrt{k}\hspace{1pt}\|x\|\bigr|\geqslant \epsilon\sqrt{k}\hspace{1pt}\|x\|\hspace{1pt}\bigr]\\
&\leqslant 2\exp\bigl(-k\epsilon^2/16\bigr)
\end{align*}
where we used Theorem \ref{RP-THM} in the last step. Since the inequality $(1-\epsilon)\|x_i-x_j\|\leqslant\|T_{\scriptscriptstyle U}x_i-T_{\scriptscriptstyle U}x_j\|\leqslant (1+\epsilon)\|x_i-x_j\| $ is always true for $i=j$ and there are ${n\choose 2}\leqslant n^2/2$ choices for $i\not=j$, we get
\begin{align*}
&\P\bigl[\,\forall\:i,j\colon (1-\epsilon)\|x_i-x_j\|\leqslant\|T_{\scriptscriptstyle U}x_i-T_{\scriptscriptstyle U}x_j\|\leqslant (1+\epsilon)\|x_i-x_j\| \bigr]\\
& = 1-\P\bigl[\,\exists\:i\not=j\colon \|T_{\scriptscriptstyle U}x_i-T_{\scriptscriptstyle U}x_j\|\not\in \bigl((1-\epsilon)\|x_i-x_j\|,(1+\epsilon)\|x_i-x_j\|\bigr)\bigr]\\
& \geqslant 1-\frac{n^2}{2}\cdot2\exp\bigl(-k\epsilon^2/16\bigr)\\
& \geqslant  1-n^2\exp\bigl(-\bigl(\medfrac{48}{\epsilon^2}\ln n\bigr)\,\epsilon^2/16\bigr)\\
& \geqslant  1-n^2\exp\bigl(\ln n^{-48/16}\bigr)\bigr)\\
& =  1-n^2 n^{-3}\\
& =  1-\medfrac{1}{n}
\end{align*}
where we used the selection of $k$ for the second last estimate.
\end{proof}

Similar to our remarks on Theorem \ref{RP-THM}, we can derive from Theorem \ref{JL-LEM} that 
$$
\medfrac{\|T_{\scriptscriptstyle U}x_i-T_{\scriptscriptstyle U}x_j\|}{\|x_i-x_j\|}\approx 1
$$
holds for all $x_i\not=x_j$ with high probability by picking $\epsilon$ small but $k$ so large that $k\epsilon^2$ exceeds $48\ln(n)$. We point out here that since the logarithm grows very slowly, making the sample set larger does only require to increase $k$ a little. If we agree to pick all the families $x_1,\dots,x_n$ from within a compact set $K\subseteq\mathbb{R}^d$, then we can again achieve that
$$
\|T_{\scriptscriptstyle U}x_i-T_{\scriptscriptstyle U}x_j\|\approx \|x_i-x_j\|
$$
holds with high probability by picking $\epsilon$ and $k$ as explained above. From this stems the interpretation that the Johnson-Lindenstrauss Lemma allows to reduce the dimension of a data set while preserving the mutual distances between the data points up to some small error. Notice that the set of points (if it is only $n$--many) is arbitrary. In particular the projection map $T_{\scriptscriptstyle U}$ and its codomain depend only on $n$ and $\epsilon$. That means after agreeing on the number of points we want to project and the error bound we want to achieve, we get the same estimate for all families $x_1,\dots,x_n$ of $n$--many points. Furthermore, the estimate holds indeed \emph{for all} of the $n$ points and not only \emph{most} (one could think of a result saying that if we pick a pair $(i,j)$ then with probability $1-1/n$ the estimate is true, which would mean that there could be some `bad' pairs whose distance gets distorted by more than $\epsilon$, but this is \emph{not} the case!).

\smallskip

The following figure shows the distortion that occurred depending on $k$ in an experiment compared with the bound provided by the Johnson-Lindenstrauss Lemma, see Problem \ref{JL-SIM}.

\begin{center}
\begin{tikzpicture}

	\begin{axis}
[
axis line style={thick, shorten >=-10pt, shorten <=-10pt},
y=120pt,
axis y line=left,
axis x line=middle,
axis line style={->},
no markers,
tick align=outside,
major tick length=2pt,
ymin=0.05,
ytick={0.2,0.4,...,1.0},
xmin=0,
xtick={0, 100, ..., 1000},
every tick label/.append style={font=\tiny},
xlabel=\small $k$,
ylabel=\small $\epsilon$,
every axis x label/.style={
    at={(ticklabel* cs:1.07)},
    anchor=west,
},
every axis y label/.style={
    at={(ticklabel* cs:1.2)},
    anchor=north,
},
]
	% use TeX as calculator:
		\addplot[ mark=none,fill=black, 
                    fill opacity=0.00] coordinates {
(20,0.64960412)
(40,0.43647613)
(60,0.38151316)
(80,0.32697068)
(100,0.28641338)
(120,0.28950316)
(140,0.29823953)
(160,0.23320381)
(180,0.21251614)
(200,0.23225801)
(220,0.22376387)
(240,0.20401529)
(260,0.18940454)
(280,0.15808729)
(300,0.14830084)
(320,0.16424003)
(340,0.15829449)
(360,0.14747864)
(380,0.16493581)
(400,0.14438415)
(420,0.15194512)
(440,0.13889315)
(460,0.12463126)
(480,0.14175293)
(500,0.14543227)
(520,0.12134347)
(540,0.1235162)
(560,0.11541249)
(580,0.12772562)
(600,0.12293452)
(620,0.13992636)
(640,0.10152627)
(660,0.11849245)
(680,0.10807353)
(700,0.10787662)
(720,0.10581649)
(740,0.10253265)
(760,0.10597364)
(780,0.10030326)
(800,0.09757551)
(820,0.11491439)
(840,0.10146898)
(860,0.11099813)
(880,0.10535177)
(900,0.09765621)
(920,0.09172647)
(940,0.0993466)
(960,0.09343611)
	};
\addplot [domain=100:960, samples=101,dashed]{sqrt(100/x)};
	\end{axis}
\end{tikzpicture}\nopagebreak[4]\begin{fig}\label{JL-FIG}Solid line: Maximal distortion that occurs when $n=300$ points $x\sim\mathcal{N}(0,1,\mathbb{R}^{1000})$ are projected to $\mathbb{R}^k$. Dashed line: JL-bound $\epsilon = (48\ln(n)/k)^{1/2}$.\end{fig}
\end{center}

\smallskip

We conclude with the following corollary that describes the length distortion for all points in $\mathbb{R}^d$ (or in a compact $K\subseteq\mathbb{R}^d)$) that we can expect if we apply the Johnson-Lindenstrauss Projection and pick the parameters to treat sets with $n$-many points.

\begin{cor} Under the assumptions of Theorem \ref{JL-LEM} we have
$$
\P\bigl[(1-\epsilon)\|x\|\leqslant\|T_{\scriptscriptstyle U}x\|\leqslant (1+\epsilon)\|x\|\bigr]\geqslant 1-{\textstyle\frac{1}{n}}
$$
for every $x\in\mathbb{R}^d$.\hfill\qed
\end{cor}

\section*{Problems}

\begin{probl}\label{NaiveTailBound} Let $X_i\sim\mathcal{N}(0,1)$ and $X=X_1+\dots+X_d$. \vspace{3pt}

\begin{compactitem}

\item[(i)] Use the formula $\E(f(X_i))=(2\pi)^{-1/2}\int_{\mathbb{R}}f(x)\exp(-x^2/2)\dd x$ to show that $\E(\exp(tX_i))=(1-2t)^{-d/2}$ holds for $t\in(0,1/2)$.\vspace{3pt}

\item[(ii)] Derive the estimate $\P\bigl[X\geqslant a\bigr] \leqslant\inf_{t\in(0,1/2)}\frac{\exp(-ta)}{(1-2t)^{d/2}}$ for $a>0$.\vspace{3pt}
\end{compactitem}
\end{probl}

\begin{probl}\label{AlternativeJLProof} Use Problem \ref{NaiveTailBound} to give an alternative proof of the Johnson-Lindenstrauss Lemma that does not rely on the Gaussian Annulus Theorem.
\end{probl}

\begin{probl}\label{PROB-LANDAU} Restate the Random Projection Theorem and the Johnson-Lindenstrauss Theorem by using Landau symbols to express the dependencies of $\epsilon$, $k$ and $n$.
\end{probl}

\begin{probl}\label{JL-SIM} Let $n>k$ be integers.\vspace{2pt}

\begin{compactitem}
\item[(i)] Implement a random projection $T\colon\mathbb{R}^d\rightarrow\mathbb{R}^k$  and test it for small $d$ and $k$.\vspace{5pt}

\item[(ii)] Put $d=1000$, generate 10\,000 points in $\mathbb{R}^d$ at random, project them via $T$ to $\mathbb{R}^k$ and compute, for different values of $k$, the worst distorsion of a pairwise distance
$$
\epsilon:=\max\Bigl(1-\min_{x\not=y}\medfrac{\|Tx-Ty\|}{\sqrt{k}\hspace{1pt}\|x-y\|}, \max_{x\not=y}\medfrac{\|Tx-Ty\|}{\sqrt{k}\hspace{1pt}\|x-y\|}-1\Bigr)
$$
that occurs.\vspace{3pt}

\item[(iii)] Compare the outcome of your experiment with the relation of $\epsilon$ and $k$ that is given in the Johnson-Lindenstrauss Lemma, by replicating Figure \ref{JL-FIG} based on the data from (ii).

\vspace{3pt}

\item[(iv)] Explain how you would pick $\epsilon$ and $k$ if you are given a dataset in $\mathbb{R}^d$ and want to do a dimensionality reduction that does not corrupt a classifier based on nearest neighbors.

\end{compactitem}
\end{probl}

%%%%%%%%%%%%%%%%%%%%%%%%%%%%%%%%%%%%%%%%%%%%%%%%
%%%%%%%%%%%%%%%%%%%%%%%%%%%%%%%%%%%%%%%%%%%%%%%%
%%                                            %%
%% Chapter 1: Introduction                    %%
%%                                            %%
%%%%%%%%%%%%%%%%%%%%%%%%%%%%%%%%%%%%%%%%%%%%%%%%
%%%%%%%%%%%%%%%%%%%%%%%%%%%%%%%%%%%%%%%%%%%%%%%%

\chapter{Separating Gaussian data}\label{Ch-Fit}

After having studied in previous chapters the properties of high-dimensional data that comes from a Gaussian distribution, we now want use this knowledge to understand under which conditions it is possible to separate (or `disentangle') the data that stems from two `Gaussians' with variance one and different means. Let us first look at low dimensions, for example at the 2-dimensional data set in Figure \ref{FIG-6-1}, which we generated by drawing samples of equal size from two Gaussians with different centers.

\vspace{-15pt}

\begin{center}
\begin{tikzpicture}[scale=1.1]
	\begin{axis}[
	axis line style={draw=none},
	ymin=-1,
    ymax=4.25,
	ticks=none,
	xmin=-5,
    xmax=4,
	height=5.5cm,
	width=8.5cm,
	yticklabels={,,},
	xticklabels={,,}
	]

\addplot[only marks,mark=*,mark options={scale=0.6, fill=black!50!white, color=black!50!white},text mark as node=true] table[col sep=comma] {2d-Gauss-1.dat};

\addplot[only marks,mark=*,mark options={scale=0.6, fill=black!50!white, color=black!50!white},text mark as node=true] table[col sep=comma] {2d-Gauss-2.dat};
	\end{axis}
\end{tikzpicture}
\vspace{-15pt}
\begin{fig}\label{FIG-6-1}2-dimensional dataset arising\\from two Gaussian distributions.\end{fig}
\end{center}
 
Separation (or disentanglement) in Figure \ref{FIG-6-1} would mean to determine for each point if it came from the lower left Gaussian or from the upper right Gaussian. Firstly, we observe that this cannot be done in a deterministic way but only `with high probability for almost all points'. Secondly we see, that for this to be possible it is necessary that the distance between the two means is sufficiently large. If this is the case, we might have a good chance that pairs of data points with small distances belong to the same Gaussian whereas pairs with large distances belong to different Gaussians. Moreover, there will be only a few points located half way between the centers which means that if we start by picking an arbitrary data point and then compute the other points' distances from the fixed one, we may expect that two clusters, corresponding to the two Gaussians, appear.

\smallskip

In high dimensions the picture looks different. We know that here most data points of a Gaussian are located in an annulus with radius $\sqrt{d}$ around the mean of the Gaussian. An analogue of the assumption in the 2-dimensional case would thus require that the distance $\Delta:=\|\mu_1-\mu_2\|$ between the means is larger than $2\sqrt{d}$, see Figure \ref{FIG-6-2}, implying that the data sets will be almost disjoint.

\begin{center}
\begin{tikzpicture}

\begin{scope}
\fill[pattern=north west lines,opacity=.6,draw] (-1.8,0) circle [radius=1.15];
\fill [white, draw=black] (-1.8,0) circle [radius=0.9];
\end{scope}

\fill[fill=black] (-1.8,0) circle (1pt);
\node (1) at (-0.7,1) {\footnotesize$\sqrt{d}$};
\draw[-latex] (-1.8,0)--(-1,0.73);
\node (2) at (-2.075,-0.1) {\footnotesize$\mu_1$};
\begin{scope}
\fill[pattern=north west lines,opacity=.6,draw] (1.8,0.7) circle [radius=1.15];
\fill [white, draw=black]  (1.8,0.7) circle [radius=0.9];
\end{scope}

\fill[fill=black] (1.8,0.7) circle (1pt);
\draw[-latex] (1.8,0.7)--(1.05,-0.1);
\node (1) at (0.73,-0.4) {\footnotesize$\sqrt{d}$};
\node (2) at (2.1,0.7) {\footnotesize$\mu_2$};
\end{tikzpicture}
\nopagebreak[4]
\begin{fig}\label{FIG-6-2}Two high-dimensional Gaussians with distance $\|\mu_1-\mu_2\|>2\sqrt{d}$.\end{fig}
 \end{center}
 
If this is the case, then we may expect that separation can be done analogously to the low-dimensional situation. There are however two principal differences: Firstly, our asymptotic approach to high dimensions means that we do not consider $d$ as a constant but always ask what happens for $d\rightarrow\infty$. Therefore, we need to understand the distance of the means $\Delta=\Delta(d)$ as a function of the dimension as well. Otherwise, i.e., for constant $\Delta$ and dimension tending to infinity, the two annuli would eventually overlap more and more and become almost coincident. From this we see that we have to require that $\Delta\rightarrow\infty$. Our low-dimensional arguments suggest that $\Delta>2\sqrt{d}$ yields one scenario in which disentanglement can be achieved. On the other hand, the picture in Figure \ref{FIG-6-3} suggests that the overlap might also be small, even if $\Delta\leqslant2\sqrt{d}$.

\smallskip

We emphasize that Figure \ref{FIG-6-2} and in particular Figure \ref{FIG-6-3} constitute rather imperfect 2-dimensional abstract illustrations of high dimensional facts. Notice, e.g., that any two points in one of the annuli have, with high probability, distance $\sqrt{2d}$. The picture does however not represent this at all!
 
\begin{center}
\begin{tikzpicture}

\begin{scope}
\fill[pattern=north west lines,opacity=.6,draw] (-1.8,0) circle [radius=1.15];
\fill [white, draw=black!60!white] (-1.8,0) circle [radius=0.9];
\end{scope}

\fill[fill=black] (-1.8,0) circle (1pt);
\node (1) at (-0.7,1) {\footnotesize$\sqrt{d}$};
\draw[-latex] (-1.8,0)--(-1,0.73);

\node (2) at (-2.075,-0.1) {\footnotesize$\mu_1$};

\fill[pattern=north west lines,even odd rule,opacity=.6,draw] (-.5,0.5) circle (1.15) (-.5,0.5) circle (0.9);

\fill[fill=black] (-.5,0.5) circle (1pt);
\draw[-latex] (-.5,0.5)--(-1.28,-0.25);
\node (1) at (-1.65,-0.5) {\footnotesize$\sqrt{d}$};
    \node (2) at (-.2,0.5) {\footnotesize$\mu_2$};
    
\end{tikzpicture}
\nopagebreak[4]
\begin{fig}\label{FIG-6-3}Two high-dimensional Gaussians with distance $\|\mu_1-\mu_2\|\leqslant2\sqrt{d}$.\end{fig}
\end{center}

Notice that also in this second scenario $\Delta\rightarrow\infty$ must hold in order to avoid that for large $d$ the the two annuli become almost coincident. The picture suggests nevertheless that this can be achieved also for $\Delta$'s that grow slower than $2\sqrt{d}$ and experiments corroborate this idea. For Figure \ref{SEP-FIG-1} below we considered two Gaussians in dimension $d=1000$, sampled 100 points from each of them, and plotted the distribution of the sample's mutual distances.
\begin{center}
\begin{tikzpicture}

	\begin{axis}
[
axis line style={thick, shorten >=-5pt, shorten <=-2pt},
y=0.075pt,
axis y line=left,
axis x line=middle,
axis line style={->},
no markers,
tick align=outside,
major tick length=2pt,
ymin=0.0,
ymax=1050,
ytick={0,200,...,1000},
xmin=38,xmax=61,
xtick={40,42, ..., 60},
every tick label/.append style={font=\tiny},
xlabel=\small $\|x-y\|$,
every axis x label/.style={
    at={(ticklabel* cs:1.07)},
    anchor=west,
},
every axis y label/.style={
    at={(ticklabel* cs:1.2)},
    anchor=north,
},
]
	% use TeX as calculator:
\addplot[ mark=none,fill=black, 
                    fill opacity=0.05] coordinates {
(40.0,0)
(40.21,0)
(40.42,0)
(40.63,1)
(40.84,0)
(41.05,2)
(41.26,4)
(41.47,7)
(41.68,6)
(41.89,29)
(42.1,27)
(42.31,51)
(42.52,104)
(42.73,130)
(42.94,228)
(43.15,288)
(43.36,404)
(43.57,502)
(43.78,558)
(43.99,696)
(44.2,787)
(44.41,832)
(44.62,864)
(44.83,769)
(45.04,697)
(45.25,632)
(45.46,562)
(45.67,450)
(45.88,401)
(46.09,282)
(46.3,190)
(46.51,145)
(46.72,100)
(46.93,55)
(47.14,34)
(47.35,32)
(47.56,17)
(47.77,7)
(47.98,4)
(48.19,2)
(48.4,1)
(48.61,0)
(48.82,0)
(49.03,0)
(49.24,0)
(49.45,0)
(49.66,0)
(49.87,0)
(50.08,0)
(50.29,0)
(50.5,0)
(50.71,0)
(50.92,2)
(51.13,6)
(51.34,5)
(51.55,7)
(51.76,10)
(51.97,29)
(52.18,46)
(52.39,83)
(52.6,98)
(52.81,171)
(53.02,237)
(53.23,341)
(53.44,404)
(53.65,509)
(53.86,626)
(54.07,687)
(54.28,732)
(54.49,769)
(54.7,755)
(54.91,736)
(55.12,733)
(55.33,613)
(55.54,505)
(55.75,442)
(55.96,394)
(56.17,282)
(56.38,222)
(56.59,152)
(56.8,126)
(57.01,86)
(57.22,64)
(57.43,40)
(57.64,39)
(57.85,18)
(58.06,15)
(58.27,5)
(58.48,5)
(58.69,3)
(58.9,1)
(59.11,0)
(59.32,0)
(59.53,1)
(59.74,0)
(59.95,0)
(60.16,0)
(60.37,0)
(60.58,0)
(60.79,0)
	};
%\addplot [domain=100:960, samples=101,dashed]{sqrt(100/x)};
	\end{axis}
\end{tikzpicture}\nopagebreak[4]\begin{fig}\label{SEP-FIG-1}Mutual distances of points in a set consisting of 100 points sampled from $\mathcal{N}(0,1,\mathbb{R}^{1000})$ and 100 points sampled from $\mathcal{N}(\mathds{1},1,\mathbb{R}^{1000})$, where $\mathds{1}=(1,1,\dots)$.\end{fig}
\end{center}

Indeed, the distances of points that belong to the same Gaussian accumulate around $\sqrt{2d}\approx44.72$. This suggests that the other `bump' in the picture belongs to those pairs of points which come from different Gaussians. In our example we have a distance of $\Delta\approx31.62$ between the Gaussians' centers and $2\sqrt{d}\approx63.25$. This shows that we are in the situation of Figure \ref{FIG-6-3}, i.e., the Gaussians overlap. The distribution of mutual distances suggest that it nevertheless should be possible to disentangle them with high accuracy. Indeed, if we pick for a start one point at random and then take its 100 nearest neighbors, Figure \ref{SEP-FIG-1} suggests that our odds are good that these 100 points will belong to one of the Gaussians and the rest of the points to the other.

\smallskip

Our aim now is to make the above formal. We start by computing the expected values for the distance between points that belong to different Gaussians.

\begin{thm}\label{EV-THM} Let $\mu_1$, $\mu_2\in\mathbb{R}^d$ and $\Delta:=\|\mu_1-\mu_2\|$. Let $X_1\sim\mathcal{N}(\mu_1,1,\mathbb{R}^d)$ and $X_2\sim\mathcal{N}(\mu_2,1,\mathbb{R}^d)$ be random vectors.\vspace{1pt}
\begin{compactitem}

\item[(i)] $\forall\:d\geqslant1\colon\bigl|\E\bigl(\|X_1-X_2\|-\sqrt{\Delta^2+2d}\hspace{2pt}\bigr)\bigr|\leqslant \medfrac{3d+8\Delta^2+\sqrt{24}\Delta d^{1/2}}{2(\Delta^2+2d)^{3/2}}$.\vspace{2pt}

\item[(ii)] $\forall\:d\geqslant1\colon\V(\|X_1-X_2\|)\leqslant \Bigl(\medfrac{3d+8\Delta^2+\sqrt{24}\Delta d^{1/2}}{2(\Delta^2+2d)^{3/2}}\Bigr)^2 + \medfrac{3d+8\Delta^2+\sqrt{24}\Delta d^{1/2}}{\Delta^2+2d}$.
\end{compactitem}
\end{thm}\vspace{-6pt}
\begin{proof} (i) Let $X_1$, $X_2\colon\Omega\rightarrow\mathbb{R}^d$ be as above. We define the auxiliary random vector $\tilde{X}_1:=\mu_2+X_1-\mu_1$. Then $\tilde{X}_1\sim\mathcal{N}(\mu_2,1,\mathbb{R}^d)$ holds. Moreover, $\tilde{X}_1$ and $X_2$ are independent. For a fixed $\omega\in\Omega$ we get the following picture for the values $x_1:=X_1(\omega)$, $x_2:=X_2(\omega)$ and $\tilde{x}_1:=\tilde{X}_1(\omega)\in\mathbb{R}^d$. Firstly, in view of Theorem \ref{EXP-VAR-NORM} we may assume that $x_1$ will be located near the surface of a ball with radius $\sqrt{d}$ centered at $\mu_1$ and that $x_2$ and $\tilde{x}_1$ will both be located near the surface of a ball with radius $\sqrt{d}$ centered at $\mu_2$. By definition of $\tilde{x}_1$ the lines connecting $\mu_1$ with $x_1$ and $\mu_2$ with $\tilde{x}_1$, respectively, will be parallel. Moreover, if we think of a coordinate system in which $\mu_2$ is the origin, then we may expect that $\tilde{x}_1$ and $x_2$ have a distance of approximately $\sqrt{2d}$, compare Theorem \ref{EXP-VAR-DIST}. Finally, we think of the line going from $\mu_2$ to $\mu_1$ marking the north pole of $\B(\mu_2,\sqrt{d}\hspace{1pt})$. Then Theorem \ref{EXP-VAR-ANGLE} suggests that $\tilde{x}_1$ and $x_2$ will be located close to the equator of $\B(\mu_2,\sqrt{d}\hspace{1pt})$, since the scalar products $\langle{}\tilde{x}_1-\mu_2,\mu_1-\mu_2\rangle{}$ and $\langle{}x_2-\mu_2,\mu_1-\mu_2\rangle{}$ both will be close to zero. In the triangle given by all three points we thus may expect that the angle $\theta$ will be close to $90^{\circ}$ which implies that $\|x_1-x_2\|^2\approx\Delta^2+2d$ should hold.
\begin{center}
\includegraphics{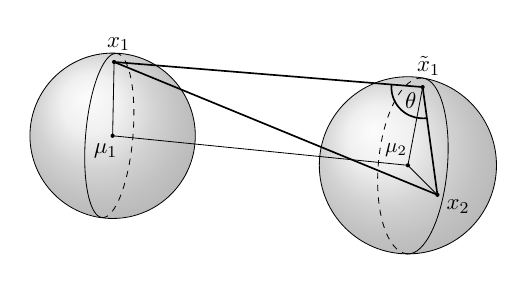}
\nopagebreak[4]
\vspace{-5pt}
\begin{fig}\label{FIG-6-30}Points sampled from different unit Gaussians lead to\\an almost orthogonal triangle with hypothenuse between $x_1$ and $x_2$.\end{fig}
\end{center}

Following the above intuition, we will now formally use the cosine law and our results from Chapter \ref{Ch-INTRO} to compute the expectation. We put $\xi:=\mu_1-\mu_2$ and start with
\begin{equation}\label{COS-LAW}
\begin{array}{rl}\vspace{3pt}
\|X_1-X_2\|^2\hspace{-7pt}\,& = \|X_2-\tilde{X}_1\|^2+\|X_1-\tilde{X}_1\|^2 - 2\|X_2-\tilde{X}_1\|\|\tilde{X}_1-X_1\|\cos(\theta)\\\vspace{3pt}
& = \|X_2-\tilde{X}_1\|^2+\|\mu_1-\mu_2\|^2 - 2\langle{}X_2-\tilde{X}_1,X_1-\tilde{X}_1\rangle{}\\\vspace{3pt}
%& = \|X_2-\tilde{X}_1\|^2+\Delta^2 \\
%&\hspace{50pt}-2\langle{}X_2-\mu_2-X_1+\mu_1, X_1-\mu_2-X_1+\mu_1\rangle{}\\
& = \|X_2-\tilde{X}_1\|^2+\Delta^2 -2(\langle{}X_2-\mu_2,\xi\rangle{}+ \langle{}X_1-\mu_1,-\xi\rangle{})=:(\circ)
\end{array}
\end{equation}
where we know $\E(\|X_2-\tilde{X}_1\|^2)=2d$ since $\tilde{X}_1\sim\mathcal{N}(\mu_2,1,\mathbb{R}^d)$ by \eqref{EXP-NORM-SQUARED}. Moreover, by Theorem \ref{EXP-VAR-ANGLE} we get $\E(\langle{}X_2-\mu_2,\xi\rangle{})=\E(\langle{}X_1-\mu_1,-\xi\rangle{})=0$ which leads to
$$
\E(\|X_1-X_2\|^2)=2d + \Delta^2 + 2(0 - 0) = \Delta+2d.
$$
Now we estimate the variance of $\|X_1-X_2\|^2$ but we have to be careful as the three non-constant expressions in \eqref{COS-LAW} are not independent. Indeed, from the second line of \eqref{COS-LAW} we get
\begin{align*}
\V(\|X_1-X_2\|^2) & = \V(\|X_2-\tilde{X}_1\|^2 + \Delta^2 - 2\langle{}X_2-\tilde{X}_1,X_1-\tilde{X}_1\rangle{})\\
& = \V(\|X_2-\tilde{X}_1\|^2) + 4\V(\langle{}X_2-\tilde{X}_1,X_1-\tilde{X}_1\rangle{})\\
&  \hspace{30pt} + \Cov(\|X_2-\tilde{X}_1\|^2,-2\langle{}X_2-\tilde{X}_1,X_1-\tilde{X}_1\rangle{})\\
& \leqslant \V(\|X_2-\tilde{X}_1\|^2) + 4\V(\langle{}X_2-\tilde{X}_1,X_1-\tilde{X}_1\rangle{})\\
&  \hspace{30pt} + \bigl(\V(\|X_2-\tilde{X}_1\|^2)\cdot\V(-2\langle{}X_2-\tilde{X}_1,X_1-\tilde{X}_1\rangle{})\bigr)^{1/2}=:(*)
\end{align*}
where we know, see Problem \ref{PROB-2}, that $\V(\|X_2-\tilde{X}_1\|^2)\leqslant 3d$ holds. For the second term under the root we compute
\begin{align*}
\V(\langle{}X_2-\tilde{X}_1,X_1-\tilde{X}_1\rangle{}) & = \V(\langle{}X_2-\mu_2,\mu_1-\mu_2\rangle{} + \langle{}X_1-\mu_1,\mu_2-\mu_1\rangle{} )\\
& = \V(\langle{}X_2-\mu_2,\mu_1-\mu_2\rangle{}) + \V(\langle{}X_1-\mu_1,\mu_2-\mu_1\rangle{} )\\
& \leqslant 2\|\mu_1-\mu_2\|^2
\end{align*}
where we use $Z:=X_i-\mu_i\sim\mathcal{N}(0,1,\mathbb{R}^d)$, $\xi:=\pm(\mu_1-\mu_2)$ together with Theorem \ref{EXP-VAR-ANGLE} to get $\V(\langle{}Z,\xi\rangle{}) = \|\xi\|^2$. This allows us to continue the estimate of the variance as follows
$$
(*) \leqslant 3d+8\|\mu_1-\mu_2\|^2 + (3d\cdot8\|\mu_1-\mu_2\|^2)^{1/2}= 3d+8\Delta^2+\sqrt{24}\Delta d^{1/2}.
$$
Now we adapt the method used in the proof of Theorem \ref{EXP-VAR-NORM} and split
$$
\|X_1-X_2\|-\sqrt{\Delta^2+2d} = \medfrac{\|X_1-X_2\|^2-(\Delta^2+2d)}{2\sqrt{\Delta^2+2d}} - \medfrac{(\|X_1-X_2\|^2-(\Delta^2+2d))^2}{2\sqrt{\Delta^2+2d}(\|X_1-X_2\|+\sqrt{\Delta^2+2d})^2}
$$
where we know that the expectation of the first term is zero. The expectation of the second term we estimate by
$$
0\leqslant \E\Bigl(\medfrac{(\|X_1-X_2\|^2-(\Delta^2+2d))^2}{2\sqrt{\Delta^2+2d}(\|X_1-X_2\|+\sqrt{\Delta^2+2d})^2}\Bigr)\leqslant \medfrac{\V(\|X_1-X_2\|^2)}{2(\Delta^2+2d)^{3/2} }\leqslant \medfrac{3d+8\Delta^2+\sqrt{24}\Delta d^{1/2}}{2(\Delta^2+2d)^{3/2} }
$$
where we used our previous estimate for $\V(\|X_1-X_2\|^2)$. This establishes (i).

\smallskip

(ii) Finally we get
\begin{align*}
\V(\|X_1-X_2\|) &= \bigl|\E\bigl((\|X_1-X_2\|)^2\bigr) - \bigl(\E(\|X_1-X_2\|)\bigr)^2\bigr|\\
& = \bigl|(\Delta^2+2d) - \bigl[\E\bigl(\|X_1-X_2\|-\sqrt{\Delta^2+2d}\hspace{2pt}\bigr)+\sqrt{\Delta^2+2d}\hspace{2pt}\bigr]^2\bigr|\\
& = \bigl|(\Delta^2+2d) -\bigl[\bigl(\E\bigl(\|X_1-X_2\|-\sqrt{\Delta^2+2d}\hspace{2pt}\bigr)\bigr)^2\\
&\hspace{50pt}+2\E\bigl(\|X_1-X_2\|-\sqrt{\Delta^2+2d}\hspace{2pt}\bigr)\sqrt{\Delta^2+2d}+(\Delta^2+2d)\bigr]\bigr|\\
& \leqslant \bigl|\E\bigl(\|X_1-X_2\|-\sqrt{\Delta^2+2d}\hspace{2pt}\bigr)\bigr|^2\\
&\hspace{50pt} + 2\bigr|\E\bigl(\|X_1-X_2\|-\sqrt{\Delta^2+2d}\hspace{2pt}\bigr)\bigr|\sqrt{\Delta^2+2d}
\end{align*}
which shows (ii) by employing twice the estimate from (i).
\end{proof}

In the following results we understand $\Delta=\Delta(d)$ as a function of the dimension. To keep the notation short we will however drop the `of $d$' most of the time.

\begin{cor}\label{EV-THM-COR} Let $\mu_1$, $\mu_2\in\mathbb{R}^d$ and $\Delta:=\|\mu_1-\mu_2\|$. Let $X_1\sim\mathcal{N}(\mu_1,1,\mathbb{R}^d)$ and $X_2\sim\mathcal{N}(\mu_2,1,\mathbb{R}^d)$ be random variables. Assume either that there exists $d_0\geqslant1$ such that $\Delta\geqslant2\sqrt{d}$ holds for all $d\geqslant d_0$ or that there exists $d_0\geqslant1$ such that $\Delta<2\sqrt{d}$ holds for all $d\geqslant d_0$. Then
$$
\lim_{d\rightarrow\infty}\E\bigl(\|X_1-X_2\|-\sqrt{\Delta^2+2d}\hspace{2pt}\bigr)=0 \;\;\text{ and }\;\; \limsup_{d\rightarrow\infty}\V\bigl(\|X_1-X_2\|\bigr)\leqslant23.
$$
\end{cor}
\begin{proof} We start with the first scenario, assume $\Delta\geqslant2\sqrt{d}$, i.e., $d\leqslant\Delta^2/4$, and estimate
$$
\medfrac{3d+8\Delta^2+\sqrt{24}\Delta d^{1/2}}{2(\Delta^2+2d)^{3/2}} \leqslant \medfrac{(3/4)\Delta^2+8\Delta^2+\sqrt{6}\Delta^2}{2(\Delta^2)^{3/2}}\leqslant\medfrac{3/4+8+\sqrt{6}}{4\sqrt{d}}\xrightarrow{d\rightarrow\infty}0.
$$
For the second scenario assume $\Delta<2\sqrt{d}$. Then
$$
\medfrac{3d+8\Delta^2+\sqrt{24}\Delta d^{1/2}}{2(\Delta^2+2d)^{3/2}} \leqslant \medfrac{3d+32d+2\sqrt{24}d}{2(2d)^{3/2}} = \medfrac{17+\sqrt{24}}{2\sqrt{2d}}\xrightarrow{d\rightarrow\infty}0
$$
holds. This shows that the expectation tends to zero. For the variance we get from the above that the first term will tend to zero and it thus remains to bound the second. In the first scenario we estimate
$$
\medfrac{3d+8\Delta^2+\sqrt{24}\Delta d^{1/2}}{\Delta^2+2d}\leqslant\medfrac{(3/4+8+\sqrt{6})\Delta^2}{\Delta^2+2d}\leqslant\medfrac{12}{1+(2d/\Delta^2)}\leqslant 12
$$
and in the second
$$
\medfrac{3d+8\Delta^2+\sqrt{24}\Delta d^{1/2}}{\Delta^2+2d} \leqslant\medfrac{45d}{\Delta^2+2d}= \medfrac{45}{(\Delta^2/d)+2}\leqslant \medfrac{45}{2}
$$
which finishes the proof.
\end{proof}

The above suggests that if we pick two points $x$ and $y$ from Gaussians with different centers at random, then we have to expect that $\|x-y\|\approx\sqrt{\Delta^2+2d}$ holds for large $d$ and that the distribution does not spread out when $d$ increases. On the other hand we know already from Chapter \ref{Ch-INTRO} that $\|x-y\|\approx\sqrt{2d}$ for large $d$ provided that $x$ and $y$ are sampled from the same Gaussian. This implies that the question if the two Gaussians can be separated, requires that the distance $\sqrt{\Delta^2+2d}-\sqrt{2d}$ between the two expectational values is large enough. For the two scenarios covered by Corollary \ref{EV-THM-COR} we get the following.

\begin{lem}\label{delta-LEM} Let $\Delta\colon\mathbb{N}\rightarrow(0,\infty)$ and $\delta\colon\mathbb{N}\rightarrow(0,\infty)$, $\delta(d):=\sqrt{\Delta^2+2d}-\sqrt{2d}$.
\begin{compactitem}\vspace{2pt}

\item[(i)] If there exists $d_0\geqslant1$ such that $\Delta\geqslant2\sqrt{d}$ holds for all $d\geqslant d_0$, then $\delta(d)\geqslant\sqrt{d}$ holds for $d\geqslant d_0$.\vspace{4pt}

\item[(ii)] If there exists $d_0\geqslant1$, $c>0$ and $\alpha\geqslant0$ such that $cd^{\hspace{1pt}1/4+\alpha}\leqslant\Delta<2\sqrt{d}$ holds for all $d\geqslant d_0$, then $\delta(d)\geqslant (c^2/4)d^{\hspace{1pt}2\alpha}$ holds for $d\geqslant d_0$.

\end{compactitem}
\end{lem}

\begin{proof} (i) In this case we compute
$$
\delta(d)=\sqrt{\Delta^2+2d}-\sqrt{2d} \geqslant \sqrt{4d+2d}-\sqrt{2d} = (\sqrt{6}-\sqrt{2})\sqrt{d}\geqslant\sqrt{d}.
$$
(ii) Using $\Delta^2\leqslant4d$ and $\Delta>cd^{\hspace{1pt}1/4+\alpha}$ we get
$$
\delta(d)=\medfrac{\Delta^2}{\sqrt{\Delta^2+2d}+\sqrt{2d}}\geqslant\medfrac{(cd^{\hspace{1pt}1/4+\alpha})^2}{\sqrt{6d}+\sqrt{2d}}\geqslant \medfrac{c^2d^{\hspace{1pt}2\alpha}}{(\sqrt{6}+\sqrt{2})}\geqslant \medfrac{c^2}{4}\cdot d^{\hspace{1pt}2\alpha}
$$
as claimed.
\end{proof}

We will see later how Lemma \ref{delta-LEM} outlines for the situation sketched in Figure \ref{FIG-6-3}, how fast $\Delta$ needs to grow in order to separate two Gaussians. Before we discuss this, we need however the following explicit probability estimate complementing our results in Theorem \ref{EV-THM} and Corollary \ref{EV-THM-COR} on expectation and variance.

\begin{thm}\label{SEP-THM-2} Let $\mu_1$, $\mu_2\in\mathbb{R}^d$ and let $\Delta:=\|\mu_1-\mu_2\|>0$. Let $x_1$ and $x_2$ be drawn at random with respect to Gaussian distribution with unit variance and means $\mu_1$ and $\mu_2$, respectively. Then 
$$
\P\bigl[\hspace{1pt}\bigl|\|x_1-x_2\|-\sqrt{\Delta^2+2d}\hspace{2pt}\bigr|\geqslant\epsilon\hspace{1pt}\bigr] \leqslant \medfrac{54}{\epsilon}+\medfrac{8}{\sqrt{d}}+\medfrac{20}{\epsilon\sqrt{d}}
$$
holds for $\epsilon>54$ and $d\geqslant\epsilon^2/9$.
\end{thm}
\begin{proof} Formally, we consider random variables $X_i\sim\mathcal{N}(\mu_i,1,\mathbb{R}^d)$ for $i=1,2$ and we estimate 
$$
P:=\P\bigl[\bigl|\|X_1-X_2\|-\sqrt{\Delta^2+2d}\hspace{2pt}\bigr|\geqslant\epsilon\bigr]
$$
by the same trick that we used already in Theorem \ref{GA-THM} and Theorem \ref{THM-distance-of-two-points-sqrt-2d}, namely by multiplying the left side of estimate in the square brackets by $\|X_1-X_2\|+\sqrt{\Delta^2+2d}$ and the right side by $\sqrt{\Delta^2+2d}$. This way we get
$$
P \leqslant \P\bigl[\hspace{1pt}\bigl|\|X_1-X_2\|^2-(\Delta^2+2d)\bigr|\geqslant\epsilon\sqrt{\Delta^2+2d}\hspace{2pt}\bigr]=:(\circ).
$$
Now we introduce the auxiliary random variable $\tilde{X}_1:=\mu_2+X_1-\mu_1\sim\mathcal{N}(\mu_2,1,\mathbb{R}^d)$ and formula \eqref{COS-LAW} that we derived from the cosine law to get
\begin{align*}
(\circ)& = \P\bigl[\hspace{1pt}\bigl|\|X_1-X_2\|^2-(\Delta^2+2d)\bigr|\geqslant\epsilon\sqrt{\Delta^2+2d}\hspace{2pt}\bigr]\\
& = 1 - \P\bigl[-\epsilon\sqrt{\Delta^2+2d}\leqslant\|X_2-\tilde{X}_1\|^2+\Delta^2-2\bigl(\langle{}X_2-\mu_2,\mu_1-\mu_2\rangle{}\\
& \hspace{50pt} +\langle{}X_1-\mu_1,\mu_2-\mu_1\rangle{}\bigr)-\Delta^2-2d\leqslant\epsilon\sqrt{\Delta^2+2d}\hspace{2pt}\bigr]\\
& \leqslant 1 - \P\bigl[-\medfrac{\epsilon\sqrt{\Delta^2+2d}}{3}\leqslant\|X_2-\tilde{X}_1\|^2-2d\leqslant\medfrac{\epsilon\sqrt{\Delta^2+2d}}{3}\hspace{2pt}\bigr]\\
& \hspace{50pt}\cdot\P\bigl[-\medfrac{\epsilon\sqrt{\Delta^2+2d}}{3}\leqslant-2\langle{}X_2-\mu_2,\mu_1-\mu_2\rangle{}\leqslant\medfrac{\epsilon\sqrt{\Delta^2+2d}}{3}\hspace{2pt}\bigr]\\
& \hspace{50pt}\cdot\P\bigl[-\medfrac{\epsilon\sqrt{\Delta^2+2d}}{3}\leqslant-2\langle{}X_1-\mu_1,\mu_2-\mu_1\rangle{}\leqslant\medfrac{\epsilon\sqrt{\Delta^2+2d}}{3}\hspace{2pt}\bigr]\\
& \leqslant 1 - \P\bigl[\hspace{1pt}\bigl|\|X_2-\tilde{X}_1\|^2-2d\hspace{1pt}\bigr|\leqslant\medfrac{\epsilon\sqrt{2d}}{3}\hspace{2pt}\bigr]\cdot\P\bigl[\hspace{1pt}|\langle{}Z,\xi\rangle{}|\leqslant\medfrac{\epsilon\sqrt{\Delta^2}}{6}\hspace{2pt}\bigr]^2 =:(*)
\end{align*}
where we use $Z:=X_i-\mu_i\sim\mathcal{N}(0,1,\mathbb{R}^d)$ and $\xi:=\mu_1-\mu_2$. In order to continue we need to estimate both probabilities in $(*)$ from below. From Corollary \ref{COR-distance-of-two-points-sqrt-2d} we get for $\epsilon/3>18$ and $d\geqslant(\epsilon/3)^2$ the estimate
$$
\P\bigl[\hspace{1pt}\bigl|\|X_2-\tilde{X}_1\|^2-2d\hspace{1pt}\bigr|\leqslant\medfrac{\epsilon\sqrt{2d}}{3}\hspace{2pt}\bigr]\geqslant1-\bigl(\medfrac{18}{\epsilon/3}+\medfrac{8}{\sqrt{d}}\bigr)=1-\bigl(\medfrac{54}{\epsilon}+\medfrac{8}{\sqrt{d}}\bigr).
$$
Using Corollary \ref{COR-ANGLE-ESTIM} we get
$$
\P\bigl[\hspace{1pt}|\langle{}Z,\xi\rangle{}|\leqslant\medfrac{\epsilon\sqrt{\Delta^2}}{6}\hspace{1pt}\bigr]\geqslant1-\medfrac{4}{\sqrt{2\pi}}\medfrac{\|\xi\|}{(\epsilon\Delta/6)\sqrt{d}} = 1-\medfrac{24}{\sqrt{2\pi}\epsilon\sqrt{d}}\geqslant1-\medfrac{10}{\epsilon\sqrt{d}}.
$$
Combining both leads to
\begin{align*}
(*) & = 1 - \P\bigl[\hspace{1pt}\bigl|\|X_2-\tilde{X}_1\|^2-2d\hspace{1pt}\bigr|\leqslant\medfrac{\epsilon\sqrt{2d}}{3}\hspace{2pt}\bigr]\cdot\P\bigl[\hspace{1pt}\bigl|\langle{}Z,\xi\rangle{}\bigr|\leqslant\medfrac{\epsilon\sqrt{\Delta^2}}{6}\hspace{2pt}\bigr]^2\\
&\leqslant 1-\Bigl[\hspace{1pt}1-\bigl(\medfrac{54}{\epsilon}+\medfrac{8}{\sqrt{d}}\bigr)\Bigr]\Bigl[1-2\cdot\medfrac{10}{\epsilon\sqrt{d}}+\bigl(\medfrac{10}{\epsilon\sqrt{d}}\bigr)^2\hspace{1pt}\Bigr]\\
&\leqslant 1-\Bigl[\hspace{1pt}1-\bigl(\medfrac{54}{\epsilon}+\medfrac{8}{\sqrt{d}}\bigr)\Bigr]\Bigl[1-\medfrac{20}{\epsilon\sqrt{d}}\hspace{1pt}\Bigr]\\
&\leqslant 1- \Bigl[\hspace{1pt}1-\bigl(\medfrac{54}{\epsilon}+\medfrac{8}{\sqrt{d}}\bigr)-\medfrac{20}{\epsilon\sqrt{d}}\hspace{1pt}\Bigr]\\
&=\medfrac{54}{\epsilon}+\medfrac{8}{\sqrt{d}}+\medfrac{20}{\epsilon\sqrt{d}}
\end{align*}
which finishes the proof.
\end{proof}

The following is the core result on separability as we can use it to see that points which have `small' mutual distance belong to the same Gaussian with high probability.

\begin{thm}\label{SEP-THM}(Separation Theorem) Let $\mu_1$, $\mu_2\in\mathbb{R}^d$ and put $\Delta:=\|\mu_1-\mu_2\|$. Let $\epsilon_1>18$ and  $\epsilon_2>54$ be constant and assume that there is $d_0\geqslant1$ such that $\sqrt{\Delta^2+2d}-\sqrt{2d}>\epsilon_1+\epsilon_2$ holds for $d\geqslant d_0$. Let $X_i\sim\mathcal{N}(\mu_i,1,\mathbb{R}^d)$ and $S_i:=\{x_i^{\scriptscriptstyle(1)},\dots,x_i^{\scriptscriptstyle(n)}\}$ be samples of $X_i$ for $i=1,2$ such that $S_1\cap S_2=\emptyset$. Let $x,y\in S_1\cup S_2$. Then
$$
\liminf_{d\rightarrow\infty}\;\P\bigl[\hspace{1pt}x,\,y \text{ come from same Gaussian}\:\big|\:|\|x-y\|-\sqrt{2d}\hspace{2pt}|<\epsilon_1 \hspace{1pt}\bigr]\geqslant\medfrac{1-18/\epsilon_1}{1+54/\epsilon_2}
$$
holds.
\end{thm}

\begin{proof} On the basis of our prior experiments, see Figure \ref{SEP-FIG-1}, and in view of Theorem \ref{EXP-VAR-DIST} and Corollary \ref{EV-THM-COR} we expect that the mutual distances will accumulate around the values $\sqrt{2d}$ and $\sqrt{\Delta^2+2d}$. Our assumptions imply $\sqrt{2d}+\epsilon_1<\sqrt{\Delta^2+2d}-\epsilon_2$ which guarantees that $\Bbar(\sqrt{2d},\epsilon_1)$ and $\Bbar(\sqrt{\Delta^2+2d},\epsilon_2)$ have empty intersection, i.e., there is some space between the two `bumps' sketched in Figure \ref{SEP-FIG-2} below.
\begin{center}
\begin{tikzpicture}

	\begin{axis}
[
axis line style={thick, shorten >=-5pt, shorten <=-2pt},
y=60pt,
axis y line=left,
axis x line=middle,
axis line style={->},
no markers,
tick align=outside,
major tick length=2pt,
ymin=0.0,
ymax=1.1,
ytick={0},
yticklabels={},
xmin=38,xmax=61,
xtick={44.72-2, 44.72,44.72+2,54.79-2, 54.79,54.79+2},
xticklabels={},
every tick label/.append style={font=\tiny},
xlabel=\small $\|x-y\|$,
every axis x label/.style={
    at={(ticklabel* cs:1.07)},
    anchor=west,
},
every axis y label/.style={
    at={(ticklabel* cs:1.2)},
    anchor=north,
},
]

%\addplot [domain=38:61, samples=101,dashed]{1};
\addplot [domain=38:61, samples=150,fill=black,fill opacity=0.05]{0.9*exp(-(x-44.72)^2/1.5)};
\addplot [domain=38:61, samples=150,fill=black,fill opacity=0.05]{exp(-(x-54.79)^2/1.5)};
\end{axis}

\begin{picture}(0,0)
%\draw [line width=0.25mm, opacity=0.4] (0,-0.025) -- (1.41,-0.025);
%\draw [line width=0.25mm, opacity=0.4] (2.595,-0.025) -- (4.41,-0.025);
%\draw [line width=0.25mm, opacity=0.4] (5.598,-0.025) -- (6.98,-0.025);
\draw [line width=0.25mm] (1.41,-0.025) -- (2.595,-0.025);
\draw [line width=0.25mm] (4.41,-0.025) -- (5.598,-0.025);
\node at (2,-0.3) {\tiny$\Bbar(\sqrt{2d},\epsilon_1)$};
\node at (5,-0.3) {\tiny$\Bbar(\sqrt{\hspace{-1.5pt}\Delta^2\hspace{-1.5pt}+\hspace{-1.5pt}2d},\epsilon_2)$};
%\node at (0.5,-0.3) {\textcolor{black!60!white}{\tiny$T$}};
%\node at (3.45,-0.3) {\textcolor{black!60!white}{\tiny$T$}};
%\node at (6.45,-0.3) {\textcolor{black!60!white}{\tiny$T$}};
\end{picture}
\end{tikzpicture}\nopagebreak[4]\begin{fig}\label{SEP-FIG-2}Picture of the distribution of mutual distances\\ accumulating around $\sqrt{2d}$ and $\sqrt{\Delta^2+2d}$.\end{fig}
\end{center}

To keep notation short let us abbreviate $x\sim y$ for the case that $x$ and $y$ come from the same Gaussian, i.e., if $x,\,y\in S_1$ or $x,\,y\in S_2$. Consequently, $x\not\sim y$ means $x\in S_1, y\in S_2$ or $x\in S_2,\,y\in S_1$. We use Bayes Theorem to compute the conditional probability
$$
\P\bigl[\hspace{1pt}x\sim y\:\big|\:\bigl|\|x-y\|-\sqrt{2d}\hspace{2pt}\bigr|<\epsilon_1\hspace{1pt}\bigr] = \frac{\P\bigl[\hspace{1pt}\bigl|\|x-y\|-\sqrt{2d}\hspace{2pt}\bigr|<\epsilon_1 \:\big|\:x\sim y\hspace{1pt}\bigr]\cdot\P[\hspace{1pt}x\sim y\hspace{1pt}]}{\P\bigl[\hspace{1pt}\bigl|\|x-y\|-\sqrt{2d}\hspace{2pt}\bigr|<\epsilon_1\hspace{1pt}\bigr]}=:(\circ)
$$
and continue with the denominator as follows
\begin{align*}
\P\bigl[\hspace{1pt}\bigl|\|x-y\|-\sqrt{2d}\hspace{1pt}\bigr|<\epsilon_1\hspace{1pt}\bigr] & = \P\bigl[\hspace{1pt}\bigl|\|x-y\|-\sqrt{2d}\hspace{1.5pt}\bigr|<\epsilon_1\:\big|\:x\sim y\bigr]\cdot\P[x\sim y]\\
& \hspace{20pt} + \P\bigl[\hspace{1pt}\bigl|\|x-y\|-\sqrt{2d}\hspace{2pt}\bigr|<\epsilon_1\:\big|\:x\not\sim y\bigr]\cdot\P[x\not\sim y]\\
& \leqslant 1\cdot\medfrac{1}{2}+\bigl(1-\P\bigl[\hspace{1pt}\bigl|\|x-y\|-\sqrt{2d}\hspace{2pt}\bigr|\geqslant\epsilon_1\:\big|\:x\not\sim y\bigr]\bigr)\cdot\medfrac{1}{2}\\
& \leqslant\medfrac{1}{2}\cdot\Bigl(1+1-\P\bigl[\hspace{1pt}\bigl|\|x-y\|-\sqrt{\Delta^2+2d}\hspace{2pt}\bigr|<\epsilon_2\:\big|\:x\not\sim y\bigr]\Bigr)\\
& \leqslant\medfrac{1}{2}\cdot\Bigl(1+\P\bigl[\hspace{1pt}\bigl|\|x-y\|-\sqrt{\Delta^2+2d}\hspace{2pt}\bigr|\geqslant\epsilon_2\:\big|\:x\not\sim y\bigr]\Bigr)\\
& \leqslant\medfrac{1}{2}\cdot\Bigl(1+\medfrac{54}{\epsilon_2}+\medfrac{8}{\sqrt{d}}+\medfrac{20}{\epsilon_2\sqrt{d}}\Bigr)\\
\end{align*}
where we employed $\P[x\not\sim y]=\medfrac{1}{2}$, the fact that $\Bbar(\sqrt{2d},\epsilon_1)\cap\Bbar(\sqrt{\Delta^2+2d},\epsilon_2)=\emptyset$ and the estimate from Theorem \ref{SEP-THM-2}. Using $\P[x\sim y] =1/2$ and the estimate from Theorem \ref{THM-distance-of-two-points-sqrt-2d} to treat the numerator of $(\circ)$ we get
\begin{equation}\label{SEP-EQ-1}
\P\bigl[\hspace{1pt}x\sim y\;\big|\;\bigl|\|x-y\|-\sqrt{2d}\hspace{2pt}\bigr|<\epsilon_1\bigr] \geqslant \frac{1-\frac{18}{\epsilon_1}-\frac{8}{\sqrt{d}}}{1+\frac{54}{\epsilon_2}+\frac{8}{\sqrt{d}}+\frac{20}{\epsilon_2\sqrt{d}}}
\end{equation}
and letting $d\rightarrow\infty$ finishes the proof.
\end{proof}

Notice that in order to apply Theorem \ref{SEP-THM} we need that $\delta=\sqrt{\Delta^2+2d}-\sqrt{2d}\geqslant\epsilon_1+\epsilon_2>72$ holds. Here, $\delta=\delta(d)$ can for example be large and bounded or even tend to infinity. If the latter is the case, we may in \eqref{SEP-EQ-1} consider $\epsilon_2\rightarrow\infty$ as well. We state concrete examples for this below.

\begin{cor}\label{Ch6-COR} Let $\epsilon_1>18$ be constant, $\epsilon_2=d^{\hspace{1pt}\alpha}$ and $\Delta=2d^{\hspace{1pt}1/4+\alpha}$ with $\alpha>0$. Then the assumptions of Theorem \ref{SEP-THM} are met and we have
$$
\liminf_{d\rightarrow\infty}\;\P\bigl[\hspace{1pt}x,\,y \text{ come from same Gaussian}\:\big|\:|\|x-y\|-\sqrt{2d}\hspace{2pt}|<\epsilon_1 \hspace{1pt}\bigr]\geqslant1-\medfrac{18}{\epsilon_1}.
$$
\end{cor} 
\begin{proof} We first consider the case $\alpha\geqslant1/4$. Then by Lemma \ref{delta-LEM}(i) we get that $\delta\geqslant d^{\hspace{1pt}1/2}>\epsilon_1+d^{\hspace{1pt}1/4}$ holds for large $d$. We may thus proceed as in the proof of the previous theorem and at its end we see, since $\epsilon_2\rightarrow\infty$ for $d\rightarrow\infty$, that the limit inferior is bounded from below by $1-18/\epsilon_1$.

\smallskip

Let now $\alpha<1/4$. Then we apply Lemma \ref{delta-LEM}(ii) and get that $\delta\geqslant (2^2/4)d^{\hspace{1pt}2\alpha}>\epsilon_1+d^{\hspace{1pt}\alpha}$ for large $d$. From this point on we can proceed as in the first part.
\end{proof}

We conclude this chapter by running the separation algorithm suggested by our previous results with the parameters outlined in Corollary \ref{Ch6-COR}. That is, we consider two $d$-dimensional Gaussians whose centers have distance $\Delta=\Delta(d)$ and sample from each of them 100 points. Then we pick one data point at random and label its 100 nearest neighbors with $0$, the rest with $1$. Finally we check how many of the points with the same label came from the same Gaussian. The charts below show the rate of correctly classified points for different choices of $\Delta$.

\begin{center}
\begin{tikzpicture}
\pgfplotsset{scaled x ticks=false}
	\begin{axis}
[
axis line style={thick, shorten >=-5pt, shorten <=-2pt},
y=100pt,
x=0.0095pt,
axis y line=left,
axis x line=middle,
axis line style={->},
no markers,
tick align=outside,
major tick length=2pt,
ymin=0.5,
ymax=1.05,
ytick={0.5,0.6,...,1},
xmin=200,xmax=10000,
xtick={500,3000,6000,9000},
every tick label/.append style={font=\tiny},
xlabel=$\scriptstyle d$,
every axis x label/.style={
    at={(ticklabel* cs:1.07)},
    anchor=west,
},
every axis y label/.style={
    at={(ticklabel* cs:1.2)},
    anchor=north,
},
]
	% use TeX as calculator:
\addplot[ mark=none] coordinates {
(200,0.9116)
(400,0.8375)
(600,0.7855)
(800,0.7568)
(1000,0.7363)
(1200,0.7186)
(1400,0.7168)
(1600,0.6928)
(1800,0.6834)
(2000,0.6739)
(2200,0.6674)
(2400,0.6626)
(2600,0.6518)
(2800,0.6572)
(3000,0.6473)
(3200,0.6441)
(3400,0.6402)
(3600,0.6278)
(3800,0.6246)
(4000,0.6363)
(4200,0.6227)
(4400,0.6271)
(4600,0.6276)
(4800,0.617)
(5000,0.6092)
(5200,0.616)
(5400,0.6101)
(5600,0.6154)
(5800,0.6001)
(6000,0.6085)
(6200,0.6043)
(6400,0.6131)
(6600,0.6094)
(6800,0.5991)
(7000,0.5953)
(7200,0.5945)
(7400,0.5915)
(7600,0.6011)
(7800,0.5933)
(8000,0.5976)
(8200,0.5851)
(8400,0.5952)
(8600,0.59)
(8800,0.5862)
(9000,0.5967)
(9200,0.5841)
(9400,0.5869)
(9600,0.5844)
(9800,0.5852)
(10000,0.5855)
	};
%\addplot [domain=100:960, samples=101,dashed]{sqrt(100/x)};
	\end{axis}\end{tikzpicture}
\begin{tikzpicture}
\pgfplotsset{scaled x ticks=false}
	\begin{axis}
[
axis line style={thick, shorten >=-5pt, shorten <=-2pt},
y=100pt,
x=0.0095pt,
axis y line=left,
axis x line=middle,
axis line style={->},
no markers,
tick align=outside,
major tick length=2pt,
ymin=0.5,
ymax=1.05,
ytick={0.5,0.6,...,1},
xmin=200,xmax=10000,
xtick={500,3000,6000,9000},
every tick label/.append style={font=\tiny},
xlabel=$\scriptstyle d$,
every axis x label/.style={
    at={(ticklabel* cs:1.07)},
    anchor=west,
},
every axis y label/.style={
    at={(ticklabel* cs:1.2)},
    anchor=north,
},
]
	% use TeX as calculator:
\addplot[ mark=none] coordinates {
(200,0.7695)
(400,0.7908)
(600,0.7926)
(800,0.7838)
(1000,0.7841)
(1200,0.7817)
(1400,0.7886)
(1600,0.7878)
(1800,0.7854)
(2000,0.7843)
(2200,0.7859)
(2400,0.7805)
(2600,0.7932)
(2800,0.7868)
(3000,0.7902)
(3200,0.79)
(3400,0.7928)
(3600,0.7929)
(3800,0.7923)
(4000,0.7967)
(4200,0.7872)
(4400,0.7905)
(4600,0.8006)
(4800,0.7874)
(5000,0.7924)
(5200,0.7962)
(5400,0.79)
(5600,0.7934)
(5800,0.791)
(6000,0.7868)
(6200,0.7917)
(6400,0.7889)
(6600,0.7886)
(6800,0.7901)
(7000,0.7881)
(7200,0.7882)
(7400,0.79)
(7600,0.7991)
(7800,0.7912)
(8000,0.8005)
(8200,0.7855)
(8400,0.796)
(8600,0.7889)
(8800,0.7916)
(9000,0.7919)
(9200,0.7882)
(9400,0.7933)
(9600,0.7809)
(9800,0.7919)
(10000,0.7852)
	};
%\addplot [domain=100:960, samples=101,dashed]{sqrt(100/x)};
	\end{axis}\end{tikzpicture}
\begin{tikzpicture}
\pgfplotsset{scaled x ticks=false}
	\begin{axis}
[
axis line style={thick, shorten >=-5pt, shorten <=-2pt},
y=100pt,
x=0.0095pt,
axis y line=left,
axis x line=middle,
axis line style={->},
no markers,
tick align=outside,
major tick length=2pt,
ymin=0.5,
ymax=1.05,
ytick={0.5,0.6,...,1},
xmin=200,xmax=10000,
xtick={500,3000,6000,9000},
every tick label/.append style={font=\tiny},
xlabel=$\scriptstyle d$,
every axis x label/.style={
    at={(ticklabel* cs:1.07)},
    anchor=west,
},
every axis y label/.style={
    at={(ticklabel* cs:1.2)},
    anchor=north,
},
]
	% use TeX as calculator:
\addplot[ mark=none] coordinates {
(200,0.8979)
(400,0.9086)
(600,0.9266)
(800,0.9301)
(1000,0.9345)
(1200,0.9456)
(1400,0.9465)
(1600,0.9458)
(1800,0.9565)
(2000,0.952)
(2200,0.9595)
(2400,0.958)
(2600,0.9612)
(2800,0.9581)
(3000,0.9609)
(3200,0.9651)
(3400,0.9633)
(3600,0.968)
(3800,0.9631)
(4000,0.9659)
(4200,0.9647)
(4400,0.9672)
(4600,0.9682)
(4800,0.9677)
(5000,0.9664)
(5200,0.9716)
(5400,0.9686)
(5600,0.9727)
(5800,0.9734)
(6000,0.9721)
(6200,0.9715)
(6400,0.9717)
(6600,0.9706)
(6800,0.9725)
(7000,0.9738)
(7200,0.9749)
(7400,0.9744)
(7600,0.9737)
(7800,0.975)
(8000,0.9752)
(8200,0.9751)
(8400,0.9745)
(8600,0.9758)
(8800,0.9771)
(9000,0.9775)
(9200,0.9778)
(9400,0.9774)
(9600,0.977)
(9800,0.9756)
(10000,0.9769)
	};
%\addplot [domain=100:960, samples=101,dashed]{sqrt(100/x)};
	\end{axis}

\end{tikzpicture}\nopagebreak[4]\begin{fig}\label{SEP-FIG-2}Rate of correctly classified data points for $\Delta\equiv{}10$, $\Delta=2d^{\hspace{1pt}1/4}$, and $\Delta=2d^{\hspace{1pt}0.3}$. The pictures show the means over 100 experiments for each dimension.\end{fig}
\end{center}

We leave it as Problem \ref{PROBL-6-2} to verify the above and we mention that for $\Delta=2\sqrt{d}$ in experiments the rate of correctly classified points turned out to be one for all dimensions. Finally let us look at the the distribution of mutual distances in a case where $\Delta$ is (too) small.

\begin{center}
\begin{tikzpicture}

	\begin{axis}
[
axis line style={thick, shorten >=-5pt, shorten <=-2pt},
y=0.095pt,
x=9pt,
axis y line=left,
axis x line=middle,
axis line style={->},
no markers,
tick align=outside,
major tick length=2pt,
ymin=0.0,
ymax=680,
ytick={0,200,...,600},
xmin=136,xmax=149,
xtick={136,138, ..., 148},
every tick label/.append style={font=\tiny},
xlabel=$\scriptstyle\|x-y\|$,
every axis x label/.style={
    at={(ticklabel* cs:1.07)},
    anchor=west,
},
every axis y label/.style={
    at={(ticklabel* cs:1.2)},
    anchor=north,
},
]
	% use TeX as calculator:
\addplot[ mark=none,fill=black, 
                    fill opacity=0.05] coordinates {
(137.0,0)
(137.11,0)
(137.22,0)
(137.33,0)
(137.44,0)
(137.55,0)
(137.66,0)
(137.77,0)
(137.88,2)
(137.99,1)
(138.1,1)
(138.21,2)
(138.32,3)
(138.43,4)
(138.54,7)
(138.65,8)
(138.76,17)
(138.87,21)
(138.98,26)
(139.09,28)
(139.2,43)
(139.31,56)
(139.42,79)
(139.53,91)
(139.64,110)
(139.75,130)
(139.86,151)
(139.97,179)
(140.08,209)
(140.19,237)
(140.3,271)
(140.41,290)
(140.52,352)
(140.63,375)
(140.74,371)
(140.85,398)
(140.96,436)
(141.07,446)
(141.18,440)
(141.29,492)
(141.4,511)
(141.51,501)
(141.62,468)
(141.73,481)
(141.84,463)
(141.95,522)
(142.06,495)
(142.17,469)
(142.28,498)
(142.39,498)
(142.5,495)
(142.61,506)
(142.72,482)
(142.83,489)
(142.94,436)
(143.05,501)
(143.16,513)
(143.27,454)
(143.38,450)
(143.49,495)
(143.6,462)
(143.71,472)
(143.82,414)
(143.93,419)
(144.04,368)
(144.15,338)
(144.26,305)
(144.37,295)
(144.48,245)
(144.59,225)
(144.7,221)
(144.81,209)
(144.92,185)
(145.03,120)
(145.14,94)
(145.25,113)
(145.36,91)
(145.47,52)
(145.58,65)
(145.69,59)
(145.8,34)
(145.91,27)
(146.02,27)
(146.13,10)
(146.24,13)
(146.35,14)
(146.46,6)
(146.57,6)
(146.68,5)
(146.79,1)
(146.9,1)
(147.01,0)
(147.12,0)
(147.23,0)
(147.34,0)
(147.45,0)
(147.56,0)
(147.67,0)
(147.78,0)
(147.89,0)
	};
%\addplot [domain=100:960, samples=101,dashed]{sqrt(100/x)};
	\end{axis}
\end{tikzpicture}
\begin{tikzpicture}

	\begin{axis}
[
axis line style={thick, shorten >=-5pt, shorten <=-2pt},
y=0.095pt,
x=9pt,
axis y line=left,
axis x line=middle,
axis line style={->},
no markers,
tick align=outside,
major tick length=2pt,
ymin=0.0,
ymax=680,
ytick={0,200,...,600},
xmin=136,xmax=149,
xtick={136,138, ..., 148},
every tick label/.append style={font=\tiny},
xlabel=$\scriptstyle\|x-y\|$,
every axis x label/.style={
    at={(ticklabel* cs:1.07)},
    anchor=west,
},
every axis y label/.style={
    at={(ticklabel* cs:1.2)},
    anchor=north,
},
]
	% use TeX as calculator:
\addplot[ mark=none,fill=black, 
                    fill opacity=0.05] coordinates {
(137.0,0)
(137.11,0)
(137.22,0)
(137.33,2)
(137.44,1)
(137.55,2)
(137.66,0)
(137.77,2)
(137.88,0)
(137.99,1)
(138.1,3)
(138.21,3)
(138.32,3)
(138.43,9)
(138.54,11)
(138.65,4)
(138.76,23)
(138.87,16)
(138.98,24)
(139.09,31)
(139.2,38)
(139.31,44)
(139.42,70)
(139.53,74)
(139.64,97)
(139.75,112)
(139.86,130)
(139.97,179)
(140.08,196)
(140.19,198)
(140.3,273)
(140.41,310)
(140.52,312)
(140.63,354)
(140.74,391)
(140.85,416)
(140.96,408)
(141.07,392)
(141.18,466)
(141.29,428)
(141.4,480)
(141.51,447)
(141.62,433)
(141.73,431)
(141.84,415)
(141.95,366)
(142.06,347)
(142.17,341)
(142.28,321)
(142.39,297)
(142.5,297)
(142.61,277)
(142.72,288)
(142.83,249)
(142.94,228)
(143.05,247)
(143.16,276)
(143.27,272)
(143.38,308)
(143.49,311)
(143.6,351)
(143.71,347)
(143.82,385)
(143.93,412)
(144.04,376)
(144.15,415)
(144.26,426)
(144.37,458)
(144.48,397)
(144.59,450)
(144.7,428)
(144.81,389)
(144.92,425)
(145.03,385)
(145.14,370)
(145.25,303)
(145.36,267)
(145.47,242)
(145.58,222)
(145.69,221)
(145.8,190)
(145.91,143)
(146.02,134)
(146.13,115)
(146.24,91)
(146.35,64)
(146.46,60)
(146.57,48)
(146.68,30)
(146.79,29)
(146.9,21)
(147.01,11)
(147.12,15)
(147.23,8)
(147.34,5)
(147.45,6)
(147.56,2)
(147.67,4)
(147.78,0)
(147.89,0)
	};
%\addplot [domain=100:960, samples=101,dashed]{sqrt(100/x)};
	\end{axis}
\end{tikzpicture}
\nopagebreak[4]\begin{fig}\label{SEP-FIG-3a}Mutual distances of points in a set consisting of 100 points sampled from $\mathcal{N}(0,1,\mathbb{R}^{10000})$ and 100 points sampled from $\mathcal{N}((\Delta,0,0,\dots),1,\mathbb{R}^{10000})$ with $\Delta=25,\,30$.\end{fig}
\end{center}

The pictures in Figure \ref{SEP-FIG-3a} show the distribution of mutual distances between 200 points sampled from two different Gaussians in $\mathbb{R}^{1000}$. In the left picture the distance between their means is 25 and in the right picture it is 30. In the experiment that lead to the right picture our algorithm still correctly classified 98\% of the points, whereas in the experiment underlying the left picture we only got 85\% right. Notice that the threshold identified in Corollary \ref{Ch6-COR} is $2d^{\hspace{1pt}1/4}=20$.

\section*{Problems}

\begin{probl}\label{PROBL-6-1} Replicate the results of Figure \ref{SEP-FIG-1}. More precisely, sample 100 points each of two Gaussians in $\mathbb{R}^{1000}$, one centered at zero and the other at $(1,1,\dots)$. Then compute the pairwise distances within the whole sample and plot their distribution.
\end{probl}

\begin{probl}\label{PROBL-6-2} Implement the naive separation algorithm, that picks one data point at random and then labels that half of the data set which is closest to the first point as $0$ and the rest as $1$. Test the algorithm on the data set from Problem \ref{PROBL-6-1}. When generating the data, mark the data points with $0$ and $1$ and after running the separation algorithm, let your code count how many data points got classified correctly.
\end{probl}

\begin{probl} Replicate the results of Figure \ref{SEP-FIG-2}. More precisely, run the code from Problem \ref{PROBL-6-2} for different dimensions $d$ and different distance functions $\Delta=\Delta(d)$, e.g., $\Delta\equiv c>0$, $\Delta=2\sqrt{d}$, $\Delta=d^{\hspace{1pt}0.3}$ or $\Delta=d^{\hspace{1pt}1/4}$. Plot the rate of correctly classified data points as a function of the dimension. Simulate also the case $\Delta=2d^{\hspace{1pt}0.2}$ and confirm that this leads to a low correct classification rate which decreases for large dimensions.

\begin{center}
\begin{tikzpicture}
\pgfplotsset{scaled x ticks=false}
	\begin{axis}
[
axis line style={thick, shorten >=-5pt, shorten <=-2pt},
y=100pt,
x=0.0095pt,
axis y line=left,
axis x line=middle,
axis line style={->},
no markers,
tick align=outside,
major tick length=2pt,
ymin=0.5,
ymax=1.05,
ytick={0.5,0.6,...,1},
xmin=200,xmax=10000,
xtick={500,3000,6000,9000},
every tick label/.append style={font=\tiny},
xlabel=$\scriptstyle d$,
every axis x label/.style={
    at={(ticklabel* cs:1.07)},
    anchor=west,
},
every axis y label/.style={
    at={(ticklabel* cs:1.2)},
    anchor=north,
},
]
	% use TeX as calculator:
\addplot[ mark=none] coordinates {
(200,0.6812)
(400,0.6777)
(600,0.6661)
(800,0.6648)
(1000,0.6502)
(1200,0.6532)
(1400,0.6528)
(1600,0.641)
(1800,0.6525)
(2000,0.6538)
(2200,0.6436)
(2400,0.6478)
(2600,0.649)
(2800,0.6432)
(3000,0.6416)
(3200,0.642)
(3400,0.6418)
(3600,0.6384)
(3800,0.64)
(4000,0.6442)
(4200,0.635)
(4400,0.6389)
(4600,0.6361)
(4800,0.6375)
(5000,0.6283)
(5200,0.6389)
(5400,0.6287)
(5600,0.6403)
(5800,0.6367)
(6000,0.6367)
(6200,0.6385)
(6400,0.6312)
(6600,0.6287)
(6800,0.6355)
(7000,0.6394)
(7200,0.6306)
(7400,0.6334)
(7600,0.6274)
(7800,0.6346)
(8000,0.6268)
(8200,0.6275)
(8400,0.6341)
(8600,0.6364)
(8800,0.6371)
(9000,0.6338)
(9200,0.6288)
(9400,0.6309)
(9600,0.6275)
(9800,0.6333)
(10000,0.634)
	};
%\addplot [domain=100:960, samples=101,dashed]{sqrt(100/x)};
	\end{axis}
	\end{tikzpicture}
\nopagebreak[4]\begin{fig}\label{SEP-FIG-4}Average rate of correctly classified data points for $\Delta=2d^{\hspace{1pt}0.2}$.\end{fig}
\end{center}
\end{probl}

\begin{probl} A weak spot of our method above is that the point that we start with could lie somewhere far away from both of the two annuli. Think of ways to improve on this.

\end{probl}

\begin{probl} The table below contains actual heights sampled from adults aged 40-49 years in the U.S. The numbers represent the percentage that has height less than the value in the top row of the same column and larger than the value in the row left of it\footnote{\tiny https://www2.census.gov/library/publications/2010/compendia/statab/130ed/tables/11s0205.pdf}.
\begin{center}\footnotesize
\begin{tabular}{ccccccccccccccccccc}
\toprule
 &\hspace{-2.5pt}147.32\hspace{-2.5pt}&\hspace{-2.5pt}149.86\hspace{-2.5pt}&\hspace{-2.5pt}152.40\hspace{-2.5pt}&\hspace{-2.5pt}154.94\hspace{-2.5pt}&\hspace{-2.5pt}157.48\hspace{-2.5pt}&\hspace{-2.5pt}160.02\hspace{-2.5pt}&\hspace{-2.5pt}162.56\hspace{-2.5pt}&\hspace{-2.5pt}165.10\hspace{-2.5pt}&\hspace{-2.5pt}167.64&\hspace{-2.5pt}170.18\\
\midrule\vspace{1pt}
Women &1.6&3.4&5.8&9.0&11.0&15.2&12.0&14.2&10.8&8.2\\
Men &0&0&0&0&1.9&1.9&1.8&4.2&9.6&10.9\\
\midrule\midrule\vspace{1pt}
&\hspace{-2.5pt}172.72\hspace{-2.5pt}&\hspace{-2.5pt}175.26\hspace{-2.5pt}&\hspace{-2.5pt}177.80\hspace{-2.5pt}&\hspace{-2.5pt}180.34\hspace{-2.5pt}&\hspace{-2.5pt}182.88\hspace{-2.5pt}&\hspace{-2.5pt}185.42\hspace{-2.5pt}&\hspace{-2.5pt}187.96\hspace{-2.5pt}&\hspace{-2.5pt}190.50\hspace{-2.5pt}&\hspace{-2.5pt}193.04\hspace{-2.5pt}&\hspace{-2.5pt}195.58\\
\midrule\vspace{1pt}
Women &3.5&3.1&1.6&0.1&0&0&0&0&0.5&0\\
Men   &10.1&14.0&15.2&9.5&8.3&5.1&5.2&1.3&0.4&0.5\\
\bottomrule
\end{tabular}\nopagebreak[4]
\begin{tab}\label{TAB-6-1}Sample of heights of adult U.S.~citizens.\end{tab}
\end{center}
Try with our algorithm (of which we know that it works in high dimensions) to separate the 1-dimensional data set. 
\end{probl}

\newpage

\bibliographystyle{amsplain}
\bibliography{Ref}

\end{document}